\documentclass[twoside,11pt]{article}

\usepackage{jmlr2e}
\usepackage{amsmath,amssymb,float,arydshln,color}
\usepackage{psfrag,setspace,wrapfig,subfigure}
\usepackage[latin1]{inputenc}
\usepackage{dsfont}
\usepackage[lined,ruled,commentsnumbered]{algorithm2e}
\usepackage{epsfig}
\usepackage{epstopdf}
\usepackage{graphicx}

\usepackage{hyperref}
\usepackage{amsfonts}
\usepackage{cite}
\usepackage{url}

\usepackage{hhline}
\usepackage[table]{xcolor}
\usepackage{multirow}
\usepackage{tabu}
\allowdisplaybreaks
\newtheorem{assumption}{{ Assumption}}

\def\tran{^{\mathsf{T}}}

\newcommand{\bp}{ \begin{proof}}
	\newcommand{\ep}{\end{proof} }

\newcommand{\Ex}{\mathbb{E}\hspace{0.05cm}}

\newcommand{\bm}[1]{\mbox{\boldmath $#1$}}

\newcommand{\be}{\begin{equation}}
	\newcommand{\ee}{\end{equation}}
\newcommand{\bqq}{\begin{eqnarray}}
	\newcommand{\eqq}{\end{eqnarray}}
\newcommand{\bal}{\begin{align}}
	\newcommand{\eal}{\end{align}}
\newcommand{\bqn}{\begin{eqnarray*}}
	\newcommand{\eqn}{\end{eqnarray*}}

\newcommand{\ba}{\left[ \begin{array}}
	\newcommand{\ea}{\\ \end{array} \right]}
\newcommand{\qd}{\hfill{$\blacksquare$}}
\newcommand{\define}{\;\stackrel{\Delta}{=}\;}

\def\btheta  {{\boldsymbol \theta}}

\def\bpsi       {{\boldsymbol \psi}}

\def\B{{\boldsymbol{B}}}

\def\H{{\boldsymbol{H}}}
\def\I{{\boldsymbol{I}}}
\def\J{{\boldsymbol{J}}}
\def\K{{\boldsymbol{K}}}

\def\M{{\boldsymbol{M}}}

\def\R{{\boldsymbol{R}}}

\def\d{{\boldsymbol{d}}}

\def\h{{\boldsymbol{h}}}

\def\s{{\boldsymbol{s}}}

\def\u{{\boldsymbol{u}}}
\def\v{{\boldsymbol{v}}}
\def\w{{\boldsymbol{w}}}
\def\x{{\boldsymbol{x}}}
\def\y{{\boldsymbol{y}}}

\newcommand{\rev}[1]{{\color{black}#1}}

\newcommand{\tw}{\widetilde{\boldsymbol{w}}}
\newcommand{\hw}{\widehat{\boldsymbol{w}}}

\newcommand{\cw}{\check{\boldsymbol{w}}}

\newcommand{\tx}{\widetilde{\boldsymbol{x}}}

\newcommand{\tpsi}{\widetilde{\boldsymbol{\psi}}}
\newcommand{\bgm}{{\boldsymbol{\gamma}}}
\newcommand{\grad}{{\nabla}} 
\def\bE{\mathbb{E}}
\def\filt{\boldsymbol{\mathcal{F}}}

\newcommand{\eq}[1]{\begin{align}#1\end{align}}
\newcommand{\beqn}{\begin{eqnarray}}
	\newcommand{\eeqn}{\end{eqnarray}}
\newcommand{\defeq}{{\stackrel{\triangle}{=}}}
\newcommand{\nnb}{\nonumber \\}
\newcommand{\dom}{{\mathrm{dom}}} % domain

\newcommand{\mua}{{\mu_{\max}}}

\newcommand{\musa}{{\mu^2_{\max}}}

\def\real{{\mathbb{R}}}

\makeatletter
\def\hlinewd#1{%
	\noalign{\ifnum0=`}\fi\hrule \@height #1 \futurelet
	\reserved@a\@xhline}
\makeatother

\jmlrheading{17}{2016}{1-66}{03/16; Revised 08/16}{10/16}{Kun Yuan, Bicheng Ying, and Ali H. Sayed}

\ShortHeadings{On the Influence of Momentum Acceleration on Online Learning}{Yuan, Ying, and Sayed}
\firstpageno{1}

\editor{Leon Bottou}

\begin{document}
	
	\title{On the Influence of Momentum Acceleration on Online Learning}
	
	\author{\name Kun Yuan \email kunyuan@ucla.edu\\
		\name Bicheng Ying \email ybc@ucla.edu \\
		\name Ali H.\ Sayed \email sayed@ucla.edu \\
		\addr Department of Electrical Engineering\\
		University of California\\
		Los Angeles, CA 90095, USA\\
		}
	
	\maketitle

	\begin{abstract}%   <- trailing '%' for backward compatibility of .sty file
		The article examines in some detail the convergence rate and mean-square-error performance of momentum stochastic gradient methods {\color{black}in the constant step-size and slow adaptation regime}. The results establish that momentum methods are equivalent to
		the standard stochastic gradient method with a re-scaled (larger) step-size value.
		The size of the re-scaling is determined by the value of the momentum parameter. The equivalence result is established for all time instants and not only in steady-state. The analysis is carried out for general {\color{black}strongly convex and smooth} risk functions, and is not limited to quadratic risks. One notable conclusion is that the well-known benefits of momentum constructions for deterministic optimization problems do not necessarily carry over to {\color{black}the adaptive online setting when small constant step-sizes are used to enable continuous adaptation and learning in the presence of persistent gradient noise.} {\color{black}From simulations, the equivalence between momentum and standard stochastic gradient methods is also observed for non-differentiable and non-convex problems.}
%		{\color{red}The analysis also suggests a method to retain some of the advantages of the momentum construction by employing a decaying momentum parameter, as opposed to a decaying step-size. In this way, the enhanced convergence rate during the initial stages of adaptation is preserved without the often-observed degradation in MSD performance.}
	\end{abstract}
	
%	{\color{red}[Professor: Previously we suggested a decaying momentum method to keep both fast convergence and satisfactory MSD performance. However, according to our analysis some reviewers doubt that this decaying moemtnum method should be equivalent to the method employing decaying step-size until to a lower bound, which is really true. Please see Section \ref{diminishing momentum} and our simulation figures \ref{fig:lms_conv} and \ref{fig:lms_gap}. Therfore we changed our argument. Instead of arguing that we are proposing a useful method, we now argue that our analysis also implies that this decaying moemtnum method is not useful. Shall we add this point to the abstract?]}
	
	
	\section{Introduction}
	Stochastic optimization focuses on the problem of optimizing the expectation of a loss function, written as
	\eq{
		\label{general prob}
		\min_{w\in \mathbb{R}^M} \; J(w)\define\bE_{\btheta}[ Q(w;\btheta)],
	}
	where $\btheta$ is a random variable whose distribution is generally unknown and $J(w)$ is a convex function (usually strongly-convex due to regularization). 
	% Problems of this kind are common in many contexts, including in several adaptation and machine learning
	%formulations \citep{widrow1985adaptive, haykin2008adaptive, Sayed2011adaptive, theodoridis2015machine}. 
	If the probability distribution of the data, $\btheta$, is known beforehand, then one can evaluate $J(w)$ and seek its minimizer by means of a variety of gradient-descent or Newton-type methods \citep{polyak1987intro,bertsekas1999nonlinear,Nesterov2003}. We refer to these types of problems, where $J(w)$ is known, as {\em deterministic} optimization problems. On the other hand, when the probability distribution of the data is unknown, then the risk function $J(w)$ is unknown as well; only instances of the loss function, $Q(w;\btheta)$, may be available at various observations $\btheta_i$, where $i$ refers to the sample index. We refer to these types of problems, where $J(w)$ is unknown but defined implicity as the expectation of some known loss form, as {\em stochastic} optimization problems. This article deals with this second type of problems, which are prevalent in online adaptation and learning contexts  \citep{widrow1985adaptive, haykin2008adaptive, Sayed2011adaptive, theodoridis2015machine}. 
	
	When $J(w)$ is differentiable, one of the most popular techniques to seek minimizers for \eqref{general prob} is to employ the {\em stochastic} gradient method. This algorithm is based on employing instantaneous approximations for the true (unavailable) gradient vectors, $\grad_w J(w)$, by using the gradients of the loss function, $\grad_w Q(w;\btheta_i)$, evaluated at successive samples of the streaming data $\btheta_i$ over the iteration index $i$, say, as:
	\be\label{stochastic gradient method-2}
	\w_i\;=\;\w_{i-1}\;-\;\mu \grad_w\, Q(\w_{i-1};\btheta_i),\;\;\;i\geq 0.
	\ee
	where $\mu>0$ is a step-size parameter. Note that we are denoting the successive iterates by $\w_i$ and using the boldface notation to refer to the fact that they are random quantities in view of the randomness in the measurements $\{\btheta_i\}$.    Due to their simplicity, robustness to noise and uncertainty, and scalability to big data, such stochastic gradient methods have become popular in large-scale optimization, machine learning, and data mining applications \citep{zhang2004solving, bottou2010large,gemulla2011large, Sutskever2013, kahou2013combining, cevher2014convex, szegedy2014going, zarkeba2015accelerated}.
	
	\subsection{Convergence Rate}
	\label{sec:conv-rate}
	Stochastic-gradient algorithms can be implemented with decaying step-sizes, such as $\mu(i)=\tau/i$ for some constant $\tau$, or with constant step-sizes, $\mu>0$. The former generally ensure asymptotic convergence to the true minimizer of \eqref{general prob}, denoted by $w^o$,  at a convergence rate that is on the order of $O(1/i)$ {\color{black} for strongly-convex risk functions.} This guarantee, however, comes at the expense of turning off adaptation and learning as time progresses since the step-size value approaches zero in the limit, as $i\rightarrow\infty$. As a result, the algorithm loses the ability to track concept drifts. In comparison, constant step-sizes keep adaptation and learning alive and infuse a desirable tracking mechanism into the operation of the
	algorithm: even if the minimizers drift with time, the algorithm will generally be able to adjust and track their locations. Moreover, convergence can now occur at the considerably faster exponential rate, $O(\alpha^i$), for some $\alpha\in (0,1)$. These favorable properties come at the expense of a small deterioration in the limiting accuracy of the iterates since almost-sure convergence is not guaranteed any longer. Instead, the algorithm converges in the mean-square-error sense towards a small neighborhood around the true minimizer, $w^o$, whose radius is on the order of $O(\mu)$. This is still a desirable conclusion because the value of $\mu$ is controlled by the designer and can be chosen sufficiently small. 
		
%		In real applications, it is customary to choose sufficiently small step-size to ensure satisfactory steady state MSE performance and simplify the convergence performance analysis.
	
	A well-known tradeoff therefore develops between convergence rate and mean-square-error (MSE) performance. The asymptotic MSE performance level approaches $O(\mu)$ while the convergence rate is given by $\alpha=1-O(\mu)$ \citep{polyak1987intro, sayed2014adaptation}. 
	It is nowadays well-recognized that the small $O(\mu)$ degradation in performance is acceptable in most large-scale learning and adaptation problems \citep{bousquet2008tradeoffs, bottou2010large, sayed2014adaptive}. This is because, in general, there are always modeling errors in formulating optimization problems of the form (\ref{general prob}); the cost function may not reflect perfectly the scenario and data under study. As such, insisting on attaining asymptotic convergence to the true minimizer may not be necessarily the best course of action or may not be worth the effort. It is often more advantageous to tolerate a small steady-state error that is negligible in most cases, but is nevertheless attained at a faster exponential rate of convergence than the slower rate of $O(1/i)$. Furthermore, the data models in many applications are more complex than assumed, with possibly local minima. In these cases, constant step-size implementations can help reduce the risk of being trapped at local solutions. 
	
	For these various reasons, and since our emphasis is on algorithms that are able to learn continuously, we shall focus on \rev{small {\em constant} step-size} implementations. \rev{In these cases, gradient noise is always present, as opposed to decaying step-size implementations where the
gradient noise terms get annihilated with time. The analysis in the paper will establish analytically, and illustrate by simulations, 
 that, for sufficiently small step-sizes, any benefit from a momentum stochastic-construction can be attained by adjusting the 
step-size parameter for the original stochastic-gradient implementation. We emphasize here the qualification ``small'' for the step-size. The
reason we focus on small step-sizes (which correspond to the slow adaptation regime) is because, in the stochastic context, mean-square-error stability and convergence require small step-sizes.}
	
%	\rev{In addition, to distinguish from other machine learning settings that we will discuss below (see Section \ref{sec:different-setting}), we call problem \eqref{general prob} in which the minimizer $w^o$ may (but not necessarily) drift with time {\em the adaptive online learning problem}, and the existence of gradient noise and the exploitation of sufficiently small constant step-size in algorithms {\em the adaptive online learning setting}. 
	
%	\rev{To avoid confusion or misunderstanding, the terminology ``adaptation" in this paper particularly indicates the ability of the algorithm to track the drift of the minimizer $w^o$. It does not mean that the step-size $\mu$ will be automatically adapted or annealed to achieve better convergence rate or training accuracy.
%	
%	}
	
	\subsection{Acceleration Methods}
	\label{sec:intro-acc-method}
	In the {\em deterministic} optimization case, when the true gradient vectors of the {\color{black}smooth} risk function $J(w)$ are available, the iterative algorithm for seeking the minimizer of $J(w)$ becomes the following gradient-descent recursion
	\be\label{gradient descent}
	w_i=w_{i-1}-\mu \nabla_w\,J(w_{i-1}),\;\;i\geq 0,
	\ee
	There have been many ingenious methods proposed in the  literature to enhance the convergence of these methods for both cases of convex and strongly-convex risks, $J(w)$. Two of the most notable and successful techniques are  the  heavy-ball method \citep{polyak1964some, polyak1987intro, qian1999momentum} and Nesterov's acceleration method \citep{Nesterov1983, Nesterov2003, Nesterov2005} ({\color{black}the recursions for these algorithms are described in Section 3.1)}. The two methods are different but they both rely on the concept of adding a momentum term to the recursion. {When the risk function $J(w)$ is $\nu$-strongly convex and has $\delta$-Lipschitz continuous gradients, both methods succeed in accelerating the
		gradient descent algorithm to attain a faster exponential convergence rate \citep{polyak1987intro} \citep{Nesterov2003}, and this rate {\color{black} is proven to be optimal for problems with smooth $J(w)$} and cannot be attained by standard gradient descent methods.}
	{Specifically, it is shown in \citep{polyak1987intro} \citep{Nesterov2003} that for heavy-ball and Nesterov's acceleration methods, the convergence of the iterates $w_i$ towards $w^o$ occurs at the rate:
		\eq{\label{acc-eq}
		\|w_i-w^o\|^2\le \left( \frac{\sqrt{\delta}-\sqrt{\nu}}{\sqrt{\delta}+\sqrt{\nu}} \right)^2 \|w_{i-1}-w^o\|^2,
		}
		{In comparison, in Theorem 2.1.15 of \citep{Nesterov2005} and Theorem 4 in Section 1.4 of \citep{polyak1987intro}, the fastest rate for gradient descent method is shown to be
		\eq{\label{gra-rate-eq}
		\|w_i-w^o\|^2\le \left( \frac{\delta-\nu}{\delta + \nu} \right)^{2} \|w_{i-1}-w^o\|^2.
	}
		It can be verified that
		\eq{\label{gra-rate-eq-2}
		 \frac{\sqrt{\delta}-\sqrt{\nu}}{\sqrt{\delta}+\sqrt{\nu}}  <
		\frac{{\delta}-{\nu}}{{\delta}+{\nu}} 
		}
		when $\delta > \nu$. This inequality confirms that the momentum algorithm can achieve a faster rate in deterministic optimization and, moreover, this faster rate cannot be attained by standard gradient descent.
	}}
	
	Motivated by these useful acceleration properties in the {\em deterministic} context, momentum terms have been subsequently introduced into {\em stochastic} optimization algorithms as well \citep{polyak1987intro, proakis1974channel, sharma1998analysis, shynk1988lms, Shynk1990, Tugay1989, bellanger2001adaptive, Wiegerinck1994, hu2009accelerated,xiao2010dual,lan2012optimal,ghadimi2012optimal,zhong2014accelerated} and applied, for example, to problems involving the tracking of chirped sinusoidal signals \citep{ting2000tracking} or
	deep learning \citep{Sutskever2013, kahou2013combining, szegedy2014going, zarkeba2015accelerated}. However, the analysis in this paper will show that
	the advantages of the momentum technique for deterministic optimization do not necessarily carry over to the \rev{{\em adaptive} online} setting
	due to the presence of stochastic gradient noise
	(which is the difference between the actual gradient vector and its approximation). Specifically, {\color{black} for sufficiently small step-sizes and for a momentum paramter not too close to one},  
we will show that any advantage brought forth by the momentum term can be achieved by staying with the original stochastic-gradient algorithm and adjusting its step-size to a larger value. For instance, for optimization problem \eqref{general prob}, we will show that if the step-sizes, $\mu_m$ for the momentum (heavy-ball or Nesterov) methods and $\mu$ for the standard stochastic gradient algorithms, are sufficiently small and satisfy the relation
	\eq{
		\label{stepsize condition}
		\mu=\frac{\mu_m}{1-\beta}
	}
	where $\beta$, {\color{black} a positive constant that is not too close to $1$,} is the momentum parameter, then 
	it will hold that
	%the mean-square distance
	%between the iterates generated by the two classes of algorithms is on the order of $O(\mu^{3/2})$ for \textit{each iteration}, i.e.,
	\be \bE\|\w_{m,i} - \w_{i}\|^2=O(\mu^{3/2}),\ i=0,1,2,\ldots\label{eq-general}
	\ee
	where $\w_{m,i}$ and $\w_{i}$ denote the iterates generated at time $i$ by the momentum and standard implementations, respectively.  In the special case when $J(w)$ is quadratic in $w$, as happens in mean-square-error design problems, we can tighten \eqref{eq-general} to
	\be \bE\|\w_{m,i} - \w_{i}\|^2=O(\mu^{2}),\ i=0,1,2,\ldots 
	\ee
	What is important to note is that, we will show that these results hold {\em for every} $i$, and not only asymptotically. Therefore, when $\mu$ is sufficiently small, property \eqref{eq-general}  establishes that the stochastic gradient method and the momentum versions are fundamentally equivalent since their iterates evolve close to each other at all times. 
	{We establish this equivalence result under the situation where the risk function is convex and differentiable. However, as our numerical simulations over \rev{a multi-layer fully connected neural network and a second convolutional neural network (see Section \ref{sec-cnn}) show}, the equivalence between standard and momentum stochastic gradient methods are also observed in non-convex and non-differentiable scenarios.}
	
	\subsection{Related Works in the Literature}
	There are useful results in the literature that deal with special instances of the general framework developed in this work. These earlier results focus mainly on the mean-square-error case when $J(w)$ is quadratic in $w$, in which case the stochastic gradient algorithm reduces to the famed least-mean-squares (LMS) algorithm. We will not be limiting our analysis to this case so that our results will be applicable to a broader class of learning problems beyond mean-square-error estimation (e.g., logistic regression would be covered by our results as well). As the analysis and derivations will reveal, the treatment of the general $J(w)$ case is demanding because the Hessian matrix of $J(w)$ is now $w-$dependent, whereas it is a constant matrix in the quadratic case.

	Some of the earlier investigations in the literature led to the following observations. It was noted in \citep{polyak1987intro} that, for quadratic costs, stochastic gradient implementations with a momentum term do not necessarily perform well. This work remarks that although the  heavy-ball method can lead to faster convergence in the early stages of learning, it nevertheless converges to a region with worse mean-square-error in comparison to standard stochastic-gradient (or LMS) iteration. A similar phenomenon is also observed in \citep{proakis1974channel, sharma1998analysis}. However, in the works \citep{proakis1974channel, polyak1987intro, sharma1998analysis}, no claim is made or established about the equivalence between momentum and standard methods.
	
	%	\vspace{-1mm}
	Heavy-ball LMS was further studied in the useful works \citep{Shynk1990} and \citep{Tugay1989}. The reference \citep{Shynk1990} claimed that no significant gain 
	is achieved in convergence speed if both the heavy-ball and standard
	LMS algorithms approach the same \textit{steady-state} MSE performance.
	%is achieved by adding the momentum term. 
	Reference \citep{Tugay1989} observed that when the step-sizes satisfy relation \eqref{stepsize condition}, then heavy-ball LMS is ``equivalent'' to standard LMS. However, they assumed Gaussian measurement noise in their data model, and the notion of ``equivalence'' in this work is only referring to the fact that the algorithms have similar starting convergence rates and similar steady-state MSE levels. There was no analysis in \citep{Tugay1989} of the behavior of the algorithms during all stages of learning  -- see also \citep{bellanger2001adaptive}. Another useful work is \citep{Wiegerinck1994}, which considered the heavy-ball stochastic gradient method for general risk, $J(w)$. By assuming a sufficiently small step-size, and by  transforming the error difference recursion into a differential equation, the work concluded that heavy-ball can be equivalent to the standard stochastic gradient method asymptotically (i.e., for $i$ large enough). No results were provided for the earlier stages of learning.

	All of these previous works were limited to examining the heavy-ball momentum technique; none of them considered other forms of acceleration such as Nesterov's technique  although this latter technique is nowadays  widely applied to stochastic gradient learning, including deep learning \citep{Sutskever2013, kahou2013combining, szegedy2014going, zarkeba2015accelerated}. 
	The performance of Nesterov's acceleration with \textit{deterministic} and \textit{bounded}  gradient error was examined in \citep{d2008smooth, devolder2014first, lessard2016analysis}. The source of the inaccuracy in the gradient vector in these works is either because the gradient was assessed by solving an auxiliary ``simpler" optimization problem or because of numerical approximations. Compared to the standard gradient descent implementation, the works by \citep{d2008smooth, lessard2016analysis} claimed that Nesterov's acceleration is not robust to the errors in gradient. The work by \citep{devolder2014first} also observed that the superiority of Nesterov's acceleration is no longer absolute when inexact gradients are used, and they further proved that the performance of Nesterov's acceleration may be even worse than gradient descent due to error accumulation. These works assumed bounded errors in the gradient vectors and focused on the context of deterministic optimization. None of the works examined the stochastic setting where the gradient error is random in nature and where the assumption of bounded errors are generally unsuitable. We may add that there have also been analyses of Nesterov's acceleration for \textit{stochastic} optimization problems albeit for {\em decaying} step-sizes in more recent literature \citep{hu2009accelerated,xiao2010dual,lan2012optimal,ghadimi2012optimal,zhong2014accelerated}. These works proved that Nesterov's acceleration can improve the convergence rate of stochastic gradient descent {\color{black}at the initial stages when deterministic risk components dominate; while at the asymptotic stages when the stochastic gradient noise dominates, the momentum correction cannot accelerate convergence any more.} {\color{black}Another useful study is \citep{flammarion2015averaging}, in which the authors showed that momentum and averaging methods for stochastic optimization are equivalent to the same second-order difference equations but with different step-sizes. However, \citep{flammarion2015averaging} does not study the equivalence between standard and momentum stochastic gradient methods, and they focus on quadratic problems and also employ decaying step-sizes.}
	
	\rev{Finally, we note that there are other forms of stochastic gradient algorithms for empirical risk minimization problems where momentum acceleration has been shown to be useful. Among them, we list recent algorithms like SAG \citep{roux2012stochastic}, SVRG \citep{johnson2013accelerating} and SAGA \citep{defazio2014saga}. In these algorithms, the variance of the stochastic gradient noise diminishes to zero and the deterministic component of the risk becomes dominant in the asymptotic regime. In these situations, momentum acceleration helps improve the convergence rate, as noted by  \citep{nitanda2014stochastic} and \citep{allen2016katyusha}. Another family of algorithms to solve empirical risk minimization problems are stochastic dual coordinate ascent (SDCA) algorithms. It is proved in \citep{shalev2015sdca,johnson2013accelerating} that SDCA can be viewed as a variance-reduced stochastic algorithm, and hence momentum acceleration can also improve its convergence for the same reason noted by  \citep{shalev2014accelerated}. 
		
	In this paper, we are studying online training algorithms where data can stream in continuously as opposed to running multiple passes over a finite amount of data. In this case, the analysis will help clarify the limitations of momentum acceleration in the slow adaptation regime. We are particularly interested in the constant step-size case, which enables continuous adaptation and learning and is regularly used, e.g., in deep learning implementations. \rev{There is a non-trivial difference between the decaying and constant step-size situations. This is because gradient noise is always present in the constant step-size case, while it is annihilated in the decaying step-size case. The presence of the gradient noise interferes with the dynamics of the algorithms in a non-trivial way, which is what our analysis discovers.} There are limited analyses for the constant step-sizes scenario. }

	\subsection{Outline of Paper}
	The outline of the paper is as follows. In Section \ref{sec: stochastic gradient method}, we introduce some basic assumptions and review the stochastic gradient method
	and its convergence properties. In Section
	\ref{Sec: GFAT} we embed the heavy-ball and Nesterov's acceleration methods into a unified momentum algorithm, and subsequently establish the mean-square stability and fourth-order stability of the error moments. Next, we analyze the equivalence between momentum and standard LMS algorithms in Section \ref{Sec: LMS} and then extend the results to general risk functions in Section \ref{sec: equi-ge}. In 
	Section \ref{sec: extension} we extend the equivalence results into a more general setting with diagonal step-size matrices.
	We illustrate our results in Section \ref{Sec: exp}, and in Section \ref{sec:stable} we comment on the stability ranges of standard and momentum stochastic gradient methods.
	
	\section{Stochastic Gradient Algorithms}
	\label{sec: stochastic gradient method}
	In this section we review the stochastic gradient method and its convergence properties. We denote the minimizer for problem \eqref{general prob} by $w^o$, i.e.,
	\eq{ w^o\define \arg\min_w\ J(w).}
	We introduce the following assumption on $J(w)$, which essentially amounts to assuming that $J(w)$ is strongly-convex with Lipschitz gradient. These conditions are satisfied by many problems of interest, especially when regularization is employed (e.g., mean-square-error risks, logistic risks, etc.). Under the strong-convexity condition, the minimizer $w^o$ is unique.
	
	\begin{assumption}[{\bf Conditions on risk function}]
		\label{ass: cost function}
		The cost function $J(w)$ is twice differentiable and its Hessian matrix satisfies
		\eq{
			\label{cost function assumption}
			0  < \nu I_M \leq \grad^2 J(w) \leq \delta I_M,
		}
		for some positive parameters $\nu\leq \delta$. Condition \eqref{cost function assumption}
		is equivalent to requiring $J(w)$ to be $\nu$-strongly convex and for its gradient vector to be
		$\delta$-Lipschitz, respectively \citep{boyd2004convex, sayed2014adaptation}.
		
		\qd
		
	\end{assumption}
	
	\noindent The stochastic-gradient algorithm for seeking $w^o$ takes the form
	(\ref{stochastic gradient method-2}), with initial condition $\w_{-1}$. The difference between the true gradient vector and its approximation is designated {\em gradient noise} and is denoted by:
	\eq{
		\label{noise-2}
		\s_i(\w_{i-1})\ \defeq\ {\grad_w  Q}(\w_{i-1};\btheta_i) - \nabla_w \bE[Q(\w_{i-1};\btheta_i)].
	}
	In order to examine the convergence of the standard and momentum stochastic gradient methods,
	it is necessary to introduce some assumptions on the stochastic gradient noise. Assumptions \eqref{noise-1-order} and \eqref{noise-2-order} below are satisfied by important cases of interest, as shown in \citep{sayed2014adaptation} and \citep{sayed2014adaptive}, such as logistic regression and mean-square-error risks. Let the symbol $\filt_{i-1}$ represent the filtration generated by the random process $\w_j$ for $j\le i-1$ (basically, the collection of past history until time $i-1$):
	$$\filt_{i-1}\ \defeq\ \mbox{filtration}\{\w_{-1},\w_0, \w_1,\ldots, \w_{i-1}\}.$$
	
	\begin{assumption}[{\bf Conditions on gradient noise}]
		\label{ass: noise}
		It is assumed that the first and second-order conditional moments of the gradient noise process satisfy
		the following conditions for any $\w\in \filt_{i-1}$:
		\beqn
		\bE[\s_i(\w)|\filt_{i-1}]&=&0 \label{noise-1-order} \\
		\bE[\|\s_i(\w)\|^2|\filt_{i-1}]&\le& {\gamma}^2 \|w^o-\w\|^2 + {\sigma}_s^2 \label{noise-2-order}
		\eeqn
		almost surely, for some nonnegative constants $\gamma^2$ and $\sigma_s^2$.
		
		\qd
		
	\end{assumption}
	
	\noindent Condition (\ref{noise-1-order}) essentially requires the gradient noise process to have zero mean, which amounts to requiring the approximate gradient to correspond to an unbiased construction for the true gradient. This is a reasonable requirement. Condition (\ref{noise-2-order}) requires the size of the gradient noise (i.e., its mean-square value) to diminish as the iterate $\w$ gets closer to the solution $w^o$. This is again a reasonable requirement since it amounts to expecting the gradient noise to get reduced as the algorithm approaches the minimizer. Under Assumptions \ref{ass: cost function} and \ref{ass: noise}, the following conclusion is proven in Lemma 3.1 of \citep{sayed2014adaptation}.
	
	\begin{lemma}[{\bf Second-order stability}]
		\label{lm: stochastic gradient method-tw-bound}
		Let Assumptions \ref{ass: cost function} and \ref{ass: noise} hold, and consider the stochastic gradient recursion \eqref{stochastic gradient method-2}. Introduce the error vector $\tw_i=w^o-\w_i$. Then,
		for any step-sizes $\mu$ satisfying 
		\eq{\label{yuanlds}
		\mu < \frac{2\nu}{\delta^2 + \gamma^2},
		}
		it holds for each iteration $i=0,1,2,\ldots$ that
		\eq{
			\label{tw-evolve-2}
			\bE\|\tw_i\|^2 \le (1-\mu \nu)\bE\|\tw_{i-1}\|^2 + \mu^2\sigma_s^2,
		}
		and, furthermore,
		\eq{
			\limsup_{i\rightarrow \infty}\bE\|\tw_i\|^2 \le \frac{\sigma_s^2 \mu}{\nu} = O(\mu). \label{2rd-ss}
		}
		
		\qd
	\end{lemma}
	
%	
%	\noindent The above lemma characterizes the stability of the second-order moment of the error vector,	$\bE\|\widetilde{\w}_i\|^2$.	
%	 One useful conclusion that follows from (\ref{tw-evolve-2}) is that, again for sufficiently small step-sizes,
%	\eq{
%		\label{tw < D}
%		\bE\|\tw_i\|^2 \le D, \;\;\;\forall i\geq 0
%	}
%	for some constant $D$ --- see equation \eqref{tw-evolve-2-5}. We will exploit this fact in the sequel.
	
	We can also examine the the stability of the fourth-order error moment, $\bE\|\tw_i\|^4$, which will be used later in Section \ref{sec: equi-ge} to establish the equivalence between the standard and momentum stochastic implementations. For this case, we tighten the assumption on the gradient noise by replacing the bound in (\ref{noise-2-order}) on its second-error moment by a similar bound involving its fourth-order moment. Again, this assumption is satisfied by problems of interest, such as mean-square-error and logistic risks \citep{sayed2014adaptation, sayed2014adaptive}.
	
	\begin{assumption}[{\bf Conditions on gradient noise}]
		\label{ass: noise-4}
		It is assumed that the first and fourth-order conditional moments of the gradient noise process satisfy
		the following conditions for any $\w\in \filt_{i-1}$:
		\beqn
		\bE[\s_i(\w)|\filt_{i-1}]&=&0 \label{noise-1-order-4} \\
		\bE[\|\s_i(\w)\|^4|\filt_{i-1}]&\le& {\gamma_4^4} \|w^o-\w\|^4 + {\sigma}_{s,4}^4 \label{noise-4-order}
		\eeqn
		almost surely, for some nonnegative constants $\gamma_4^4$ and $\sigma_{s,4}^4$.	
		\qd
		
	\end{assumption}
	
	\noindent It is straightforward to check that if Assumption \ref{ass: noise-4} holds, then Assumption \ref{ass: noise} will also hold. The following conclusion is a modified version of Lemma 3.2 of \citep{sayed2014adaptation}.
	
	\begin{lemma}[{\bf Fourth-order stability}]
		\label{lm: x4 bound}
		Let the conditions under Assumptions \ref{ass: cost function} and  \ref{ass: noise-4} hold, and
		consider the stochastic gradient iteration \eqref{stochastic gradient method-2}. For sufficiently small step-size $\mu$, it holds that
		{\color{black}
		\eq{
			\label{lm: x4-evolve-3}
%			\bE\|\tw_i\|^4 \le (1-\mu \nu)^{i+1}\bE\|\tw_{-1}\|^4 + c \mu^2,
	\bE\|\tw_i\|^4 \le \rho^{i+1}\bE\|\tw_{-1}\|^4  + A \sigma_s^2 (i+1)\rho^{i+1}\mu^2 + \frac{B \sigma_s^4\mu^2}{\nu^2}
		}
		where $\rho \define 1 - \mu\nu$, and $A$ and $B$ are some constants. Furthermore,
		\eq{
			\label{lm: x4-evolve-4}
%			\limsup_{i\rightarrow \infty}\bE\|\tw_i\|^4 \le O(\mu^2).
			\limsup_{i\rightarrow \infty}\bE\|\tw_i\|^4 \le \frac{B \sigma_s^4\mu^2}{\nu^2} = O \left( \mu^2 \right) 
		}}
	\end{lemma}
	\begin{proof} See Appendix \ref{app-lm-4-sta}.
		%\begin{proof} \noindent It follows from Eq.~(3.76) in \citep{sayed2014adaptation} that $\bE\|\tw_i\|^4$ satisfies:
		%\bqq
		%\bE\|\tw_i\|^4 \hspace{-0.1cm}&\le&\hspace{-0.1cm} (1-\mu \nu) \bE\|\tw_{i-1}\|^4 + a_1 \mu^2\bE \|\tw_{i-1}\|^2 + a_2 \mu^4,\nn\\
		%\hspace{-0.1cm}&\stackrel{\footnotesize (\ref{tw < D})}{\leq}& \hspace{-0.1cm}(1-\mu \nu) \bE\|\tw_{i-1}\|^4 + a_1 \mu^2D + a_2 \mu^4,\nn\\
		%\hspace{-0.1cm}&\leq&\hspace{-0.1cm} (1-\mu \nu) \bE\|\tw_{i-1}\|^4 + O(\mu^2),
		%\label{x4-evolve}
		%\eqq
		%for some constants $a_1$ and $a_2$. The last relation establishes (\ref{lm: x4-evolve-3}). On the other hand, computing the limit superior of both sides of the first inequality given above, and using (\ref{2rd-ss}), we conclude that (\ref{lm: x4-evolve-4}) holds.	
	\end{proof}
	
	\section{Momentum Acceleration}
	\label{Sec: GFAT}
	In this section, we present a generalized momentum stochastic gradient method, which captures both the heavy-ball and Nesterov's acceleration methods as special cases. Subsequently, we derive results for its convergence property.
	
	\subsection{Momentum Stochastic Gradient Method}
	Consider the following general form of a stochastic-gradient
	implementation, with two momentum parameters $\beta_1,\beta_2\in [0,1)$:
	\bqq
	\hspace{-0.4cm}\bpsi_{i-1}\hspace{-0.3cm} &=&\hspace{-0.3cm}\w_{i-1} + \beta_1 (\w_{i-1} - \w_{i-2}), \label{sgfat-1}\\
	\hspace{-0.4cm}\w_i \hspace{-0.3cm} &=&\hspace{-0.3cm} \bpsi_{i-1} - \mu_m {\grad_w Q}(\bpsi_{i-1};\btheta_i) + \beta_2(\bpsi_{i-1}\hspace{-0.1cm}-\hspace{-0.1cm}\bpsi_{i-2}), \label{sgfat-2}
	\eqq
	with initial conditions 
	\bqq
	\hspace{-0.2cm}&&\w_{-2} = \bpsi_{-2}=\mbox{initial states}, \label{sgfat-3}\\
	\hspace{-0.2cm}&&\w_{-1}=\w_{-2}-\mu_m {\grad_w Q}(\w_{-2};\btheta_{-1}),  \label{sgfat-4}
	\eqq
	where $\mu_m$ is some constant step-size. We refer to this formulation as
	the momentum stochastic gradient method. \footnote{Traditionally, the terminology of a ``momentum method'' has been used more frequently for the heavy-ball method, which corresponds to the special case $\beta_1=0$ and $\beta_2=\beta$. Given the unified description (\ref{sgfat-1})--(\ref{sgfat-2}), we will use this same terminology to refer to both the heavy-ball and Nesterov's acceleration methods.}
	
	When $\beta_1=0$ and $\beta_2 = \beta$ we recover
	the heavy-ball algorithm \citep{polyak1964some,polyak1987intro}, and when $\beta_2=0$ and $\beta_1 = \beta$, we recover Nesterov's algorithm \citep{Nesterov2003}. We note that Nesterov's method has several useful variations that fit different scenarios, such as situations involving smooth but not strongly-convex risks \citep{Nesterov1983, Nesterov2003} or non-smooth risks \citep{Nesterov2005, beck2009fast}. However, for the case when $J(w)$ is strongly convex and has Lipschitz continuous gradients, the Nesterov construction reduces to what is presented above, with a constant momentum parameter. {This type of
	construction has also been studied in \citep{lessard2016analysis, Dieuleveut2016harder} and 
	applied in deep learning implementations \citep{Sutskever2013, kahou2013combining, szegedy2014going, zarkeba2015accelerated}}.
	% {[\color{black}Add more engineering reference]}.
	
	In order to capture both the heavy-ball and Nesterov's acceleration methods in a unified treatment, we will assume that
	\eq{
		\beta_1+\beta_2=\beta,\quad \beta_1 \beta_2 =0,\label{ass: gfat}
	}
	for some fixed constant $\beta \in [0,1)$. {\color{black}Next we introduce a condition on the momentum parameter.
	\begin{assumption}\label{ass:beta-cond}
		The momentum parameter $\beta$ is a constant that is not too close to $1$, i.e., there exists a small fixed constant $\epsilon>0$ such that $\beta \le 1-\epsilon$. 
	\qd
	\end{assumption}	
	Assumption \ref{ass:beta-cond} is quite common in studies on adaptive signal processing and neural networks --- see, e.g., \citep{Tugay1989,Shynk1990,bellanger2001adaptive,Wiegerinck1994,attoh1999analysis}. Also, in recent deep learning applications it is common to set $\beta = 0.9$, which satisfies Assumption \ref{ass:beta-cond} \citep{krizhevsky2012imagenet, szegedy2014going, zhang2015text}. Under \eqref{ass: gfat}, the work \citep{flammarion2015averaging} also considers recursions related to \eqref{sgfat-1}--\eqref{sgfat-2} for the special case of quadratic risks.}
	
	\subsection{Mean-Square Error Stability}
	In preparation for studying the performance of the momentum stochastic gradient method, we first show in the next result how recursions \eqref{sgfat-1}-\eqref{sgfat-2} can be transformed into a first-order
	recursion by defining extended state vectors. We introduce the transformation matrices:
	\begin{align}
	\label{V and V_inv}
	V =
	\left[
	\begin{array}{cc}
	I_M & -\beta I_M\\
	I_M & -I_M\\
	\end{array}
	\right],\
	V ^{-1}= \frac{1}{1-\beta}
	\left[
	\begin{array}{cc}
	I_M & -\beta I_M\\
	I_M & -I_M\\
	\end{array}
	\right].
	\end{align}
	Recall $\tw_i=w^o-\w_i$ and define the transformed error vectors, each of size $2M\times 1$:
	\bqq
	\left[
	\begin{array}{c}
		\hw_i \\
		\cw_i
	\end{array}
	\right]
	\define 
	V ^{-1}
	\left[                 %???
	\begin{array}{c}   %?????3??????????
		\tw_{i}\\  %?????
		\tw_{i-1}\\  %?????
	\end{array}
	\right] =\frac{1}{1-\beta}\left[\begin{array}{c}
		\widetilde{\w}_i-\beta\widetilde{\w}_{i-1}\\
		\widetilde{\w}_i-\widetilde{\w}_{i-1}
	\end{array}
	\right].
	\label{transform}
	\eqq
	
	\begin{lemma}[{\bf Extended recursion}]
		\label{lm:momentum transform}
		Under Assumption \ref{ass: cost function} and condition \eqref{ass: gfat}, the momentum stochastic gradient recursion \eqref{sgfat-1}--\eqref{sgfat-2} can be transformed into the following extended recursion:
		\begin{align}
		\label{sta-hb iteration 2}
		\left[                 %???
		\begin{array}{c}   %?????3??????????
		\widehat{\w}_{i}\\  %?????
		\cw_{i}\\  %?????
		\end{array}
		\right]                %???
		=&
		\left[   \hspace{-0.03cm}
		\begin{array}{cc}
		I_M-\frac{\mu_m}{1-\beta }\H_{i-1} &\frac{\mu_m\beta^\prime }{1-\beta}\H_{i-1}\\
		-\frac{\mu_m}{1-\beta} \H_{i-1} & \beta  I_M+\frac{\mu_m \beta^\prime}{1-\beta }\H_{i-1}\\
		\end{array}\hspace{-0.03cm}
		\right]
		\left[     \hspace{-0.08cm}
		\begin{array}{c}
		\hw_{i-1}\\
		\cw_{i-1}\\
		\end{array}\hspace{-0.08cm}
		\right]   +
		\frac{\mu_m}{1-\beta}\left[
		\begin{array}{c}
		\s_i(\bpsi_{i-1})\\
		\s_i(\bpsi_{i-1})\\
		\end{array}
		\right],
		\end{align}
		where $\s_i(\bpsi_{i-1})$ is defined according to \eqref{noise-2} and
		\bqq
		\beta^\prime &\define& \beta \beta_1 + \beta_2,  \label{beta prime} \\
		\H_{i-1} &\define& \int_0^1 \grad^2_w J(w^o - t \tpsi_{i-1})dt,\label{beta prime-22}
		\eqq
		where $\tpsi_{i-1}=w^o-\bpsi_{i-1}$.
	\end{lemma}
	\begin{proof}
		See Appendix \ref{app:lm-gfat-transform}.
	\end{proof}
	
	\noindent The transformed recursion \eqref{sta-hb iteration 2} is important for at least two reasons.  First, it is a first-order recursion, which facilitates the convergence analysis of $\hw_i$ and $\cw_i$ and, subsequently, of the error vector $\tw_i$ in view of \eqref{transform} --- see next theorem. Second, as we will explain
	later, the first row of \eqref{sta-hb iteration 2} turns out to be closely related to the standard stochastic gradient iteration; this relation will play a critical role in establishing the claimed equivalence between momentum and standard stochastic gradient methods. 
	
	The following statement establishes the convergence property of the momentum stochastic gradient algorithm. It shows that recursions \eqref{sgfat-1}--\eqref{sgfat-2} converge exponentially fast to a small neighborhood around $w^o$ with a steady-state error variance that is on the order of $O(\mu_m)$. Note that in the following theorem the notation $a\preceq b$, for two vectors $a$ and $b$, signifies element-wise comparisons.

	\begin{theorem}[{\bf Mean-square stability}]\label{thm-conv-cw} Let Assumptions \ref{ass: cost function} , \ref{ass: noise} and \ref{ass:beta-cond} hold and recall conditions \eqref{ass: gfat}. Consider the momentum stochastic gradient method \eqref{sgfat-1}--\eqref{sgfat-2} and the extended recursion \eqref{sta-hb iteration 2}. 
		Then, when step-sizes $\mu_m$ satisfies 
		\eq{\label{23789adhjj-stepsizes}
			\mu_m \le  \frac{(1-\beta)^2 \nu}{32\gamma^2\nu^2 + 4\delta^2},
		}
		it holds that the mean-square values of the transformed error vectors evolve according to the following recursive inequality:
		\eq{
			\label{thm: sta-hw-cw-bound}
			\left[                 %???
			\begin{array}{c}   %?????3??????????
				\bE\|\hw_{i}\|^2\\  %?????
				\bE\|\cw_{i}\|^2\\  %?????
			\end{array}
			\right]
			\preceq
			\left[                 %???
			\begin{array}{cc}   %?????3??????????
				a&b\\  %?????
				c&d\\  %?????
			\end{array}
			\right]
			\left[                 %???
			\begin{array}{c}   %?????3??????????
				\bE\|\hw_{i-1}\|^2\\  %?????
				\bE\|\cw_{i-1}\|^2\\  %?????
			\end{array}
			\right]
			+
			\left[                 %???
			\begin{array}{c}   %?????3??????????
				e\\  %?????
				f\\  %?????
			\end{array}
			\right],
		}
		where
		{\color{black}
		\begin{equation}
			\begin{array}{lll}
			a=1-\frac{\mu_m\nu}{1-\beta}+O(\mu_m^2),\quad\quad  & b= \frac{\mu_m \beta^{\prime 2} \delta^2}{\nu(1-\beta)} + O(\mu_m^2), \quad\quad\quad & c= \frac{2\mu_m^2 \delta^2}{(1-\beta)^3} + \frac{2\mu_m^2 \gamma^2 (1+\beta_1)^2v^2}{(1-\beta)^2}  \\
			d=\beta + O(\mu_m^2),\quad\quad  & e= \frac{\mu_m^2\sigma_s^2}{(1-\beta)^2},\quad\quad\quad & f= \frac{\mu_m^2\sigma_s^2}{(1-\beta)^2} \label{f u2}
			\end{array}
		\end{equation}
	}
%		\eq{
%			a&=1-\frac{\mu_m\nu}{1-\beta}+O(\mu_m^2), \\
%			b&= O(\mu_m) \\
%			c&=O(\mu_m^2) \\
%			d&=\beta + O(\mu_m^2) \\
%			e&=O(\mu_m^2) \\
%			f&=O(\mu_m^2) \label{f u2}
%		}
		and the coefficient matrix appearing in (\ref{thm: sta-hw-cw-bound}) is stable, namely,
		\eq{
			\rho \left( \left[                 %???
			\begin{array}{cc}   %?????3??????????
				a&b\\  %?????
				c&d\\  %?????
			\end{array}
			\right] \right)  < 1.
		}
		Furthermore, if $\mu_m$ is sufficiently small it follows from \eqref{thm: sta-hw-cw-bound} that
		{\color{black}
		\eq{
			\limsup_{i \rightarrow \infty} \bE\|\hw_i\|^2 = O\left(\frac{\mu_m \sigma_s^2}{(1-\beta) \nu} \right), \quad \quad  \limsup_{i \rightarrow \infty} \bE\|\cw_i\|^2 = O\left(\frac{\mu_m^2 \sigma_s^2}{(1-\beta)^3} \right),\label{thm-conv-cw-2}
		}
		and, consequently,
		
		\eq{
			\limsup_{i \rightarrow \infty} \bE\|\tw_i\|^2 = O\left(\frac{\mu_m \sigma_s^2}{(1-\beta) \nu} \right). \label{thm-conv-tw}
		}
	}
	\end{theorem}
	\begin{proof}
		See Appendix \ref{app:thm-2}.
	\end{proof}
	
	\noindent Although $\bE\|\cw_i\|^2=O(\mu_m^2)$ in result (\ref{thm-conv-cw-2}) is shown to hold asymptotically in the statement of the theorem, it can actually be strengthened and shown to hold for {\em all} time instants. This fact is crucial for our later proof of the equivalence between standard and momentum stochastic gradient methods. 
	
	\begin{corollary}[{\bf Uniform mean-square bound}]
		\label{co:hb-cw-u2}
		Under the same conditions as Theorem~\ref{thm-conv-cw}, it holds for sufficiently small step-sizes that
		{\color{black}
%				\eq{\label{hb-tw-u2}
%					\bE\|\cw_i\|^2 = O \left( \frac{(1 + \gamma^2 + \sigma_s^2 )\mu_m^2}{(1-\beta)^2} +  \frac{\delta^2 \mu_m^2}{(1-\beta)^3} \right), \quad \forall i=0,1,2,\cdots
%				}
%\eq{\label{hb-tw-u2}
%	\bE\|\cw_i\|^2 = O\left( \rho_2^{i+1}\mu^2 + \frac{(\delta^2 + \gamma^2)\rho_1^{i+1}\mu^2}{(1-\beta)^3} + \frac{\sigma_s^2 \mu^2}{1-\beta} \right), \forall i = 0,1,2,\cdots
%}
%
\eq{\label{hb-tw-u2}
	\bE\|\cw_i\|^2 = O\left( \frac{(\delta^2 + \gamma^2)\rho_1^{i+1}\mu_m^2}{(1-\beta)^4} + \frac{\sigma_s^2 \mu_m^2}{(1-\beta)^3} \right), \forall i = 0,1,2,\cdots
}
where $\rho_1 \define 1 - \frac{\mu_m \nu}{2(1-\beta)}$, and $\cw_i$ is defined in \eqref{sta-hb iteration 2}.

			}
		
%		\eq{
%			\label{hb-tw-u2}
%			\bE\|\cw_i\|^2=O(\mu_m^2),\quad \forall i=0,1,2,3,\ldots
%		}
	\end{corollary}
	\begin{proof}
		See Appendix \ref{app-cor-2-ub}.
	\end{proof}
	{\color{black}Corollary \ref{co:hb-cw-u2} has two implications. First, since $\beta$, $\delta$, $\gamma$, $\sigma_s^2$ are all constants, and $\rho_1<1, \alpha<1$, we conclude that
		\eq{\label{28adsah}
		\bE\|\cw_i\|^2 = O(\mu_m^2),\quad \forall i = 0,1,2,\cdots
		}
		Besides, since $\rho_1^i\to 0$ as $i\to \infty$, according to \eqref{hb-tw-u2} we also achieve
		\eq{\label{asdklj7234}
			\limsup_{i\to \infty} \bE\|\cw_i\|^2  =O \left( \frac{\sigma_s^2 \mu_m^2}{(1-\beta)^3} \right),
		}
		which is consistent with \eqref{thm-conv-cw-2}.
	}
	
	\subsection{Stability of Fourth-Order Error Moment}
	In a manner similar to the treatment in Section~\ref{sec: stochastic gradient method}, we can also
	establish the convergence of the fourth-order moments of the error vectors, $\bE\|\hw_i\|^4$ and $\bE\|\tw_i\|^4$.
	
	\begin{theorem}[{\bf Fourth-order stability}]\label{lm:gfat-conv-4-4}
		Let Assumptions \ref{ass: cost function}, \ref{ass: noise-4} and \ref{ass:beta-cond} hold and recall conditions (\ref{ass: gfat}). Then, for sufficiently small step-sizes $\mu_m$, it holds that
		{
		\eq{
			\limsup_{i \rightarrow \infty} \bE\|\hw_i\|^4 &= O (\mu_m^2),\label{thm-conv4-hw}\\
			\limsup_{i \rightarrow \infty} \bE\|\cw_i\|^4 &= O\left( \mu_m^4\right),\label{thm-conv4-cw}\\
			\limsup_{i \rightarrow \infty} \bE\|\tw_i\|^4 &= O \left(\mu_m^2\right).\label{thm-conv4-tw}
		}
	}
	\end{theorem}
	{\color{black}
	\begin{proof} See Appendix \ref{app:4-th order moment}.
	\end{proof}
}
	
	\noindent Again, result (\ref{thm-conv4-cw}) is only shown to hold asymptotically in the statement of the theorem. In fact, $\bE\|\cw_i\|^4$ can also be shown to be bounded for all time instants, as the following corollary states. 
	
	\begin{corollary}[{\bf Uniform forth-moment bound}]
		\label{co:hb-cw-u2-4}
		Under the same conditions as Theorem~\ref{lm:gfat-conv-4-4}, it holds for sufficiently small step-sizes that
		{\color{black}
%		\eq{\label{238asdbmnc}
%			\bE\|\cw_i\|^4 =&\ O \left( \frac{\gamma^2(i+1)\alpha^i}{(1-\beta)^2}\mu_m^2 + \frac{\sigma_s^2(\delta^2 + \gamma^2)(i^2+i)\alpha^{i-1}}{(1-\beta)^6}\mu_m^4 + \frac{\gamma^2\sigma_s^4 + \sigma_{s,4}^4}{(1-\beta)^5\nu^2}\mu_m^4\right),
%		}
		\eq{\label{238asdbmnc}
			\bE\|\cw_i\|^4 = O \left( \frac{\gamma^2\rho_2^{i+1}}{(1-\beta)^3}\mu_m^2 + \left[ \frac{\sigma_s^2(\delta^2 + \gamma^2)(i+1)\rho_2^{i+1}}{(1-\beta)^7} + \frac{(\gamma^2 + \nu^2)\sigma_s^4 + \nu^2\sigma_{s,4}^4}{(1-\beta)^6\nu^2}\right] \mu_m^4\right)
		}
		where $\rho_2 \define 1 - \frac{\mu_m \nu}{4(1-\beta)} \in(0,1)$.
	}
%		\eq{
%			\label{hb-tw-u2-4}
%			{\bE\|\cw_i\|^4=O(\mu_m^2),\quad \forall i=0,1,2,3,\ldots}
%		}
%		Furthermore, in the limit,
%		\eq{\label{corr-conv4-cw}
%			\limsup_{i \rightarrow \infty} \bE\|\cw_i\|^4 &= O(\mu_m^4).
%		}
	\end{corollary}
	\begin{proof} See Appendix \ref{app:corr-2aa}.
	\end{proof}
	{\color{black}Corollary \ref{co:hb-cw-u2-4} also has two implications. First, since $\beta$, $\delta$, $\gamma$, $\sigma_s$ and $\sigma_{s,4}$ are constants, we conclude that
		\eq{\label{28adsah-237a}
			\bE\|\cw_i\|^4 = O(\mu_m^2),\quad \forall i = 0,1,2,\cdots
		}
		Besides, since $\rho_2^i \to 0$ and $i\rho_2^i \to 0$ as $i\to \infty$, we will achieve the following fact according to \eqref{238asdbmnc}
		\eq{\label{asdklj7234-97sd}
			\limsup_{i\to \infty} \bE\|\cw_i\|^4  =O \left( \frac{(\gamma^2 + \nu^2)\sigma_s^4 + \nu^2\sigma_{s,4}^4}{(1-\beta)^6\nu^2} \mu_m^4 \right) = O(\mu_m^4),
		}
		which is consistent with \eqref{thm-conv4-cw}.
	}
	
	\section{Equivalence in the Quadratic Case}
	\label{Sec: LMS}
	In Section \ref{Sec: GFAT} we showed the momentum stochastic gradient algorithm
	(\ref{sgfat-1})--(\ref{sgfat-2}) converges exponentially for sufficiently small step-sizes. But some
	important questions remain. Does the momentum implementation converge faster than the
	standard stochastic gradient method (\ref{stochastic gradient method-2})? Does the momentum implementation
	lead to superior steady-state mean-square-deviation (MSD) performance, measured in terms of the limiting value of $\Ex\|\widetilde{\w}_i\|^2$? Is the momentum method generally superior to the standard method when considering both the convergence rate and MSD performance? In this and the next sections, we answer these questions in some detail. Before treating the case of general risk functions, $J(w)$, we examine first the special case when $J(w)$ is quadratic in $w$ to illustrate the main conclusions that will follow. 
	
%	
%	{\color{black}
%	The notion of equivalence in the stochastic setting is slightly different from equivalence in the deterministic setting. 
%	Suppose $\{x_i\}_{i=0}^\infty$ and $\{w_i\}_{i=0}^\infty$ are sequences of variables generated by two different algorithms.
%	In deterministic optimization, the two algorithms are said to be equivalent if, and only if, $x_i = w_i$ for any $i=0,1,2,\cdots$. In the stochastic case, when gradient noise is present, we shall regard two algorithms as equivalent if it holds that $\bE\|\w_i - \x_i\|^2$ is sufficiently small for any $i=0,1,2,\cdots$.} {\color{red}[Professor, Reviewer 4 was wondering the precise definition of the notion of equivalence (see Major technical remark $(1)$ of Reviewer 4). That's why I added this paragraph.]}
%	

	\subsection{Quadratic Risks}
	We consider mean-square-error risks of the form
	\eq{
		\label{lms model}
		J(w) = \frac{1}{2}\bE \left(\d(i)-\u_i \tran w \right)^2,
	}
	where $\d(i)$ denotes a streaming sequence of zero-mean random variables with variance $\sigma_d^2=\bE \d^2(i)$, and  $\u_i \in \real^M$ denotes
	a streaming sequence of independent zero-mean random vectors with covariance matrix $R_u=\bE\u_i \u_i\tran > 0$. The cross covariance vector
	between $\d(i)$ and $\u_i$ is denoted by $r_{du}=\bE \d(i)\u_i$. The data $\{\d(i), \u_i\}$ are assumed to be wide-sense stationary and related via a linear regression model of the form:
	\eq{
		\label{data model}
		\d(i)=\u_i\tran w^o + \v(i),
	}
	for some unknown $w^o$, and where $\v(i)$ is a zero-mean white noise process with power $\sigma_v^2=\bE \v^2(i)$ and assumed independent of $\u_j$ for all $i,j$. If we multiply (\ref{data model}) by $\u_i$ from the left and take expectations, we find that the model parameter $w^o$ satisfies the normal equations $R_u w^o=r_{du}$. The unique solution that minimizes (\ref{lms model}) also satisfies these same equations. Therefore, minimizing the quadratic risk (\ref{lms model}) enables us to recover the desired $w^o$. This observation explains why mean-square-error costs are popular in the context of regression models.

	\subsection{Adaptation Methods}
	For the least-mean-squares problem \eqref{lms model}, the true gradient vector at any location $\w$ is
	\eq{
		\label{lms true gradient}
		{\grad_w J}(\w) = R_u \w - r_{du}\;=\;-R_u (w^o-\w),
	}
	while the approximate gradient vector constructed from an instantaneous sample realization is:
	\eq{
		\label{lms approximate gradient}
		\grad_w Q(\w;\d(i),\u_i) = -\u_i (\d(i)-\u_i\tran \w).
	}
	Here the loss function is defined by
	\be Q(w;\d(i),\u_i)\define {1\over 2}\bE \left(\d(i)-\u_i \tran w \right)^2\ee
	The resulting LMS (stochastic-gradient) recursion is given by
	\be\label{eq-LMS}
	\w_i=\w_{i-1}+\mu \u_i (\d(i)-\u_i\tran \w_{i-1})
	\ee
	and the corresponding gradient noise process is
	\eq{
		\label{lms noise}
		\s_i(\w) &= (R_u-\u_i \u_i\tran) (w^o-\w) - \u_i \v(i).
	}
	It can be verified that this noise process satisfies Assumption \ref{ass: noise} --- see Example 3.3 in \citep{sayed2014adaptation}. Subtracting $w^o$ from both sides of (\ref{eq-LMS}), and recalling that $\tw_i=w^o-\w_i$, we obtain the error recursion that corresponds to the LMS implementation:
	\eq{
		\label{stochastic gradient method tw update-2}
		\tw_i =(I_M-\mu R_u)\tw_{i-1} + \mu \s_i(\w_{i-1}),
	}
	where $\mu$ is some constant step-size.
	In order to distinguish the variables for LMS from the variables for the momentum LMS version described below, we replace the notation $\{\w_i,\tw_i\}$ for LMS by $\{\x_i,\tx_i\}$ and keep the notation $\{\w_i,\tw_i\}$  for momentum LMS, i.e., for the LMS implementation \eqref{stochastic gradient method tw update-2} we shall write instead
	\eq{
		\label{stochastic gradient method tx update}
		\tx_i =(I_M-\mu R_u)\tx_{i-1} + \mu \s_i(\x_{i-1}).
	}

	On the other hand, we conclude from (\ref{sgfat-1})--(\ref{sgfat-2}) that the momentum LMS recursion will be given by:
	\bqq
	\bpsi_{i-1} &=&\w_{i-1} + \beta_1 (\w_{i-1} - \w_{i-2}), \label{sgfat-1-lms}\\
	\w_i  &=& \bpsi_{i-1} +
	\mu_m \u_i (\d(i)-\u_i\tran \bpsi_{i-1}) +\beta_2(\bpsi_{i-1}-\bpsi_{i-2}), \label{sgfat-2-lms}
	\eqq
	Using the transformed recursion (\ref{sta-hb iteration 2}), we can transform the resulting relation for
	$\widetilde{\w}_i$  into:
	\begin{align}
	\label{hb iteration 2}
	\left[                 %???
	\begin{array}{c}   %?????3??????????
	\widehat{\w}_{i}\\  %?????
	\cw_{i}\\  %?????
	\end{array}
	\right]                %???
	=&
	\left[
	\begin{array}{cc}
	I_M-\frac{\mu_m}{1-\beta}R_u &\frac{\mu_m\beta^\prime}{1-\beta}R_u\\
	-\frac{\mu_m}{1-\beta} R_u & \beta I_M+\frac{\mu_m \beta^\prime}{1-\beta}R_u\\
	\end{array}
	\right]
	\left[
	\begin{array}{c}
	\hw_{i-1}\\
	\cw_{i-1}\\
	\end{array}
	\right]  
	+
	\frac{\mu_m}{1-\beta}\left[
	\begin{array}{c}
	\s_i(\bpsi_{i-1})\\
	\s_i(\bpsi_{i-1})\\
	\end{array}
	\right],
	\end{align}
	where the Hessian matrix, $\H_{i-1}$, is independent of the weight iterates and given by $R_u$ for quadratic risks. It follows from the first row that 
	\beqn
	\label{rl tw}
	\hw_i&=&\left(I_M-\frac{\mu_m}{1-\beta}R_u\right)\hw_{i-1}+\frac{\mu_m \beta^\prime}{1-\beta}R_u\cw_{i-1} +\frac{\mu_m}{1-\beta}\s_i(\bpsi_{i-1}). \label{wi update}
	\eeqn
	Next, we assume the step-sizes $\{\mu,\mu_m\}$ and the momentum parameter are selected to satisfy
	\eq{\mu=\frac{\mu_m}{1-\beta}.\label{cond.stepsize}}
	Since $\beta\in [0,1)$, this means that $\mu_m < \mu$. Then, recursion
	\eqref{rl tw} becomes
	\be
	\label{rl tw-2}
	\hw_i=(I_M-\mu R_u)\hw_{i-1}+\mu \beta^\prime R_u\cw_{i-1} +\mu \s_i(\bpsi_{i-1}).
	\ee
	Comparing \eqref{rl tw-2} with the LMS recursion (\ref{stochastic gradient method tx update}), we find that both relations are quite similar, except that the momentum recursion has an extra driving term dependent on $\cw_{i-1}$. However, recall from \eqref{transform} that $\cw_{i-1}=(\tw_{i-2}-\tw_{i-1})/(1-\beta)$, which is the difference
	between two consecutive points generated by momentum LMS. Intuitively, it is not hard to see that $\cw_{i-1}$ is in the order of $O(\mu)$, which makes $\mu \beta^\prime R_u \cw_{i-1}$
	in the order of $O(\mu^2)$. When the step-size $\mu$ is very small, this $O(\mu^2)$ term can be ignored. Consequently, the above recursions for $\hw_i$ and $\tx_i$ should evolve close to each other, which would {\color{black}help to prove that} $\w_i$ and $\x_i$ will also evolve close to each other as well. This conclusion can be established formally as follows, which proves the equivalence between
	the momentum and standard LMS methods.
	
	\begin{theorem}[{\bf Equivalence for LMS}]
		\label{thm:gfat-stochastic gradient method-dist-hw} Consider the
		LMS and momentum LMS recursions (\ref{eq-LMS}) and (\ref{sgfat-1-lms})--(\ref{sgfat-2-lms}). {\color{black}Let Assumptions \ref{ass: cost function}, \ref{ass: noise} and \ref{ass:beta-cond} hold.} Assume both algorithms start from the same initial states, namely, $\bpsi_{-2}=\w_{-2}=\x_{-1}$. Suppose conditions (\ref{ass: gfat}) holds, and that the step-sizes $\{\mu,\mu_m\}$ satisfy (\ref{cond.stepsize}). Then, it holds for sufficiently small $\mu$ that for $\forall\ i = 0,1,2,3,\cdots$
		{\color{black}
		\eq{
			\label{hb-stochastic gradient method-dist-tw-2}
			\bE\|\w_i-\x_i\|^2=O\left( \left[\frac{\delta^2 + \gamma^2}{(1-\beta)^2} \rho_1^{i+1} +  \frac{\delta^2 \sigma_s^2  }{\nu^2(1-\beta)} \right]\mu^2 +  \frac{\delta^2(\delta^2 + \gamma^2) (i+1)\rho_1^{i+1}}{\nu(1-\beta)^2}\mu^3  \right).
		}
	where $\rho_1 = 1 - \frac{\mu\nu}{2} \in (0,1)$.
	}
	\end{theorem}
	\begin{proof} See Appendix \ref{proof of thm 3}.
	\end{proof}
	\noindent {\color{black}Similar to Corollary \ref{co:hb-cw-u2} and \ref{co:hb-cw-u2-4}, Theorem \ref{thm:gfat-stochastic gradient method-dist-hw} also has two implications. First, it holds that 
		\eq{\label{28adsah-2}
			\bE\|\w_i-\x_i\|^2 = O(\mu^2),\quad \forall i = 0,1,2,\cdots
		}
		Besides, since $\rho_1^i\to 0$ and $i\rho_1^{i}$ as $i\to \infty$, we also conclude 
		\eq{\label{asdklj7234-2}
			\limsup_{i\to \infty} \bE\|\w_i-\x_i\|^2  =O \left( \frac{\delta^2 \sigma_s^2 \mu^2}{\nu^2 (1-\beta)} \right).
		}
		
		}
	Theorem \ref{thm:gfat-stochastic gradient method-dist-hw}  establishes that the standard and momentum LMS algorithms are fundamentally equivalent since their iterates evolve close to each other at all times for sufficiently small step-sizes. More interpretation of this result is discussed in Section \ref{interpretation}. 
%	\rev{
%	\begin{remark}
%		h
%	\end{remark}
%	}

	\section{Equivalence in the General Case}
	\label{sec: equi-ge}
	We now extend the analysis from quadratic risks to more general risks (such as logistic risks). The analysis in this case is more demanding because the Hessian matrix of $J(w)$ is now $w-$dependent, but the same equivalence conclusion will continue to hold as we proceed to show.
	
	\subsection{Equivalence in the General Case}
	Note from the momentum recursion \eqref{sta-hb iteration 2} that
	\eq{
		\label{ge-hb-hw}
		\hw_i=&\left(I_M-\frac{\mu_m}{1-\beta}\H_{i-1}\right)\hw_{i-1}+\frac{\mu_m \beta^\prime}{1-\beta}\H_{i-1}\cw_{i-1} +\frac{\mu_m}{1-\beta}\s_i(\bpsi_{i-1}),
	}
	where $\H_{i-1}$ is defined by (\ref{beta prime-22}). In the quadratic case, this matrix was constant and equal to the covariance matrix, $R_u$. Here, however, it is time-variant and depends on the error vector,
	$\widetilde{\bpsi}_{i-1}$, as well. Likewise, for the standard stochastic gradient iteration \eqref{stochastic gradient method-2}, we obtain that
	the error recursion in the general case is given by:
	\beqn
	\label{ge-stochastic gradient method-tx}
	\tx_i =(I_M-\mu  \R_{i-1})\tx_{i-1} +\mu   \s_i(\x_{i-1}),
	\eeqn
	where we are introducing the matrix
	\be \R_{i-1}=\int_0^1 \grad^2_w J(w^o - r \tx_{i-1})dr\ee
	and $\tx_i=w^o-\x_i$. Note that $\H_{i-1}$ and $\R_{i-1}$ are different matrices. In contrast, in the quadratic case, they are both equal to $R_u$.
	
	Under the assumed condition (\ref{cond.stepsize}) relating $\{\mu,\mu_m\}$, if we subtract \eqref{ge-stochastic gradient method-tx} from \eqref{ge-hb-hw} we obtain:
	\eq{
		\label{ge-dist-hw-tx}
		\hw_i-\tx_i=&\ (I_M-\mu \H_{i-1})(\hw_{i-1}-\tx_{i-1}) + \mu(\R_{i-1}-\H_{i-1})\tx_{i-1} \nnb
		& + {\mu \beta^\prime}\H_{i-1}\cw_{i-1} + \mu[\s_i(\bpsi_{i-1})-\s_i(\x_{i-1})].
	}
	In the quadratic case, the second term on the right-hand side is zero since $\R_{i-1}=\H_{i-1}=R_u$. It is
	the presence of this term that makes the analysis more demanding in the general case.
	
	To examine how close $\hw_i$ gets to $\tx_i$ for each iteration, we start by noting that	
%	\[\hspace{-6cm} \bE\|\hw_i-\tx_i\|^2 \vspace{-0.2cm}\]
	\eq{
		\label{ge-dist-hw-tx-e}
		\bE\|\hw_i-\tx_i\|^2
		& = \bE \left\| (I_M-\mu \H_{i-1})(\hw_{i-1}-\tx_{i-1})+  \mu(\R_{i-1}-\H_{i-1})\tx_{i-1} + {\mu \beta^\prime}\H_{i-1}\cw_{i-1} \right\|^2 \nnb
		&\ \quad + \mu^2 \bE\| \s_i(\bpsi_{i-1})-\s_i(\x_{i-1}) \|^2.
	}
	Now, applying a similar derivation to the one used to arrive at \eqref{sta-hw update-expe-filt} in Appendix \ref{app:thm-2}, and the inequality $\|a+b\|^2\leq 2\|a\|^2+2\|b\|^2$, we can conclude from (\ref{ge-dist-hw-tx-e}) that
	\eq{
		\label{ge-dist-hw-tx-e-2}
		\bE\|\hw_i-\tx_i\|^2
		\le&\ (1-\mu \nu)\bE \| \hw_{i-1}-\tx_{i-1}\|^2 + \frac{2\mu \beta^{\prime 2} \delta^2 }{\nu} \bE\|\cw_{i-1} \|^2 \nnb
		&\ + \frac{2\mu}{\nu} \bE \| (\R_{i-1}-\H_{i-1})\tx_{i-1} \|^2 + \mu^2 \bE\| \s_i(\bpsi_{i-1})-\s_i(\x_{i-1}) \|^2.
	}
	Using the Cauchy-Schwartz inequality we can bound the cross term as
%	
%	\[
%	\hspace{-3cm}\bE \| (\R_{i-1}-\H_{i-1})\tx_{i-1} \|^2\vspace{-0.2cm}\]
	\bqq
	\bE\| (\R_{i-1}\hspace{-1mm}-\hspace{-1mm}\H_{i-1})\tx_{i-1}\|^2 \hspace{-0.5mm} \le \hspace{-0.5mm} \bE (\| \R_{i-1}-\H_{i-1} \|^2 \|\tx_{i-1}\|^2 ) \hspace{-0.5mm} \le \hspace{-0.5mm} \sqrt{\bE \| \R_{i-1}-\H_{i-1} \|^4\, \bE\|\tx_{i-1}\|^4}.\label{bound-72}
	\eqq
	%where we used the property that $\Ex\a\leq(\Ex\a^2)^{1/2}$ for any nonnegative random variable $\a$.
	In the above inequality, the term $\bE\|\tx_{i-1}\|^4$ can be bounded by using the result of Lemma \ref{lm: x4 bound}. Therefore, we focus on bounding  $\bE \| \R_{i-1}-\H_{i-1} \|^4$ next. To do so, we need to introduce the following smoothness assumptions on the second and fourth-order moments of the gradient noise process and on the Hessian matrix of the risk function. These assumptions hold automatically for important cases of interest, such as least-mean-squares and logistic regression problems --- see Appendix \ref{app-feasibility} for the verification.
	
	\begin{assumption}
		\label{ass:noise-lipschitz}
		Consider the iterates $\bpsi_{i-1}$ and $\x_{i-1}$ that are generated by
		the momentum recursion \eqref{sgfat-1} and the stochastic gradient recursion \eqref{stochastic gradient method-2}. It is assumed that the gradient noise process satisfies:
		%\eq{
		%\bE \|\s_i(\bpsi_{i-1})-\s_i(\x_{i-1})\|^2 & \le \xi_1 \bE \|\bpsi_{i-1}-\x_{i-1}\|^2,\label{cond-1a} \\
		%\bE \|\s_i(\bpsi_{i-1})-\s_i(\x_{i-1})\|^4 & \le \xi_2 \bE \|\bpsi_{i-1}-\x_{i-1}\|^4.\label{cond-2a}
		%}
		\eq{
			\bE [\|\s_i(\bpsi_{i-1})\hspace{-0.8mm}-\hspace{-0.8mm}\s_i(\x_{i-1})\|^2|\filt_{i-1}] & \le \xi_1 \|\bpsi_{i-1}\hspace{-0.8mm}-\hspace{-0.8mm}\x_{i-1}\|^2,\label{cond-1a} \\
			\bE [\|\s_i(\bpsi_{i-1})\hspace{-0.8mm}-\hspace{-0.8mm}\s_i(\x_{i-1})\|^4|\filt_{i-1}] & \le \xi_2 \|\bpsi_{i-1}\hspace{-0.8mm}-\hspace{-0.8mm}\x_{i-1}\|^4.\label{cond-2a}
		}
		for some constants $\xi_1$ and $\xi_2$.
		\qd
	\end{assumption}
	
	\begin{assumption}
		\label{ass: hessian lipschitz}
		The Hessian of the risk function $J(w)$ in \eqref{general prob} is Lipschitz continuous, i.e., for any two variables $w_1, w_2 \in \dom\ J(w)$, it holds that
		\eq{
			\|\grad_w ^2 J(w_1) - \grad_w ^2 J(w_2)\| \le \kappa \|w_1 - w_2 \|.
		}
		for some constant $\kappa\geq 0$. \qd
		
	\end{assumption}
	
	\noindent Using these assumptions, we can now establish two auxiliary results in preparation for the main equivalence theorem in the general case.
	
	\begin{lemma}[{\bf Uniform bound}]
		\label{lm: hw-tx-4-bound}
		Consider the standard and momentum stochastic gradient recursions \eqref{stochastic gradient method-2} and \eqref{sgfat-1}-\eqref{sgfat-2} and assume they start from the same initial states, namely, $\bpsi_{-2}=\w_{-2}=\x_{-1}$. We continue to assume conditions (\ref{ass: gfat}), and (\ref{cond.stepsize}). {\color{black}Under Assumptions \ref{ass: cost function}, \ref{ass: noise-4}, \ref{ass:beta-cond}, \ref{ass:noise-lipschitz}} and for sufficiently small step-sizes $\mu$, the following result holds:
		{\color{black}
		\eq{\label{23789asdnm}
		\bE\|\tpsi_i-\tx_i\|^4 = 
		O\left( \frac{\delta^4 (i+1) \rho_2^{i+1}\mu}{\nu^3} + \frac{\delta^4\sigma_s^4}{\nu^6}\mu^2\right),
		}
		where $\rho_2 = 1 - \mu\nu/4$.
	}
%		\bqq
%		\bE\|\tpsi_i-\tx_i\|^4&=&O(\mu),\quad \forall i=0,1,2,3\ldots,\label{hb-stochastic gradient method-dist-hw-76}
%		\\
%		\limsup_{i\rightarrow \infty} \bE\|\tpsi_i-\tx_i\|^4&=&O(\mu^2).
%		\label{hb-stochastic gradient method-dist-tw}
%		\eqq
	\end{lemma}
	
	\begin{proof}
		See Appendix \ref{app-lm-hw-tx-4-bound}.
	\end{proof}
	
	\noindent Although sufficient for our purposes, we remark that the bound
	(\ref{23789asdnm}) for $\bE\|\tpsi_i-\tx_i\|^4$ is not 
	tight. The reason is that
	in the derivation in Appendix \ref{app-lm-hw-tx-4-bound} we employed a looser bound for
	the term $\bE\|(\R_{i-1}-\H_{i-1})\tx_{i-1}\|^4$ in order to avoid the appearance of
	higher-order powers, such as  $\bE\|\R_{i-1}-\H_{i-1}\|^8$
	and $\bE\|\tx_{i-1}\|^8$. To avoid this possibility, we employed the following bound
	(using (\ref{cost function assumption}) to bound $\|\R_{i-1}\|^4$ and $\|\H_{i-1}\|^4$ and the
	inequality $\|a+b\|^4\leq 8\|a\|^4+8\|b\|^4$):
	\eq{
	\bE \|(\R_{i-1}-\H_{i-1})\tx_{i-1}\|^4 &\le \bE \|\R_{i-1}-\H_{i-1}\|^4 \|\tx_{i-1}\|^4 \nnb
	&\le 8\bE\big\{ (\|\R_{i-1}\|^4 + \|\H_{i-1}\|^4) \|\tx_{i-1}\|^4 \big\} \le  16\delta^4\bE\|\tx_{i-1}\|^4.
	}
	Based on Lemma \ref{lm: hw-tx-4-bound}, we can now bound $\bE\|\R_{i-1}-\H_{i-1}\|^4$, which is what the following lemma states.
	\begin{lemma}[{\bf Bound on Hessian difference}]
		\label{lm: hw-tx-4-bound-hessian}
		Consider the same setting of Lemma~\ref{lm: hw-tx-4-bound}.  Under Assumptions \ref{ass: cost function}, \ref{ass: noise-4}, \ref{ass:beta-cond}, \ref{ass: hessian lipschitz} and for sufficiently small step-sizes $\mu$, the following two result holds:
		{\color{black}
			\eq{\label{23789asdnm-23bgasd}
				\bE \| \R_{i-1}-\H_{i-1} \|^4 = O\left( \frac{\delta^4 i \rho_2^{i}\mu}{\nu^3} + \frac{\delta^4\sigma_s^4}{\nu^6}\mu^2\right),
			}
			where $\rho_2 = 1 - \mu\nu/4$.
		}
%		\bqq
%		\bE \| \R_{i-1}-\H_{i-1} \|^4 &\le& C\mu,\ \ \forall i=0,1,2,3\ldots,\label{bound-e-h}\\
%		\limsup_{i\rightarrow \infty} \bE \| \R_{i-1}-\H_{i-1} \|^4  &=& O(\mu^2).\label{limsup e-h bound}
%		\eqq
%		for some constant $C$.
	\end{lemma}
	
	\begin{proof}
		See Appendix \ref{Hessian diff-app}.
	\end{proof}
	
	\noindent With the upper bounds of $\bE \| \R_{i-1}-\H_{i-1} \|^4$ and $\bE\|\tx_{i-1}\|^4$ established in Lemma \ref{lm: hw-tx-4-bound-hessian} and Lemma \ref{lm: x4 bound} respectively,
	we are able to bound $\| (\R_{i-1}-\H_{i-1})\tx_{i-1}\|$ in \eqref{bound-72}, which in turn helps to establish the main equivalence result.
	
	\begin{theorem}[{\bf Equivalence for general risks}]
		\label{thm:ge-dist} Consider the standard and momentum stochastic gradient recursions
		\eqref{stochastic gradient method-2} and \eqref{sgfat-1}--\eqref{sgfat-2} and assume they start from the same
		initial states, namely, $\bpsi_{-2}=\w_{-2}=\x_{-1}$. Suppose conditions (\ref{ass: gfat})  and
		(\ref{cond.stepsize}) hold. {\color{black}Under Assumptions \ref{ass: cost function}, \ref{ass: noise-4}, \ref{ass:beta-cond}, \ref{ass:noise-lipschitz},
		and \ref{ass: hessian lipschitz}}, and for sufficiently small step-size $\mu$, it holds that
		{\color{black}
		\eq{
			\label{hb-stochastic gradient method-dist-hw-asdj}
			\bE\|\tw_i-\tx_i\|^2=O \left( \frac{\delta^2\sigma_s^2 i^2 \tau_2^{i+1} \mu^{3/2} }{(1-\beta)\nu^{5/2}} + \left[\frac{(\delta^2 + \gamma^2)\rho_1^{i+1}}{(1-\beta)^2} +
			\frac{\delta^2 \sigma_s^4 }{(1-\beta)\nu^2}\right]\mu^2 \right),\quad \forall i=0,1,2,3,\ldots
		}
		where $\rho_1 = 1 - \frac{\mu\nu}{2} \in (0,1)$ and $\tau_2 \define \sqrt{1 - \mu\nu/4}\in (0,1)$. 
%		Furthermore, in the limit,
%		\eq{
%			\label{hb-stochastic gradient method-dist-tw-aslhui}
%			\limsup_{i\rightarrow \infty} \bE\|\tw_i-\tx_i\|^2=O\left(\frac{\delta^2 \sigma_s^4 \mu^2}{(1-\beta)\nu^2}\right)=O(\mu^2).
%		}
	}
%		\eq{
%			\label{hb-stochastic gradient method-dist-hw}
%			\bE\|\tw_i-\tx_i\|^2=O(\mu^{3/2}),\quad \forall i=0,1,2,3,\ldots
%		}
%		Furthermore, in the limit,
%		\eq{
%			\label{hb-stochastic gradient method-dist-tw}
%			\limsup_{i\rightarrow \infty} \bE\|\tw_i-\tx_i\|^2=O(\mu^2).
%		}
		
	\end{theorem}
	\begin{proof}
		See Appendix \ref{thm-ge-equi-app}.
	\end{proof}
	\noindent {\color{black}Similar to Corollary \ref{co:hb-cw-u2}, \ref{co:hb-cw-u2-4} and Theorem \ref{thm:gfat-stochastic gradient method-dist-hw}, Theorem \ref{thm:ge-dist} implies that
		\eq{\label{28adsah-2-yuw}
			\bE\|\w_i-\x_i\|^2 = O(\mu^{3/2}),\quad \forall i = 0,1,2,\cdots
		}
		Besides, since $\rho_1^i\to 0$ and $i^2\tau_2^{i}\to 0$ as $i\to \infty$, we will also conclude 
		\eq{
			\label{hb-stochastic gradient method-dist-tw-aslhui}
			\limsup_{i\rightarrow \infty} \bE\|\tw_i-\tx_i\|^2=O\left(\frac{\delta^2 \sigma_s^4 \mu^2}{(1-\beta)\nu^2}\right)=O(\mu^2).
		}}
		
	{\color{black}\noindent {\bf Remark}
			{ When we refer to ``sufficiently small step-sizes'' in Theorems \ref{thm:gfat-stochastic gradient method-dist-hw} and \ref{thm:ge-dist}, we mean that step-sizes are smaller than the
			stability bound, and are also small enough to ensure a desirable level of mean-square-error based on the performance expressions.} 
	} 
	
	\subsection{Interpretation of Equivalence Result}
	\label{interpretation}
	The result of Theorem \ref{thm:ge-dist} shows that, for sufficiently small step-sizes, the trajectories of momentum and standard stochastic gradient methods remain within $O(\mu^{3/2})$ from each other for {\em every} $i$ (for quadratic cases the trajectories will remain within $O(\mu^2)$ as stated in Theorem \ref{thm:gfat-stochastic gradient method-dist-hw}).
	This means that these trajectories evolve together for all practical purposes and, hence, we shall say that the two implementations are \rev{``equivalent'' (meaning that their trajectories remain close to each other in the mean-square-error sense).}
	
	A second useful insight from Theorem \ref{thm:gfat-stochastic gradient method-dist-hw} is that
	the momentum method is essentially equivalent to running a standard stochastic gradient method with a larger step-size (since $\mu>\mu_m$). This
	interpretation explains why the momentum method is observed to converge faster during the transient phase albeit towards a worse MSD level in steady-state than the standard method. This is because, as is well-known in the adaptive filtering literature \citep{Sayed2011adaptive, sayed2014adaptation} that larger step-sizes for stochastic gradient method do indeed lead to faster convergence but worse limiting performance.  
	
	In addition, Theorem \ref{thm:ge-dist} enables us to compute the steady-state MSD performance of the momentum stochastic gradient method. It is guaranteed by Theorem \ref{thm:ge-dist} that momentum method is equivalent to standard stochastic gradient method with larger step-size, $\mu=\mu_m/(1-\beta)$. Therefore, once we compute the MSD performance of the standard stochastic gradient, according to \citep{haykin2008adaptive, Sayed2011adaptive, sayed2014adaptation}, we will also know the MSD performance for the momentum method.
	
	Another consequence of the equivalence result is that
	any benefits that would be expected from a momentum {\color{black}stochastic gradient descent} can be attained by simply using a standard stochastic gradient implementation with a larger step-size; this is achieved without the additional computational or memory burden that the momentum method entails.
	
	Besides the theoretical analysis given above, there is an intuitive explanation as to why the momentum variant leads to worse steady-state performance. While the momentum terms $\w_i-\w_{i-1}$ and $\bpsi_i-\bpsi_{i-1}$ can smooth the convergence trajectories, and hence accelerate the convergence rate, they nevertheless introduce
	additional noise into the evolution of the algorithm because all iterates $\w_i$ and $\bpsi_i$ are distorted by perturbations.
	This fact illustrates the essential difference between stochastic methods with constant step-sizes, and stochastic or deterministic methods with decaying step-sizes: in the former case, the presence of gradient noise essentially eliminates the benefits of the momentum term.

	\subsection{Stochastic Gradient Method with Diminishing Momentum}\label{diminishing momentum}
	{\color{black} \citep{tygert2016poor,yuan2016influence} suggest one useful technique to retain the advantages of the momentum implementation by employing a {\em diminishing} momentum parameter, $\beta(i)$, and by ensuring $\beta(i)\rightarrow 0$ in order not to degrade the limiting performance of the implementation. By doing so, the momentum term will help accelerate the convergence rate during the transient phase {\color{black} because it will smooth the trajectory \citep{nedic2001convergence,xiao2010dual,lan2012optimal}}. On the other hand, momentum will not cause degradation in MSD performance because the momentum effect would have died before the algorithm reaches state-state. }
	
	{\color{black} According to \citep{tygert2016poor,yuan2016influence}, we adapt the momentum stochastic method into the following algorithm}
	\bqq
	\hspace{-0.4cm}\bpsi_{i-1}\hspace{-0.2cm} &=&\hspace{-0.2cm}\w_{i-1} + \beta_{1}(i) (\w_{i-1} - \w_{i-2}), \label{d-sgfat-1}\\
	\hspace{-0.4cm}\w_i \hspace{-0.2cm} &=&\hspace{-0.2cm} \bpsi_{i-1} \hspace{-0.1cm}-\hspace{-0.1cm} \mu \grad_w Q(\bpsi_{i-1};\btheta_i)\hspace{-0.1cm} +\hspace{-0.1cm} \beta_{2}(i)(\bpsi_{i-1}\hspace{-0.1cm}-\hspace{-0.1cm}\bpsi_{i-2}), \label{d-sgfat-2}
	\eqq
	with the same initial conditions as in \eqref{sgfat-3}--\eqref{sgfat-4}. Similar to condition \eqref{ass: gfat}, $\beta_{1}(i)$ and $\beta_{2}(i)$ also need to satisfy
	\eq{
		\beta_{1}(i) + \beta_{2}(i) = \beta(i), \quad \beta_{1}(i) \beta_{2}(i)=0,
	}
	The efficacy of \eqref{d-sgfat-1}--\eqref{d-sgfat-2} will depend on how the momentum decay, $\beta(i)$, is selected. A satisfactory sequence $\{\beta(i)\}$ should decay slowly during the initial stages of adaptation so that the momentum term can induce an acceleration effect. However, the sequence $\{\beta(i)\}$ should also decrease drastically prior to steady-state so that the vanishing momentum term will not introduce additional gradient noise and degrade performance. One strategy, which is also employed in the numerical experiments in Section \ref{Sec: exp}, is to design $\beta(i)$ to decrease in a stair-wise fashion, namely, 
	\eq{\label{beta-decrease stair-wise}
		\beta(i)=
		\begin{cases}
			{\begin{array}{ll}
					\beta_0 & \mbox{if $i\in [1,T]$}, \\
					\beta_0/T^\alpha & \mbox{if $i\in [T+1,2T]$}, \\
					\beta_0/(2T)^\alpha & \mbox{if $i\in [2T+1,3T]$}, \\
					\beta_0/(3T)^\alpha & \mbox{if $i\in [3T+1,4T]$}, \\
					\cdots
				\end{array}}
			\end{cases}
		}
		where the constants $\beta_0\in[0,1)$, $\alpha\in (0,1)$ and $T>0$ determines the width of the stair steps. Fig.~\ref{fig:beta_vary} illustrates how $\beta(i)$ varies when $T=20$, $\beta_0=0.5$ and $\alpha=0.4$.
		
		\begin{figure}[h]
			\centering
			\includegraphics[scale=0.5]{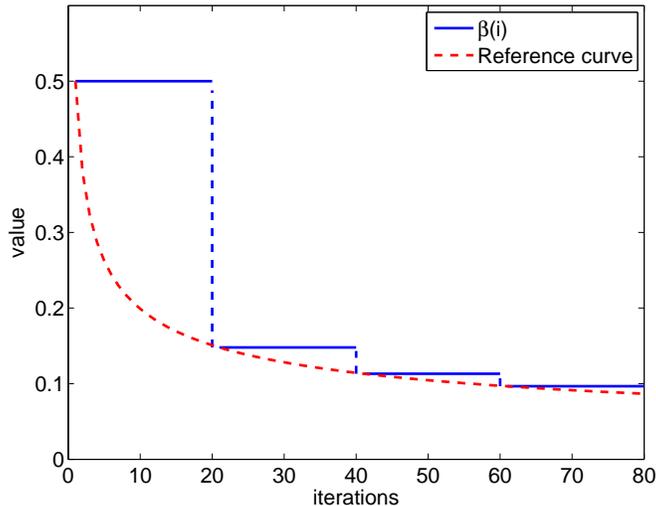}
			\vspace{-5mm}
			\caption{$\beta(i)$ changes with iteration $i$ according to \eqref{beta-decrease stair-wise}, where $\beta_0=0.5$, $T=20$ and $\alpha=0.4$. The reference curve is $f(i)=0.5/i^{0.4}$.}
			\label{fig:beta_vary}
		\end{figure}
		
		{\color{black}Algorithm \eqref{d-sgfat-1}--\eqref{d-sgfat-2} works well when $\beta(i)$ decreases according to \eqref{beta-decrease stair-wise} (see Section \ref{Sec: exp}). However, with Theorems \ref{thm:gfat-stochastic gradient method-dist-hw} and \ref{thm:ge-dist}, we find that this algorithm is essentially equivalent to the standard stochastic gradient method with decaying step-size, i.e., 
		\eq{\label{nclkjasdu}
		\x_i = \x_{i-1} - \mu_s(i) \grad_w Q(\x_{i-1};\theta_i), 
		}
		where 
		\eq{\label{mu_s-i}
		\mu_s(i) = \frac{\mu}{1-\beta(i)}
		}
		will decrease from $\mu/[1-\beta(0)]$ to $\mu$. In another words, the stochastic algorithm with decaying momentum is still not helpful.    
		
		}
		
		\section{Diagonal Step-size Matrices}
		\label{sec: extension}
		\noindent Sometimes it is advantageous to employ separate step-size for the individual entries of the weight vectors, see \citep{duchi2011adaptive}. In this section we comment on how the results from the previous sections extend to this scenario. First, we note that recursion (2) can be generalized to the following form, with a diagonal matrix serving as the step-size parameter:
		\eq{\label{stochastic gradient method-2-m}
			\x_i=\x_{i-1} - D \grad_w Q(\x_{i-1};\btheta_i),\;i\geq 0,
		}
		where $D=\mathrm{diag}\{\mu_1,\mu_2,\ldots,\mu_M\}$. Here, we continue to use the letter ``$\x$'' to refer to the variable iterates for the standard stochastic gradient descent iteration, while we reserve the letter ``$\w$'' for the momentum recursion. We let $\mua=\max\{\mu_1,\ldots,\mu_M\}$. Similarly, recursions \eqref{sgfat-1} and \eqref{sgfat-2} can be extended in the following manner:
		\bqq
		\hspace{-0.4cm}\bpsi_{i-1}\hspace{-0.2cm} &=&\hspace{-0.2cm}\w_{i-1} + B_1 (\w_{i-1} - \w_{i-2}), \label{sgfat-1-m}\\
		\hspace{-0.4cm}\w_i \hspace{-0.2cm} &=&\hspace{-0.2cm} \bpsi_{i-1} \hspace{-0.1cm}-\hspace{-0.1cm} D_m \grad_w Q(\bpsi_{i-1};\btheta_i)\hspace{-0.1cm} +\hspace{-0.1cm} B_2 (\bpsi_{i-1}\hspace{-0.1cm}-\hspace{-0.1cm}\bpsi_{i-2}), \label{sgfat-2-m}
		\eqq
		with initial conditions 
		\bqq
		\hspace{-0.2cm}&&\w_{-2} = \bpsi_{-2}=\mbox{initial states}, \label{sgfat-3-m}\\
		\hspace{-0.2cm}&&\w_{-1}=\w_{-2}-D_m {\grad_w Q}(\w_{-2};\btheta_{-1}),  \label{sgfat-4-m}
		\eqq
		where $B_1=\mathrm{diag}\{\beta^1_1,\ldots,\beta^1_M\}$ and
		$B_2=\mathrm{diag}\{\beta^2_1,\ldots,\beta^2_M\}$ are momentum coefficient matrices, while $D_m$ is a diagonal step-size matrix for momentum stochastic gradient method. In a manner similar to \eqref{ass: gfat}, we also assume that 
		\eq{
			\label{Lambda cond}
			0 \le B_k < I_M,\  k=1,2,\quad\quad B_1+B_2=B,\quad\quad B_1 B_2=0.
		}
		where $B=\mathrm{diag}\{\beta_1, \ldots, \beta_M \}$ and $0 < B < I_M$. In addition, we further assume that $B$ is not too close to $I_M$, i.e.
		\eq{\label{Lambda not close to 1}
			B \le (1-\epsilon) I_M,\ \mbox{for some constant $\epsilon > 0$}. 
		}
		
		\noindent The following results extend Theorems 1, 3, and 4 and they can be established following similar derivations.
		
		\vspace{1mm}
		\noindent \textbf{Theorem 1B (Mean-square stability).} \textit{Let Assumptions \ref{ass: cost function} and \ref{ass: noise} hold and recall conditions (\ref{Lambda cond}) and \eqref{Lambda not close to 1}. Then, for the momentum stochastic gradient method \eqref{sgfat-1-m}--\eqref{sgfat-2-m}, it holds under sufficiently small step-size $\mua$ that
			\eq{
				\limsup_{i \rightarrow \infty} \bE\|\tw_i\|^2 = O(\mua). \label{thm-conv-tw-diagonal}
			}
			\qd}
		
		%
		%\begin{theorem}[{\bf Mean-square stability}]\label{thm-conv-cw-m} Let Assumptions \ref{ass: cost function} and \ref{ass: noise} hold and recall conditions (\ref{Lambda cond}) and \eqref{Lambda not close to 1}. Then, for the momentum stochastic gradient method \eqref{sgfat-1-m}--\eqref{sgfat-2-m}, it holds under sufficiently small step-size $\mua$ that
		%	\eq{
		%		\limsup_{i \rightarrow \infty} \bE\|\tw_i\|^2 = O(\mua). \label{thm-conv-tw}
		%	}
		%	\qd
		%\end{theorem}  
		
		%In Section \ref{Sec: LMS} and \ref{sec: equi-ge}, we proved the equivalence between standard and momentum stochastic gradient method. With similar proof technique, we can extend conclusions in Theorem \ref{thm:gfat-stochastic gradient method-dist-hw} and Theorem \ref{thm:ge-dist} to the more generalized case. 
		
		\noindent \textbf{Theorem 3B (Equivalence for quadratic costs).} \textit{Consider recursions \eqref{eq-LMS} and \eqref{sgfat-1-lms}--\eqref{sgfat-2-lms} with $\{\mu,\mu_m,\beta_1,\beta_2\}$ replaced by $\{D,D_m,B_1,B_2\}$. Assume they start from the same initial states, namely, {$\bpsi_{-2}=\w_{-2}=\x_{-1}$}.
			Suppose further that conditions \eqref{Lambda cond} and \eqref{Lambda not close to 1} hold, and that the step-sizes matrices $\{D,D_m\}$ satisfy a relation similar to \eqref{cond.stepsize}, namely,
			\eq{\label{cond.stepsize-m}
				D=(I-B)^{-1} D_m.
			}
			Then, it holds under sufficiently small $\mua$, that
			\eq{
				\label{hb-stochastic gradient method-dist-tw-2-m}
				\bE\|\tw_i-\tx_i\|^2=O(\musa),\quad \forall i=0,1,2,3,\ldots
			}
			\qd}
		
		\noindent \textbf{Theorem 4B (Equivalence for general costs).} \textit{Consider the stochastic gradient recursion \eqref{stochastic gradient method-2-m} and the momentum stochastic gradient recursions \eqref{sgfat-1-m}--\eqref{sgfat-2-m} to solve the general problem \eqref{general prob}. Assume they start from the same
			initial states, namely, $\bpsi_{-2}=\w_{-2}=\x_{-1}$. Suppose conditions (\ref{Lambda cond}), (\ref{Lambda not close to 1}), and
			(\ref{cond.stepsize-m}) hold. Under Assumptions \ref{ass: cost function}, \ref{ass: noise-4}, \ref{ass:noise-lipschitz},
			and \ref{ass: hessian lipschitz}, and for sufficiently small step-sizes, it holds that
			\eq{
				\label{hb-stochastic gradient method-dist-hw}
				\bE\|\tw_i-\tx_i\|^2=O(\mu^{3/2}_{\max}),\quad \forall i=0,1,2,3,\ldots
			}
			Furthermore, in the limit,
			\eq{
				\label{hb-stochastic gradient method-dist-tw}
				\limsup_{i\rightarrow \infty} \bE\|\tw_i-\tx_i\|^2=O(\mu_{\max}^2).
			}
			\qd}
		
		\section{Experimental Results}
		\label{Sec: exp}
		In this section we illustrate the main conclusions by means of
		computer simulations for both cases of mean-square-error designs and logistic regression designs. We also run simulations for algorithm \eqref{d-sgfat-1}--\eqref{d-sgfat-2} and verify its advantages in the stochastic context.
		
		\subsection{Least Mean-Squares Error Designs}
		\label{subsec: ER-LMS}
		We apply the standard LMS algorithm to \eqref{lms model}. To do so, 
		we generate data according to the linear regression model \eqref{data model}, 
		where $w^o\in \real^{10}$ is chosen randomly, and $\u_i \in \real^{10}$ is i.i.d and follows
		$
		\u_i \sim \mathcal{N}(0,\Lambda)
		$ where $\Lambda \in \real^{10\times 10}$ is randomly-generated diagonal matrix with positive diagonal entries.
		Besides,
		$\v(i)$ is also i.i.d and follows
		$
		\v(i) \sim \mathcal{N}(0,\sigma_s^2I_{10}),
		$
		where $\sigma_s^2=0.01$. All results are averaged over 300 random trials. For each trial we generated $800$ samples of $\u_i$, $\v(i)$ and $\d(i)$.
		
		We first compare the standard and momentum LMS algorithms using $\mu=\mu_m=0.003$. {\color{black}The momentum parameter $\beta$ is set as $0.9$}. Furthermore, we employ the heavy-ball option for the momentum LMS, i.e., $\beta_1=0, \beta_2=\beta$. Both the standard and momentum LMS methods are illustrated in {\color{black}the left plot in} Fig.~\ref{fig:lms_conv} with blue and red curves,  respectively. It is seen that the momentum LMS converges faster, but the MSD performance is much worse. Next we set $\mu_m=\mu(1-\beta)=0.0003$ and illustrate this case with the magenta curve. It is observed that the magenta and blue curves are almost indistinguishable, which confirms the equivalence predicted by Theorem \ref{thm:gfat-stochastic gradient method-dist-hw} for all time instants. We also  illustrate an implementation with a decaying momentum parameter $\beta(i)$ by the green curve. In this simulation, we set $\mu_m=0.003$ and make $\beta(i)$ decrease in a stair-wise fashion: when $i\in [1,100]$, $\beta(i)=0.9$; when $i\in [101,200]$, $\beta(i)=0.9/(100^{0.3})$; $\ldots$; when $i\in [2401,2500]$, $\beta(i)=0.9/(2400^{0.3})$. With this decaying $\beta(i)$, it is seen that the momentum LMS method recovers its faster convergence rate and attains the same steady-state MSD performance as the LMS implementation. {\color{black}Finally, we also implemented the standard LMS with initial step-size $\mu = 0.003$ and then decrease it gradually according to $\mu_s(i) = \mu/[1-\beta(i)]$. As implied by Theorem \ref{thm:gfat-stochastic gradient method-dist-hw}, it is observed that the green and black curves are also almost indistinguishable, which confirms that the LMS algorithm with decaying momentum is still equivalent to the standard LMS with appropriately chosen decaying step-sizes.} {\color{black}We also compared the standard and momentum LMS algorithms when $\mu=\mu_m=0.003$ and  $\beta$ is set as $0.5, 0.6, 0.7, 0.8$, and the same performance as the left plot in Fig. \ref{fig:lms_conv} is observed. To save space, we show the right plot in  Fig. \ref{fig:lms_conv} in which $\beta=0.5$ and omit the figures when $\beta$ is set as $0.6, 0.7, 0.8$. }
		
		\begin{figure}[h]
			\centering
			\includegraphics[scale=0.4]{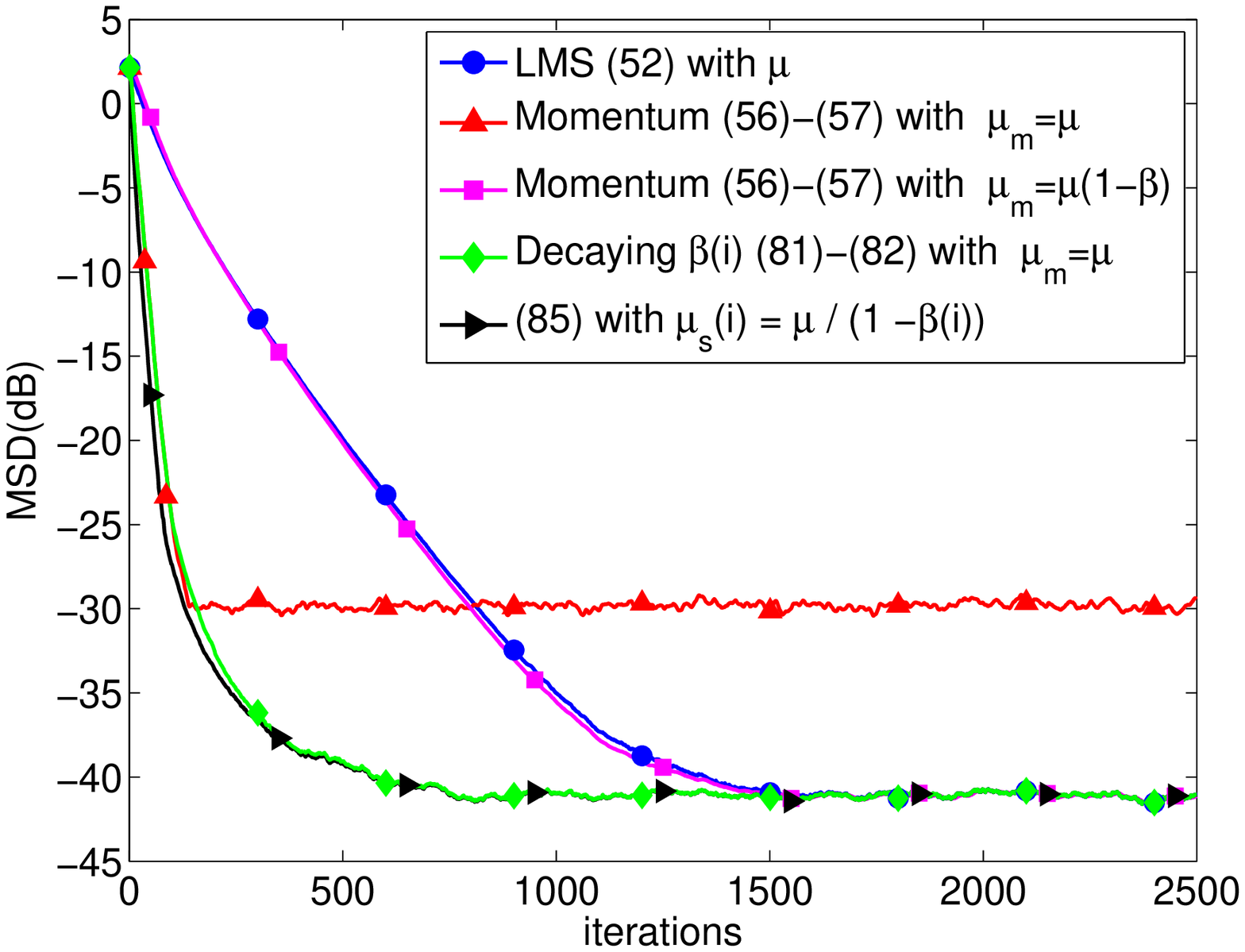}%{LMS_s_m_0303.eps}
			\includegraphics[scale=0.4]{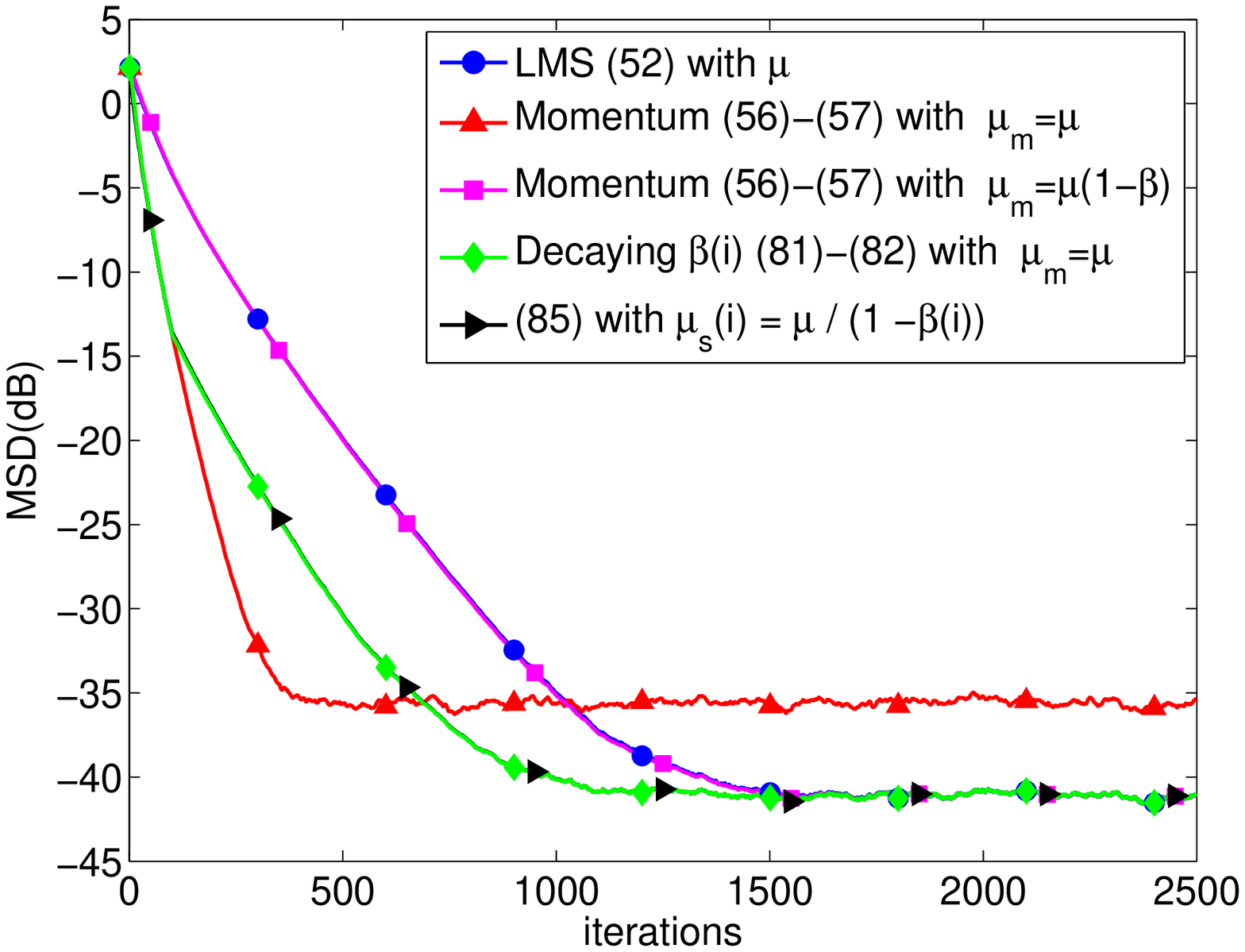}
			\vspace{-5mm}
			\caption{Convergence behavior of standard and momentum LMS ({\color{black}heavy-ball LMS}) algorithms applied to the mean-square-error design problem \eqref{lms model} with  {\color{black}$\beta=0.9$ in the left plot and $\beta=0.5$ in the right plot.} Mean-square-deviation (MSD) means $\bE\|w^o-\w_i\|^2$.}
			\label{fig:lms_conv}
		\end{figure}
		
		{\color{black}Next we employ the Nesterov's acceleration option for the momentum LMS method, and compare it with standard LMS. The experimental settings are exactly the same as the above except that $\beta_1=\beta$ and $\beta_2=0$. Both the standard and momentum LMS methods are illustrated in Fig. \ref{fig:lms_conv-2}. As implied by Theorem \ref{thm:gfat-stochastic gradient method-dist-hw}, it is observed that Nesterov's acceleration applied to LMS is equivalent to standard LMS with rescaled step-size. Besides, by comparing Figs. \ref{fig:lms_conv} and \ref{fig:lms_conv-2}, it is also observed that both momentum options, the heavy-ball and the Nesterov's acceleration, have the same performance. To save space, in the following experiments in Section \ref{subsec: ER-LR}--\ref{sec-cnn} we just show the performance of momentum method with the option of heavy-ball.
			
		\begin{figure}[h]
			\centering
			\includegraphics[scale=0.4]{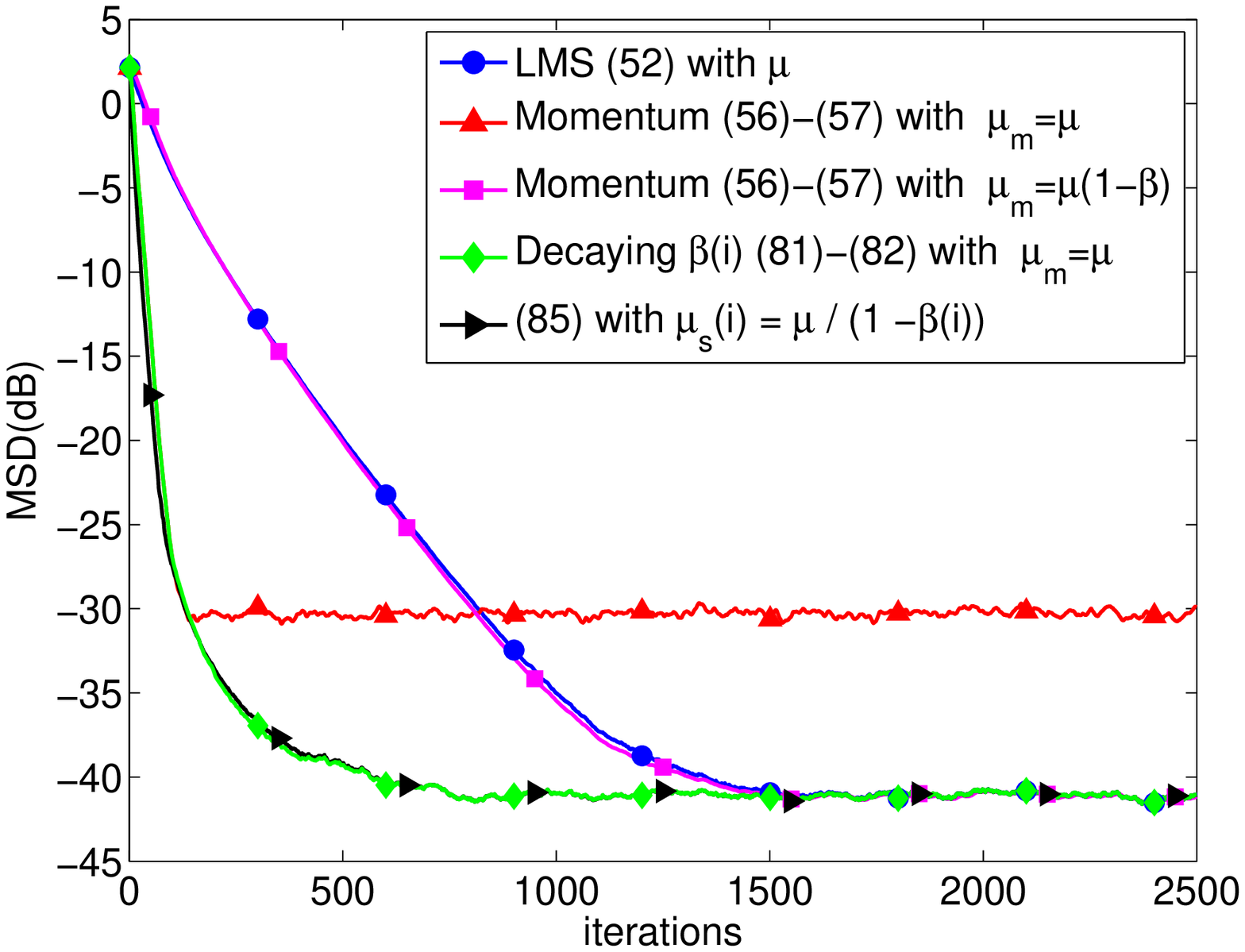}%{LMS_s_m_0303.eps}
			\includegraphics[scale=0.4]{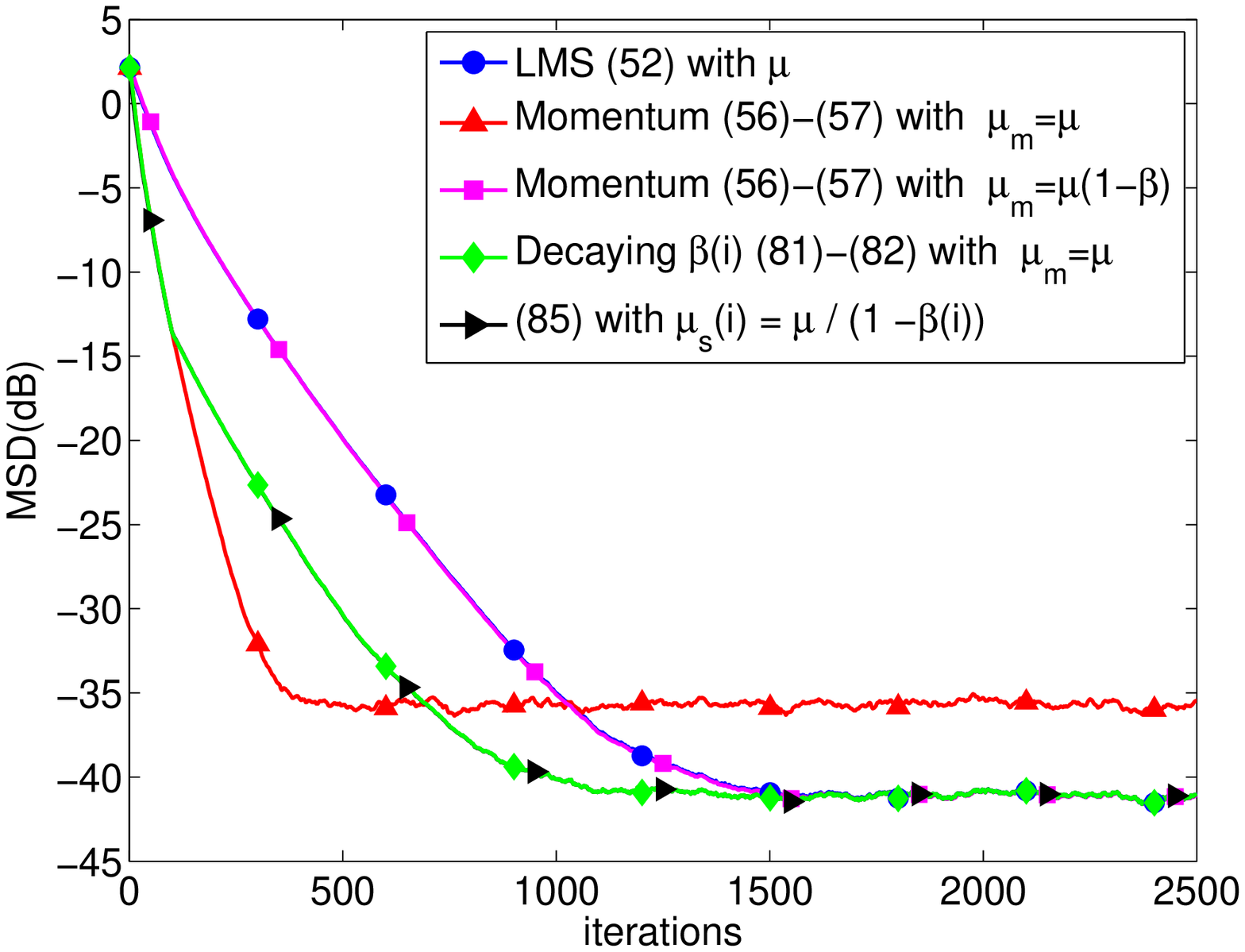}
			\vspace{-5mm}
			\caption{{\color{black}Convergence behavior of standard and momentum LMS (Nesterov's acceleration LMS) algorithms applied to the mean-square-error design problem \eqref{lms model} with $\beta=0.9$ in the left plot and $\beta=0.5$ in the right plot.}}
			\label{fig:lms_conv-2}
		\end{figure}
		} 
		
		\subsection{Regularized Logistic Regression}
		\label{subsec: ER-LR}
		We next consider a regularized logistic regression risk of the form:
		\eq{
			\label{app-prob-log-4}
			\ J(w)\define \frac{\rho}{2} \|w\|^2 + \bE \Big\{ \ln \big[ 1+ \exp(-\bgm(i) \h_i \tran w) \big] \Big\}
		}
		where the approximate gradient vector is chosen as
		\eq{
			{\grad_w  Q}(\w;\h_i,\bgm(i)) & = \rho \w-\frac{\exp(-\bgm(i) \h_i \tran \w)}{1+ \exp(-\bgm(i) \h_i \tran \w) } \bgm(i) \h_i
		}
		In the simulation, we generate $20000$ samples $(\h_i, \bm \gamma(i))$.
		Among these training points, $10000$ feature vectors $\h_i$ correspond to label $\bgm(i)=1$ and each $\h_i \sim \mathcal{N}(1.5\times \mathds{1}_{10},R_h)$ for some diagonal covariance $R_h$.
		The remaining $10000$ feature vectors $\h_i$ correspond to label $\bgm(i)=-1$ and each $\h_i \sim \mathcal{N}(-1.5\times \mathds{1}_{10},R_h)$. We set $\rho=0.1$. The optimal solution $w^o$ is computed via the classic gradient descent method. All simulation results shown below are averaged over $300$ trials.
		
		Similar to the least-mean-squares error problem, we first compare the standard and momentum stochastic methods using $\mu=\mu_m=0.005$. {\color{black}The momentum parameter $\beta$ is set to $0.9$}. These two methods are illustrated in Fig.~\ref{fig:lms_gap} with blue and red curves,  respectively. It is seen that the momentum method converges faster, but the MSD performance is much worse. Next we set $\mu_m=\mu(1-\beta)=0.0005$ and illustrate this case with the magenta curve. It is observed that the magenta and blue curves are indistinguishable, which confirms the equivalence predicted by Theorem \ref{thm:ge-dist} for all time instants. Again we illustrate an implementation with a decaying momentum parameter $\beta(i)$ by the green curve. In this simulation, we set $\mu_m=0.005$ and make $\beta(i)$ decrease in a stair-wise manner: when $i\in [1,200]$, $\beta(i)=0.9$; when $i\in [201,400]$, $\beta(i)=0.9/(200^{0.3})$; when $i\in [401,600]$, $\beta(i)=0.9/(400^{0.3})$; $\ldots$; when $i\in [1801,2000]$, $\beta(i)=0.9/(1800^{0.3})$. With this decaying $\beta(i)$, it is seen that the momentum method recovers its faster convergence rate and attains the same steady-state MSD performance as the stochastic-gradient implementation. {\color{black}Finally, we implemented the standard stochastic gradient descent with initial step-size $\mu_m = \mu = 0.005$ and then decrease it gradually according to $\mu_s(i) = \mu/[1-\beta(i)]$. As implied by Theorem \ref{thm:ge-dist}, it is observed that the green and black curves are almost indistinguishable, which confirms that the algorithm with decaying momentum is still equivalent to the standard stochastic gradient descent with appropriately chosen decaying step-sizes.}
		
		\begin{figure}[h]
			\centering
			\includegraphics[scale=0.5]{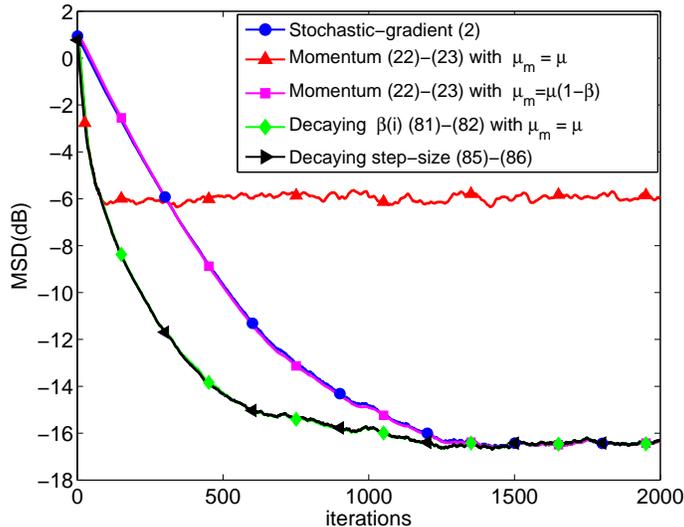}%{0917_lr_0303.eps}
			\vspace{-5mm}
			\caption{\small{Convergence behaviors of standard and momentum stochastic gradient methods applied to the logistic regression problem \eqref{app-prob-log-4}.}}
			\label{fig:lms_gap}
		\end{figure}
		
		{Next, we test the standard and momentum stochastic methods for regularized logistic regression problem over a benchmark data set --- the Adult Data Set\footnote{Source: \url{https://www.csie.ntu.edu.tw/~cjlin/libsvmtools/datasets/} or \url{http://archive.ics.uci.edu/ml/datasets/Adult}}. The aim of this dataset is to predict whether a person earns over $\$50$K a year based on census data such as age, workclass, education, race, etc. The set is divided into 6414 training data and 26147 test data, and each feature vector has 123 entries. In the simulation, we set $\mu=0.1$, $\rho=0.1$, and $\beta=0.9$. To check the equivalence of the algorithms, we set $\mu_m=(1-\beta)\mu=0.01$. 
			In Fig.~\ref{fig: lr_adult_data}, the curve shows how the accuracy performance, i.e., the percentage of correct prediction, over the test dataset evolved as the algorithm received more training data\footnote{To smooth the performance curve, we applied the weighted average technique from equation (74) of \citep{ying2015performance,ying2016performances}. 
				%				{\color{black}[Professor: this weighted average technique is used to smooth the convergence curve. Suppose $\w_0, \w_1,\cdots, \w_L$ is iterative variable of some algorithm. The weighted average technique will output $\bar{\w}_L=\frac{1}{S_L}(\alpha^L \w_0+\alpha^{L-1}\w_1 + \cdots + \alpha\w_{L-1} + \w_L),$ where $S_L=\sum_{j=0}^L \alpha^j$. ]} 
			}.
			The horizontal x-axis indicates the number of training data used.
			It is observed that the momentum and standard stochastic gradient methods cannot be distinguished, which confirms their equivalence when training the Adult Data Set.
			
			{\color{black}For the experiments shown in this section, Section \ref{sec:further-verification} and \ref{sec-cnn}, we also tested the cases when $\beta$ is set as $0.5$, $0.6$, $0.7$, $0.8$. Since the experimental results with different $\beta$ are similar, we just plot the situation when $\beta=0.9$, a setting which is usually employed in practice \citep{szegedy2014going,krizhevsky2012imagenet,zhang2015text}.}
			
			\begin{figure}[h]
				\centering
				\includegraphics[scale=0.5]{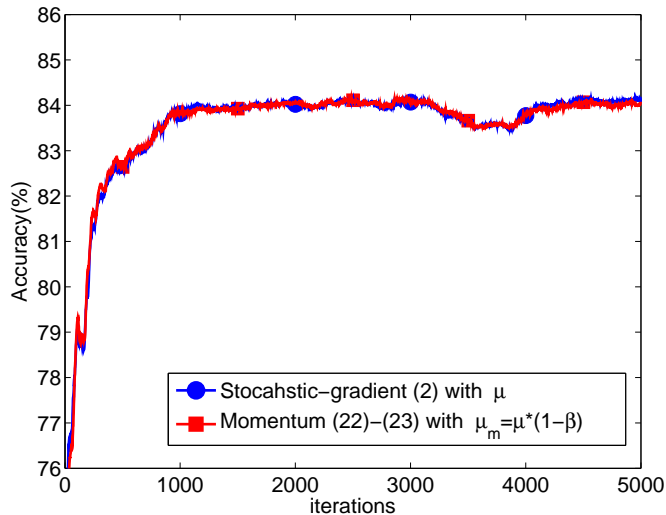}%{lr_real_data_0303.eps}
				\vspace{-4mm}
				\caption{\small{Performance accuracy of the standard and momentum stochastic gradient methods applied to logistic regression classification on the adult data test set.}}
				\label{fig: lr_adult_data}
			\end{figure}
		}
		
		\subsection{Further Verification of Theorems \ref{thm:gfat-stochastic gradient method-dist-hw} and  \ref{thm:ge-dist}}
		\label{sec:further-verification}
		
		In this section we further illustrate the conclusions of
		Theorems \ref{thm:gfat-stochastic gradient method-dist-hw} and \ref{thm:ge-dist} by checking the behavior of the iterate difference, i.e., $\bE\|\w_i-\x_i\|^2$, between
		the standard and momentum stochastic gradient methods.
		
		For the least-mean-squares error problem, the selection of $\u_i$, $\v(i)$, $\d(i)$ and $\beta$ is the same as in the simulation generated earlier in Subsection \ref{subsec: ER-LMS}. 
		For some specific step-size $\mu$, $\x_i$ is the iterate generated through LMS recursion \eqref{eq-LMS} with step-size $\mu$, and $\w_i$ is the iterate generated momentum LMS recursion \eqref{sgfat-1-lms}--\eqref{sgfat-2-lms} with step-size $\mu_m=\mu(1-\beta)$.
		Now we introduce the maximum difference:
		\eq{d_{\max}(\mu)=\max_i \bE\|\w_{i}-\x_{i}\|^2}
		and the difference at steady state 
		\eq{d_{\rm ss}(\mu)=\limsup_{i\rightarrow \infty} \bE\|\w_{i}-\x_{i}\|^2.}
		Note that both $d_{\max}(\mu)$ and $d_{\rm ss}(\mu)$ are related with $\mu$ and we will examine how they vary according to different step-sizes. Obviously, since $\bE\|\w_{i}-\x_{i}\|^2 \le d_{\max}(\mu)$, if $d_{\max}(\mu)$ is illustrated to be on the order of $O(\mu^2)$, then it follows that $\bE\|\w_i-\x_i\|^2=O(\mu^2)$ for $i\ge 0$. Similarly, if we can illustrate $d_{\rm ss}(\mu)=O(\mu^2)$, then it follows that $\limsup_{i\rightarrow \infty} \bE\|\w_{i}-\x_{i}\|^2=O(\mu^2)$.
		
		Note that the fact $d_{\max}(\mu)=c\mu^2$ for some constant $c$ holds if and only if
		\eq{\label{db relation}
			d_{\max}(\mu) (\mathrm{dB}) = 20 \log \mu + 10\log c,
		}
		where $d_{\max}(\mu) (\mathrm{dB})=10\log d_{\max}(\mu)$.	
		Relation \eqref{db relation} can be confirmed with red circle line in Fig.~\ref{fig:lms_gap-2}. In this simulation, we choose $8$ different step-size values $\{\mu_k\}_{k=1}^8$, and it can be verified that each data pair $\Big(\log \mu_k,\ d_{\max}(\mu_k)(\mathrm{dB}) \Big)$ satisfies relation \eqref{db relation}. For example, in the red circle solid line, at $\mu_1=10^{-2}$ we read $d_{\max}(\mu_1)({\rm dB})=-32\mathrm{dB}$; while at $\mu_2=10^{-4}$ we read $d_{\max}(\mu_2)({\rm dB})=-72\mathrm{dB}$. It can be verified that
		\eq{
			\label{gap btw mu1 and mu2 -1}
			d_{\max}(\mu_1) ({\rm dB}) - d_{\max}(\mu_2) ({\rm dB}) & = 20 (\log \mu_1 - \log \mu_2) = 40. 
		}
		Using a similar argument, the blue square solid line can also implies that $d_{\mathrm{ss}}=O(\mu^2)$.
		
		Figure \ref{fig:lms_gap-2} also reveals the order of $d_{\max}$ and $d_{\mathrm{ss}}$, with magenta and green dash lines respectively, for the regularized logistic regression problem from Subsection \ref{subsec: ER-LR}. With the same argument as above, $d_{\rm ss}(\mu)$ can be confirmed on the order of $O(\mu^2)$. Now we check the order of $d_{\max}(\mu)$. The fact that $d_{\max}(\mu)=c\mu^{3/2}$ holds if and only if
		\eq{\label{db relation-2}
			d_{\max}(\mu) (\mathrm{dB}) = 15 \log \mu + 10\log c.
		}
		According to the above relation, at $\mu_1=10^{-2}$ and $\mu_2=10^{-4}$ we should have
		\eq{
			\label{gap btw mu1 and mu2 -2}
			d_{\max}(\mu_1) ({\rm dB}) - d_{\max}(\mu_2) ({\rm dB}) & = 15 (\log \mu_1 - \log \mu_2) = 30. 
		}
		However, in the triangle magenta dash line we read $d_{\max}(\mu_1)=-30\mathrm{dB}$ while $d_{\max}(\mu_2)=-66\mathrm{dB}$ and hence
		$$ 30 {\rm dB} < d_{\max}(\mu_1) ({\rm dB}) - d_{\max}(\mu_2) ({\rm dB}) < 40 {\rm dB}$$
		Therefore, the order of $d_{\max}$ should be between $O(\mu^{3/2})$ and $O(\mu^2)$, which still confirms Theorem \ref{thm:ge-dist}. 
		\begin{figure}[h]
			\centering
			\includegraphics[scale=0.5]{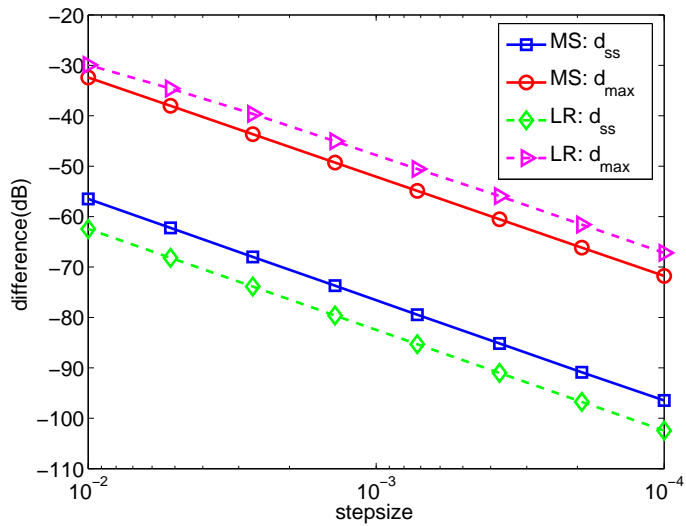}
			\vspace{-5mm}
			\caption{\small{$d_{\max}$ and $d_{\rm ss}$ as a function of the
					step-size $\mu$. \textit{MS} stands for mean-square-error and \textit{LR} stands for logistic regression.}
				%step-size value for the standard LMS and momentum LMS
				%recursions (\ref{eq-LMS}) and (\ref{sgfat-1-lms})--(\ref{sgfat-2-lms}) under condition \eqref{stepsize condition}.
			}
			\label{fig:lms_gap-2}
		\end{figure}
		
		{\color{black}\subsection{Visual Recognition}\label{sec-cnn}
		In this subsection we illustrate the conclusions of this work by re-examining the problem of training a neural network to recognize objects from images. We employ the CIFAR-10 database\footnote{\url{https://www.cs.toronto.edu/~kriz/cifar.html}}, which is a classical benchmark dataset of images for visual recognition. The CIFAR-10 dataset consists of 60000 color images in 10 classes, each with $32\times 32$ pixels. There are 50000 training images and 10000 test images. Similar to \citep{Sutskever2013}, and since the focus of this paper is on optimization, we only report training errors in our experiment. 
		
		To help illustrate that the conclusions also hold for non-differentiable and non-convex problems, in this experiment we train the data with two different neural network structures: (a) a 6-layer fully connected neural network and (b) a 4-layer convolutional neural network, both with ReLU activation functions. For each neural network, we will compare the performance of the momentum and standard stochastic gradient methods.
		
		\textbf{6-Layer Fully Connected Nuerual Network.} For this neural network structure, we employ the softmax measure with $\ell_2$ regularization as a cost objective, and the ReLU as an activation function. Each hidden layer has 100 units, the coefficient of the $\ell_2$ regularization term is set to $0.001$, and the initial value $w_{-1}$ is generated by a Gaussian distribution with $0.05$ standard deviation. We employ mini-batch stochastic-gradient learning with batch size equal to $100$. First, we apply a momentum backpropagation (i.e., momentum stochastic gradient) algorithm to train the 6-layer neural network. The momentum parameter is set to $\beta = 0.9$, and the initial step-size $\mu_m$ is set to $0.01$. To achieve better accuracy, we follow a common technique (e.g., \citep{szegedy2014going}) and reduce $\mu_m$ to $0.95\mu_m$ after every epoch.  With the above settings, we attain an accuracy of about $90\%$ in $80$ epochs. 
		
		However, what is interesting, and somewhat surprising, is that the same $90\%$ accuracy can also be achieved with the standard backpropagation (i.e., stochastic gradient descent) algorithm in $80$ epochs. According to the step-size relation $\mu=\mu_m/(1-\beta)$, we set the initial step-size $\mu$ of SGD to $0.1$. Similar to the momentum method, we also reduce $\mu$ to $0.95\mu$ after every epoch for SGD, and hence the relation $\mu=\mu_m/(1-\beta)$ still holds for each iteration. From Figure \ref{fig: digit_recog}, we observe that the accuracy performance curves for both scenarios, with and without momentum, are overlapping even when the overall risk is not necessarily convex or differentiable. 
%		This result indicates that the performance of momentum SGD can still be achieved by standard SGD by properly adjusting the step-size according to $\mu=\mu_m/(1-\beta)$. 
		
		\begin{figure}[h]
			\centering
			\includegraphics[scale=0.5]{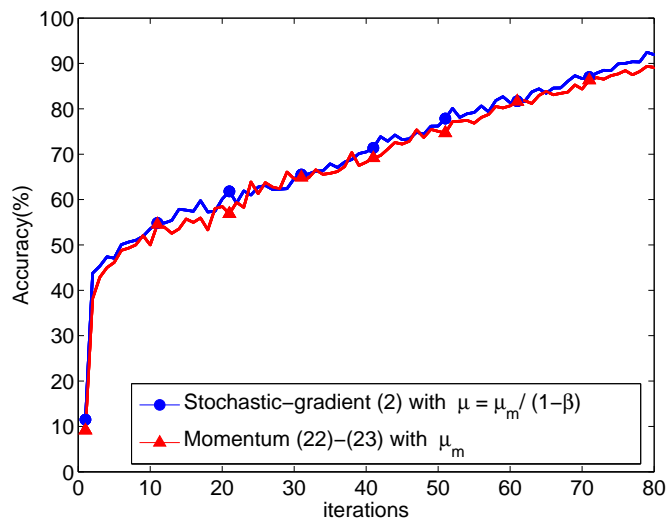}
			\vspace{-4mm}
			\caption{\small{{Classification accuracy of the standard and momentum stochastic gradient methods applied to a 6-layer fully-connected neural network on the CIFAR-10 test data set.}}}
			\label{fig: digit_recog}
		\end{figure}
		
		\textbf{4-Layer Convolutional Neural Network.} In a second experiment, we consider a 4-layer convolutional neural network. We employ the same objective and activation functions. This network has the structure:
		\eq{\nonumber
		\mbox{$\big($conv -- ReLU -- pool$\big)$ $\times 2$  -- $\big($ affine -- ReLU $\big)$ -- affine}
		}
		In the first convolutional layer, we use filters of size $7\times 7\times 3$, stride value $1$, zero padding $3$, and the number of these filters is $32$. In the second convolutional layer, we use filters of size $7\times 7 \times 32$, stride value $1$, zero padding $3$, and the number of filters is still $32$. We implement MAX operation in all pooling layers, and the pooling filters are of size $2\times 2$, stride value $2$ and zero padding $0$. The hidden layer has $500$ units. The coefficient of the $\ell_2$ regularization term is set to $0.001$, and the initial value $w_{-1}$ is generated by a Gaussian distribution with $0.001$ standard deviation. We employ mini-batch stochastic-gradient learning with batch size equal to $50$, and the step-size decreases by $5\%$ after each epoch.
		
		First, we apply the momentum backpropagation algorithm to train the neural network. The momentum parameter is set at $\beta = 0.9$, and we performed experiments with step-sizes $\mu_m \in \{0.01, 0.005, 0.001, 0.0005, 0.0001\}$ and find that $\mu_m = 0.001$ gives the highest training accuracy  after $10$ epochs. In Fig. \ref{fig: digit_recog-2} we draw the momentum stochastic gradient method with red curve when $\mu_m = 0.001$ and $\beta=0.9$. The curve reaches an accuracy of $94\%$. Next we set the step-size of the standard backpropagation $\mu = \mu_m/(1-\beta) = 0.01$, and illustrate its convergence performance with the blue curve. It is also observed that the two curves are indistinguishable. The numerical results shown in Figs. \ref{fig: digit_recog} and \ref{fig: digit_recog-2} imply that the performance of momentum SGD can still be achieved by standard SGD by properly adjusting the step-size according to $\mu=\mu_m/(1-\beta)$. 
%		
%		which implies that the performance of the momentum SGD in 4-layer convolutional neural network can also be achieved by the standard SGD.
		
		\begin{figure}[h]
			\centering
			\includegraphics[scale=0.5]{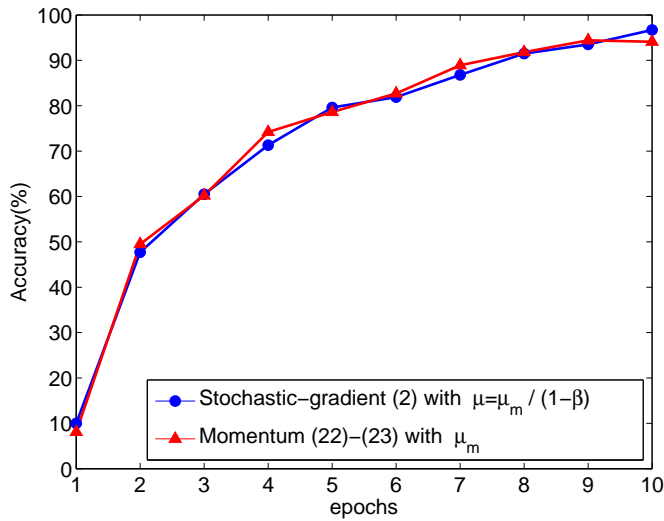}
			\vspace{-4mm}
			\caption{\small{{Classification accuracy of the standard and momentum stochastic gradient methods applied to a 4-layer convolutional neural network on the CIFAR-10 training data set.}}}
			\label{fig: digit_recog-2}
		\end{figure}

		}

		\section{Comparison for Larger Step-sizes}
		\label{sec:stable}
		According to Theorem \ref{thm:ge-dist}, the equivalence results between the standard and momentum stochastic gradient methods hold for sufficiently small step-sizes $\mu$. When larger values for $\mu$ are used, the $O(\mu^{3/2})$ term is not negligible any longer so that the momentum and gradient-descent implementations are not equivalent anymore under these conditions. While in practical implementations small step-sizes are widely employed in order to ensure satisfactory steady-state MSD performance, one may still wonder how both algorithms would compare to each other under larger step-sizes. For example, it is known that the larger the step-size value is, the more likely it is that the stochastic-gradient algorithm will become unstable. Does the addition of momentum help enlarge the stability range and allow for proper adaptation and learning over a wider range of step-sizes? 
		
		Unfortunately, the answer to the above question is generally negative. In fact, we can construct a simple numerical example in which the momentum can hurt the stability range. 
		This example considers the case of quadratic risks, namely problems of the form \eqref{lms model}. We suppose $M=5$, $\u_i\sim \mathcal{N}(0,0.5 I_5)$ and $\d(i)=\u_i\tran w^o + \v(i)$ where $\v(i)\sim \mathcal{N}(0,0.01)$. We compare the convergence of standard LMS and Nesterov's acceleration method with fixed parameter $\beta_2=0$ and $\beta=0.5$. Both algorithms are set with the same step-size $\mu=\mu_m=0.4$, which is a relatively large step-size. All results are averaged over $1000$ random trials. For each trial we generated $200$ samples of $\u_i$, $\v(i)$ and $\d(i)$. In Fig.~\ref{Fig:na-LMS-rho}, it shows that standard LMS converges at $\mu=0.4$ while momentum LMS diverges, which indicates that momentum LMS has narrower stability range than standard LMS. 	
		\begin{figure}[h]
			\centering
			\includegraphics[scale=0.5]{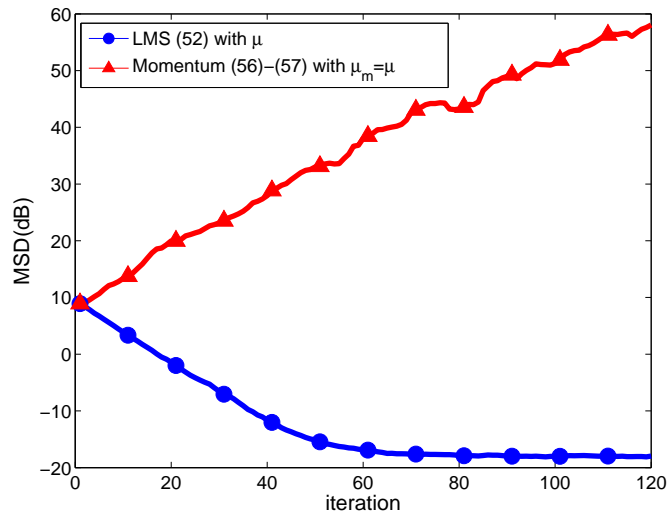}
			\vspace{-5mm} 
			\caption{Convergence comparison between standard and momentum LMS algorithms when $\mu=\mu_m=0.4$ and $\beta=0.5$.}
			\label{Fig:na-LMS-rho}
		\end{figure}

		\section{Conclusion}
		\label{Sec: conclu}
		In this paper we analyzed the convergence and performance behavior 
		of momentum stochastic gradient methods in the constant step-size and slow adaptation regime.
		The results establish that the momentum method is equivalent to
		employing the standard stochastic gradient method with a re-scaled (larger) step-size value.
		The size of the re-scaling is determined by the momentum parameter, $\beta$. The analysis was carried out
		under general conditions and was not limited to quadratic risks, but is also applicable to broader choices of the
		risk function. Overall, the conclusions indicate that the well-known benefits of momentum constructions in the
		deterministic optimization scenario do not necessarily carry over to the stochastic setting when
		adaptation becomes necessary and gradient noise is present. The analysis also comments on a way to retain some of the advantages of the
		momentum construction by employing a decaying momentum parameter: one that starts at a constant level and decays to zero over time. 
%		In this way, the enhanced convergence rate during the initial stages of 
		adaptation is retained without the often-observed degradation in MSD performance.
		
%			\section*{Acknowledgement}
			\acks{This work was supported in part by NSF grants CIF-1524250 and ECCS-1407712, by DARPA project N66001-14-2-4029, and by a Visiting Professorship from the Leverhulme Trust, United Kingdom.
			The authors would like to thank PhD student Chung-Kai Yu for contributing to Section~\ref{diminishing momentum}, \rev{and undergraduate student Gabrielle Robertson for contributing to the simulation in Section~\ref{sec-cnn}.}}  
			
%			\noindent {\color{red}[Professor: Since your two references NOW book and IEEE proceedings are in the same year, this latex template automatically distinguish them by using ``2014a'' and ``2014b''. I cannot remove the letter ``a'' and ``b''.]}

		\appendix
		
		\section{Proof of Lemma \ref{lm: x4 bound}}
		\label{app-lm-4-sta}
		It is shown in Eq.~(3.76) of \citep{sayed2014adaptation} that $\bE\|\tw_i\|^4$ evolves as follows:
		\eq{
			\label{x4-evolve}
			\bE\|\tw_i\|^4 \le (1-\mu \nu) \bE\|\tw_{i-1}\|^4 + a_1 \mu^2\bE \|\tw_{i-1}\|^2 + a_2 \mu^4,
		}
		{\color{black}
		where the constants $a_1$ and $a_2$ are defined as
		\eq{\label{9ewyu}
		a_1 \define 16\sigma_s^2,\quad a_2 \define 3\sigma_{s,4}^4.
		}
		If we iterate \eqref{tw-evolve-2} we find that
		\eq{
			\label{tw-evolve-2-5}
			\bE\|\tw_{i}\|^2 \le (1-\mu \nu)^{i+1} \bE \| \tw_{-1}\|^2 +  a_3 \mu,
		}
		where {$a_3$ is defined as
		\eq{\label{238anck9}
		a_3 \define \frac{\sigma_s^2}{\nu}.
		}
	}
		Substituting inequality \eqref{tw-evolve-2-5} into \eqref{x4-evolve}, we find that it holds for each iteration $i=0,1,2,\ldots$
		\eq{
		\bE\|\tw_i\|^4  \le & (1-\mu \nu) \bE\|\tw_{i-1}\|^4  + a_2 \mu^4 +a_1a_3 \mu^3 + a_4 \mu^2 (1-\mu \nu)^{i},\nnb
		=&\ \rho \bE\|\tw_{i-1}\|^4  + a_2 \mu^4 +a_1a_3 \mu^3 + a_4 \mu^2 \rho^i
		\label{x4-evolve-2}
	}
		where 
		\eq{\label{238adsh}
		\rho \define 1- \mu \nu ,\quad a_4 \define a_1\bE\|\tw_{-1}\|^2.
	} Iterating the inequality \eqref{x4-evolve-2} we get
		\eq{
			\label{x4-evolve-4}
			 \bE\|\tw_i\|^4 \le&\ \rho^{i+1} \bE \|\tw_{-1}\|^4 \hspace{-0.5mm}+\hspace{-0.5mm} a_2 \mu^4 \sum_{s=0}^i \rho^s \hspace{-0.5mm} +\hspace{-0.5mm} a_1a_3 \mu^3 \sum_{s=0}^i \rho^s \hspace{-0.5mm}+\hspace{-0.5mm} a_4\mu^2 (i+1) \rho^{i} \nnb
			\le&\ \rho^{i+1} \bE \|\tw_{-1}\|^4 \hspace{-0.5mm}+\hspace{-0.5mm} \frac{a_2 \mu^4}{1-\rho} + \frac{a_1a_3 \mu^3}{1-\rho} + a_4\mu^2 (i+1) \rho^{i} \nnb
			\le&\ \rho^{i+1} \bE \|\widetilde{\w}_{-1}\|^4 + a_5 \mu^3 + a_6 \mu^2  + a_4 \mu^2 (i+1) \rho^{i} \nnb
			\overset{\rm (a)}{\le}&\ \rho^{i+1} \bE \|\widetilde{\w}_{-1}\|^4 + 2 a_6 \mu^2  + a_4 \mu^2 (i+1) \rho^{i},
		}
		where 
		\eq{\label{07234ha}
		a_5 \define \frac{a_2}{\nu}, \quad a_6 \define \frac{a_1a_3}{\nu},
		}
		and (a) holds because for sufficiently small $\mu$ such that $a_6 \mu^2 > a_5 \mu^3$, we have 
		\eq{
		a_5 \mu^3 + a_6 \mu^2 = 2 a_6 \mu^2 - (a_6 \mu^2 - a_5 \mu^3) \le 2 a_6 \mu^2.
		}

		Substituting \eqref{9ewyu}, \eqref{238anck9}, \eqref{238adsh} and \eqref{07234ha} into \eqref{x4-evolve-4}, we get
		\eq{
			\label{x4-evolve-4-237abk}
			\bE\|\tw_i\|^4 & \le \rho^{i+1}\bE\|\tw_{-1}\|^4  + A_1 \sigma_s^2 (i+1)\rho^{i}\mu^2 + \frac{A_2 \sigma_s^4\mu^2}{\nu^2} \nnb
			&= \rho^{i+1}\bE\|\tw_{-1}\|^4  + \frac{A_1}{\rho} \sigma_s^2 (i+1)\rho^{i+1}\mu^2 + \frac{A_2 \sigma_s^4\mu^2}{\nu^2}
		}
		for some constants $A_1$ and $A_2$.  When $\mu$ is sufficiently small, there must exist some constant $A_3$ such that 
		\eq{\label{a8c7hak}
		\frac{A_1}{\rho} = \frac{A_1}{1-\mu\nu}  \le A_3.
		}
		Therefore, \eqref{x4-evolve-4-237abk} becomes 
		\eq{
			\label{x4-evolve-4-237abk-78mnm}
			\bE\|\tw_i\|^4 \le \rho^{i+1}\bE\|\tw_{-1}\|^4  + A_3 \sigma_s^2 (i+1)\rho^{i+1}\mu^2 + \frac{A_2 \sigma_s^4\mu^2}{\nu^2}.
		}
	}

		\section{Proof of Lemma \ref{lm:momentum transform}}
		\label{app:lm-gfat-transform}
		We substitute the expression for the gradient noise from (\ref{noise-2}), evaluated at $\bpsi_{i-1}$, into \eqref{sgfat-2} to get:
		\bqq
		\w_{i}=\ \bpsi_{i-1} - \mu_m \grad_w J(\bpsi_{i-1})  + \beta_2 (\bpsi_{i-1} - \bpsi_{i-2}) - \mu_m \s_i(\bpsi_{i-1}).\label{hb iteration s}
		\eqq
		Let again $\tw_{i}=w^o-\w_{i}$ and $\tpsi_{i}=w^o-\bpsi_i$. Subtracting both sides of \eqref{hb iteration s} from $w^o$ gives:
		\beqn
		\label{hb iteration s - 2}
		\tw_{i}= \tpsi_{i-1} + \mu_m \grad_w J(\bpsi_{i-1}) - \beta_2(\bpsi_{i-1} - \bpsi_{i-2}) + \mu_m \s_i(\bpsi_{i-1}).
		\eeqn
		We now appeal to the mean-value theorem (relation (D.9) in \citep{sayed2014adaptation}) to
		write
		\beqn
		\grad  J_w(\bpsi_{i-1}) = -\Big(\int_0^1 \grad^2_w J_w(w^o - t \tpsi_{i-1})dt \Big) \tpsi_{i-1} \define -\H_{i-1}\tpsi_{i-1} \label{grad J = -Hw}.
		\eeqn
		and express the momentum term in the form
		\beqn
		\label{momentum}
		\bpsi_{i-1} - \bpsi_{i-2} =\bpsi_{i-1} - w^o + w^o - \bpsi_{i-2} = - \tpsi_{i-1} + \tpsi_{i-2}.
		\eeqn
		Then, expression \eqref{hb iteration s - 2} can be rewritten as
		\begin{align}
		\label{tw update}
		\tw_{i}\hspace{-0.5mm}=\hspace{-0.5mm}(\hspace{-0.3mm} I_M \hspace{-0.3mm}+\hspace{-0.3mm} \beta_2 I_M \hspace{-0.3mm}- \hspace{-0.3mm} \mu_m \H_{i\hspace{-0.4mm}-\hspace{-0.4mm}1})\tpsi_{i\hspace{-0.3mm}-\hspace{-0.3mm}1} \hspace{-0.3mm}-\hspace{-0.3mm}  \beta_2 \tpsi_{i\hspace{-0.3mm}-\hspace{-0.3mm}2} \hspace{-0.5mm}+\hspace{-0.5mm} \mu_m \s_i(\bpsi_{i\hspace{-0.3mm}-\hspace{-0.3mm}1}).
		\end{align}
		On the other hand, expression \eqref{sgfat-1} gives
		\eq{
			\tpsi_{i-1}=\tw_{i-1}+\beta_1 (\tw_{i-1}-\tw_{i-2}). \label{sgfat-1-tilde}
		}
		Substituting \eqref{sgfat-1-tilde} into \eqref{tw update}, we have
		\eq{
			\label{sgfat-tw-update}
			\tw_i=\J_{i-1} \tw_{i-1} + \K_{i-1} \tw_{i-2} + L \tw_{i-3} + \mu_m \s_i(\bpsi_{i-1}),
		}
		where boldface quantities denote random variables:
		\eq{
		\J_{i-1}&=(1+\beta_1)(1+\beta_2)I_M - \mu_m (1+\beta_1) \H_{i-1}\stackrel{\footnotesize (\ref{ass: gfat})}{=}(1+\beta)I_M - \mu_m (1+\beta_1) \H_{i-1} \\
		\K_{i-1}&=-(\beta_1+\beta_2+2\beta_1\beta_2)I_M + \mu_m \beta_1 \H_{i-1}=-\beta I_M + \mu_m \beta_1 \H_{i-1}\\
		L&=\beta_1 \beta_2=0
		}
		It follows that we can write the extended relation:
		\begin{align}
		\label{sta-eq:lms-ac-recursion}
		\left[                 %???
		\begin{array}{c}   %?????3??????????
		\tw_{i}\\  %?????
		\tw_{i-1}
		\end{array}
		\right]                %???
		=&
		\underbrace{\left[
			\begin{array}{cc}
			\J_{i-1}& \K_{i-1} \\
			I_M & 0\\
			\end{array}
			\right]}_{\define \B_{i-1}}
		\left[
		\begin{array}{c}
		\tw_{i-1}\\
		\tw_{i-2}\\
		\end{array}
		\right]  +
		\mu_m \left[
		\begin{array}{c}
		\s_i(\bpsi_{i-1})\\
		0  \end{array}
		\right].
		\end{align}
		where we are denoting the coefficient matrix by $\B_{i-1}$, which can be written as the difference
		\beqn
		\B_{i-1}\define P -\M_{i-1},
		\eeqn
		with
		\beqn
		P =
		\left[
		\hspace{-1mm}
		\begin{array}{cc}
			(1+\beta )I_M& \hspace{-2mm}-\beta  I_M\\
			I_M & \hspace{-2mm} 0\\
		\end{array}
		\hspace{-1mm}
		\right],   \quad
		\M_{i-1}=
		\left[
		\hspace{-1mm}
		\begin{array}{cc}
			\mu_m (1+\beta_1)\H_{i-1} & \hspace{-1mm}-\mu_m \beta_1 \H_{i-1}\\
			0 & \hspace{-1mm}0\\
		\end{array}
		\hspace{-1mm}
		\right]\hspace{-1mm}. \label{M_i-1}
		\eeqn
		The eigenvalue decomposition of $P $ can be easily seen to be given by $P =V D V ^{-1}$,
		where
		\begin{align}
		V =
		\left[
		\begin{array}{cc}
		I_M & -\beta I_M\\
		I_M & -I_M\\
		\end{array}
		\right],
		\quad V ^{-1}=\frac{1}{1-\beta}
		\left[
		\begin{array}{cc}
		I_M & -\beta I_M\\
		I_M & -I_M\\
		\end{array}
		\right],\quad
		D =
		\left[
		\begin{array}{cc}
		I_M & 0\\
		0 & \beta  I_M\\
		\end{array}
		\right].\label{D}
		\end{align}
%		and
%		\beqn
%		D =
%		\left[
%		\begin{array}{cc}
%			I_M & 0\\
%			0 & \beta  I_M\\
%		\end{array}
%		\right].
%		\eeqn
		Therefore, we have
		\begin{align}
		\B_{i-1}&=V (D -V ^{-1}\M_{i-1} V )V ^{-1}=
		V  \left[
		\begin{array}{cc}
		I_M-\frac{\mu_m}{1-\beta}\H_{i-1} &\frac{\mu_m\beta^\prime }{1-\beta }\H_{i-1}\\
		-\frac{\mu_m}{1-\beta} \H_{i-1} & \beta  I_M+\frac{\mu_m\beta^\prime }{1-\beta }\H_{i-1}\\
		\end{array}
		\right]V ^{-1},
		\end{align}
		where
		\bqq \beta^\prime &\define&  \beta \beta_1 + \beta - \beta_1 = \beta \beta_1 + \beta_2.
		\eqq
		Multiplying both sides of \eqref{sta-eq:lms-ac-recursion} by $V^{-1}$ from the left and recalling definition
		(\ref{transform}), we obtain
		\begin{align}
		\label{app-sta-hb iteration 2}
		\left[                 %???
		\begin{array}{c}   %?????3??????????
		\widehat{\w}_{i}\\  %?????
		\cw_{i}\\  %?????
		\end{array}
		\right]                %???
		=&
		\left[
		\begin{array}{cc}
		I_M-\frac{\mu_m}{1-\beta }\H_{i-1} &\frac{\mu_m\beta^\prime }{1-\beta}\H_{i-1}\\
		-\frac{\mu_m}{1-\beta} \H_{i-1} & \beta  I_M+\frac{\mu_m \beta^\prime}{1-\beta }\H_{i-1}\\
		\end{array}
		\right]\left[
		\begin{array}{c}
		\hw_{i-1}\\
		\cw_{i-1}\\
		\end{array}
		\right]
		+
		\frac{\mu_m}{1-\beta} \left[
		\begin{array}{c}
		\s_i(\bpsi_{i-1})\\
		\s_i(\bpsi_{i-1})\\
		\end{array}
		\right].
		\end{align}

		\section{Proof of Theorem \ref{thm-conv-cw}}
		\label{app:thm-2}
		
		From the first row of recursion \eqref{sta-hb iteration 2} we have
		\bqq
		\hw_i&=&\left(I_M-\frac{\mu_m \H_{i-1} }{1-\beta}\right)\hw_{i-1}+\frac{\mu_m\beta^\prime \H_{i-1}}{1-\beta}\cw_{i-1} +\frac{\mu_m}{1-\beta}\s_i(\bpsi_{i-1}).
		\label{sta-hw update}
		\eqq
		Let $t\in(0,1)$. Squaring both sides and taking expectations conditioned on $\filt_{i-1}$, and using Jensen's inequality, we obtain under Assumptions~\ref{ass: cost function}  and \ref{ass: noise}:
		\eq{
		&\ \ \ \  \bE[\|\hw_i\|^2|\filt_{i-1}] \nnb
		&=  \left\|\left(I_M-\frac{\mu_m}{1-\beta}\H_{i-1}\right)\hw_{i-1}+\frac{\mu_m \beta^\prime}{1-\beta}\H_{i-1}\cw_{i-1}\right\|^2 +\frac{\mu_m^2}{(1-\beta)^2}\bE[\|\s_i(\bpsi_{i-1})\|^2|\filt_{i-1}]  \nnb
		&\overset{\text{(a)}}{\le} 
		\left\|(1\hspace{-0.5mm}-\hspace{-0.5mm}t)\frac{1}{1\hspace{-0.5mm}-\hspace{-0.5mm}t}\left(I_M\hspace{-0.5mm}-\hspace{-0.5mm}\frac{\mu_m}{1\hspace{-0.5mm}-\hspace{-0.5mm}\beta}\H_{i\hspace{-0.5mm}-\hspace{-0.5mm}1}\right)\hw_{i\hspace{-0.5mm}-\hspace{-0.5mm}1}\hspace{-0.5mm}+\hspace{-0.5mm} t\frac{1}{t}\frac{\mu_m \beta^\prime}{1\hspace{-0.5mm}-\hspace{-0.5mm}\beta}\H_{i\hspace{-0.5mm}-\hspace{-0.5mm}1}\cw_{i\hspace{-0.5mm}-\hspace{-0.5mm}1}\right\|^2 \hspace{-0.5mm}+\hspace{-0.5mm} \frac{\mu_m^2}{(1\hspace{-0.5mm}-\hspace{-0.5mm}\beta)^2}(\gamma^2 \|\tpsi_{i\hspace{-0.5mm}-\hspace{-0.5mm}1}\|^2+\sigma_s^2)\nnb
		&{\le} \frac{1}{1-t}\left\|\left(I_M-\frac{\mu_m}{1-\beta}\H_{i-1}\right)\hw_{i-1}\right\|^2 \hspace{-0.8mm}+\hspace{-0.8mm} \frac{1}{t}\left\|\frac{\mu_m \beta^\prime}{1-\beta}\H_{i-1}\cw_{i-1}\right\|^2 \hspace{-0.5mm}+\hspace{-0.5mm} \frac{\mu_m^2}{(1-\beta)^2}(\gamma^2 \|\tpsi_{i-1}\|^2+\sigma_s^2) \nnb
		&\overset{\text{(b)}}{\le}  \frac{1}{1-t}\left(1-\frac{\mu_m \nu}{1-\beta}\right)^2 \|\hw_{i-1}\|^2 + \frac{1}{t}\frac{\mu_m^2\beta^{\prime2} \delta^2}{(1-\beta)^2}\|\cw_{i-1}\|^2+ \frac{\mu_m^2}{(1-\beta)^2}(\gamma^2 \|\tpsi_{i-1}\|^2+\sigma_s^2)\nnb
		&\overset{\text{(c)}}{=}  \left(1-\frac{\mu_m \nu}{1-\beta}\right) \|\hw_{i-1}\|^2 + \frac{\mu_m \beta^{\prime 2} \delta^2}{\nu(1-\beta)}\|\cw_{i-1}\|^2+ \frac{\mu_m^2}{(1-\beta)^2}(\gamma^2 \|\tpsi_{i-1}\|^2+\sigma_s^2).
		\label{sta-hw update-expe-filt}
		}
		where (a) holds because of equation \eqref{noise-2-order} in Assumption \eqref{ass: noise}, (b) holds because $\nu I \le \H_{i-1} \le \delta I $ under Assumption \eqref{ass: cost function}, and (c) holds because we selected $t=\frac{\mu_m\nu}{1-\beta}$. Taking expectation again, we remove the conditioning to find:
		\eq{
			\label{app-sta-hw-final-bound-2}
			\bE\|\hw_i\|^2 \le & \left(1\hspace{-0.5mm}-\hspace{-0.5mm}\frac{\mu_m \nu}{1\hspace{-0.5mm}-\hspace{-0.5mm}\beta}\right) \bE\|\hw_{i-1}\|^2 \hspace{-0.5mm}+\hspace{-0.5mm} \frac{\mu_m \beta^{\prime 2} \delta^2}{\nu(1\hspace{-0.5mm}-\hspace{-0.5mm}\beta)}\bE\|\cw_{i\hspace{-0.5mm}-\hspace{-0.5mm}1}\|^2 \hspace{-0.5mm}+\hspace{-0.5mm} \frac{\mu_m^2}{(1\hspace{-0.5mm}-\hspace{-0.5mm}\beta)^2}(\gamma^2 \bE\|\tpsi_{i\hspace{-0.5mm}-\hspace{-0.5mm}1}\|^2\hspace{-0.5mm}+\hspace{-0.5mm}\sigma_s^2).
		}
		Furthermore, squaring \eqref{sgfat-1-tilde} and using the inequality $\|a+b\|^2\leq 2\|a\|^2+2\|b\|^2$ we get
		\eq{
			\label{bound tpsi}
			\|\tpsi_{i-1}\|^2 &\le 2(1+\beta_1)^2 \|\tw_{i-1}\|^2 + 2\beta_1^2\|\tw_{i-2}\|^2 \le 2(1+\beta_1)^2 (\|\tw_{i-1}\|^2+\|\tw_{i-2}\|^2) \nnb
			&= 2(1+\beta_1)^2 \left\|
			\left[                 %???
			\begin{array}{c}   %?????3??????????
				\tw_{i-1}\\  %?????
				\tw_{i-2}\\  %?????
			\end{array}
			\right]
			\right\|^2 = 2(1+\beta_1)^2\left\Vert
			V V^{-1}\left[                 %???
			\begin{array}{c}   %?????3??????????
				\tw_{i-1}\\  %?????
				\tw_{i-2}\\  %?????
			\end{array}
			\right]
			\right\Vert^2 \nnb
			& \le 2(1+\beta_1)^2 \|V\|^2
			\left\Vert \left[                 %???
			\begin{array}{c}   %?????3??????????
				\hw_{i-1}\\  %?????
				\cw_{i-1}\\  %?????
			\end{array}
			\right]
			\right\Vert^2. 
%			= 2(1+\beta_1)^2 v^2 (\|\hw_{i-1}\|^2 + \|\cw_{i-1}\|^2),
		}
		{\color{black}
		It is known that there exists some constant $d>0$ such that $\|V\|^2 \le d \|V\|_F^2 $. From expression \eqref{V and V_inv} for $V$ we have 
		$$ \|V\|_F^2 = 3 \|I_M\|^2_F + \beta^2 \|I_M\|^2_F \le 4 \|I_M\|^2_F = 4M. $$
		Let $v^2 \define 4 d M$, so that $\|V\|^2 \le v^2$.
	}
	Therefore, under expectation, we conclude that it also holds:
		\eq{
			\label{app-tpsi-bound}
			\bE\|\tpsi_{i-1}\|^2 \le 2(1+\beta_1)^2 v^2 (\bE\|\hw_{i-1}\|^2 + \bE\|\cw_{i-1}\|^2).
		}
		Substituting \eqref{app-tpsi-bound} into \eqref{app-sta-hw-final-bound-2}, we get
		\eq{
			\bE\|\hw_i\|^2 \leq&\ \Big(1-\frac{\mu_m \nu}{1-\beta} + \frac{ 2(1+\beta_1)^2 \gamma^2 v^2 }{(1-\beta)^2} \mu_m^2\Big) \bE\|\hw_{i-1}\|^2 + \frac{\mu_m^2\sigma_s^2}{(1-\beta)^2} \nnb
			& +  \Big( \frac{\mu_m \beta^{\prime 2} \delta^2}{\nu(1-\beta)} + \frac{ 2(1+\beta_1)^2\gamma^2 v^2}{(1-\beta)^2} \mu_m^2\Big)\bE\|\cw_{i-1}\|^2.
			\label{app-snag-ehwi bound}
		}
		Now, let us consider the second row of \eqref{sta-hb iteration 2}, namely,
		\eq{
			\label{second row}
			\cw_i=-\frac{\mu_m}{1-\beta} \H_{i-1} \hw_{i-1} & + \left(\beta I_M + \frac{\mu_m \beta^\prime}{1-\beta} \H_{i-1}\right)\cw_{i-1} + \frac{\mu_m}{1-\beta} \s_i(\bpsi_{i-1}).
		}
		As before, squaring and taking expectations of both sides, and using Jensen's inequality, we obtain under Assumptions~\ref{ass: cost function}  and \ref{ass: noise}:
		\bqq
		& & \hspace{-10mm} \bE\|\cw_i\|^2  \nnb
		&\leq&\ \bE\left\|\beta \cw_{i-1} + \left(\frac{\mu_m \beta^\prime \H_{i-1}}{1-\beta}\cw_{i-1} -\frac{\mu_m \H_{i-1}}{1-\beta} \hw_{i-1}\right) \right\|^2  + \frac{\mu_m^2}{(1-\beta)^2}(\gamma^2 \bE\|\tpsi_{i-1}\|^2 + \sigma_s^2)\nnb
		&\overset{\text{(a)}}{\le}& \beta \bE\|\cw_{i\hspace{-0.5mm}-\hspace{-0.5mm}1}\|^2 \hspace{-0.5mm}+\hspace{-0.5mm} \frac{1}{1\hspace{-0.5mm}-\hspace{-0.5mm}\beta}\bE\left\|\frac{\mu_m\beta^\prime \H_{i\hspace{-0.5mm}-\hspace{-0.5mm}1}}{1\hspace{-0.5mm}-\hspace{-0.5mm}\beta}\cw_{i\hspace{-0.5mm}-\hspace{-0.5mm}1}\hspace{-0.5mm}-\hspace{-0.5mm}\frac{\mu_m \H_{i\hspace{-0.5mm}-\hspace{-0.5mm}1}}{1\hspace{-0.5mm}-\hspace{-0.5mm}\beta} \hw_{i\hspace{-0.5mm}-\hspace{-0.5mm}1} \right\|^2 \hspace{-0.5mm}+\hspace{-0.5mm}  \frac{\mu_m^2}{(1\hspace{-0.5mm}-\hspace{-0.5mm}\beta)^2}(\gamma^2 \bE\|\tpsi_{i\hspace{-0.5mm}-\hspace{-0.5mm}1}\|^2 \hspace{-0.5mm}+\hspace{-0.5mm} \sigma_s^2)\nnb
		&\le&  \beta \bE\|\cw_{i\hspace{-0.5mm}-\hspace{-0.5mm}1}\|^2 \hspace{-0.5mm}+\hspace{-0.5mm} \frac{2\mu_m^2\beta^{\prime 2} \delta^2}{(1\hspace{-0.5mm}-\hspace{-0.5mm}\beta)^3} \bE\|\cw_{i\hspace{-0.5mm}-\hspace{-0.5mm}1}\|^2 \hspace{-0.5mm}+\hspace{-0.5mm} \frac{2\mu_m^2 \delta^2}{(1\hspace{-0.5mm}-\hspace{-0.5mm}\beta)^3} \bE\|\hw_{i\hspace{-0.5mm}-\hspace{-0.5mm}1}\|^2\hspace{-0.5mm}+\hspace{-0.5mm}    \frac{\mu_m^2}{(1\hspace{-0.5mm}-\hspace{-0.5mm}\beta)^2}(\gamma^2 \bE\|\tpsi_{i\hspace{-0.5mm}-\hspace{-0.5mm}1}\|^2 \hspace{-0.5mm}+\hspace{-0.5mm} \sigma_s^2)\nnb
		&=& \Big(\beta \hspace{-0.5mm}+\hspace{-0.5mm}  \frac{2\mu_m^2 \beta^{\prime 2} \delta^2}{(1\hspace{-0.5mm}-\hspace{-0.5mm}\beta)^3}  \Big)\bE\|\cw_{i\hspace{-0.5mm}-\hspace{-0.5mm}1}\|^2 \hspace{-0.5mm}+\hspace{-0.5mm} \frac{2\mu_m^2 \delta^2}{(1\hspace{-0.5mm}-\hspace{-0.5mm}\beta)^3} \bE\|\hw_{i\hspace{-0.5mm}-\hspace{-0.5mm}1}\|^2\hspace{-0.5mm}+\hspace{-0.5mm}  \frac{\mu_m^2}{(1\hspace{-0.5mm}-\hspace{-0.5mm}\beta)^2}(\gamma^2 \bE\|\tpsi_{i\hspace{-0.5mm}-\hspace{-0.5mm}1}\|^2 \hspace{-0.5mm}+\hspace{-0.5mm} \sigma_s^2),
		\label{sta-cwi-bound-1}
		\eqq
		where (a) holds since $\bE\|\beta \x + \y\|^2=\bE\|\beta \x + (1-\beta)\frac{1}{1-\beta}\y\|^2\le \beta \bE\|\x\|^2 + \frac{1}{1-\beta} \bE\|\y \|^2$.
		Substituting \eqref{bound tpsi} into (\ref{sta-cwi-bound-1}), it follows that:
		\eq{
			\label{sta-cwi-final-bound}
			\bE\|\cw_i\|^2 \le&  \Big(\beta +  \frac{2\mu_m^2 \beta^{\prime 2} \delta^2}{(1-\beta)^3} +
			\frac{2 \mu_m^2 \gamma^2 (1+\beta_1)^2v^2}{(1-\beta)^2}  \Big)\bE\|\cw_{i-1}\|^2  + \frac{\mu_m^2 \sigma_s^2}{(1-\beta)^2} \nnb
			& +  \Big( \frac{2\mu_m^2 \delta^2}{(1-\beta)^3} + \frac{2\mu_m^2 \gamma^2 (1+\beta_1)^2v^2}{(1-\beta)^2} \Big) \bE\|\hw_{i-1}\|^2 
		}
		Combining relations \eqref{app-snag-ehwi bound} and \eqref{sta-cwi-final-bound} leads to
		the desired result \eqref{thm: sta-hw-cw-bound}--\eqref{f u2}. Let us now examine the stability of the
		$2\times 2$ coefficient matrix:
		\eq{
			\Gamma\define  \left[                 %???
			\begin{array}{cc}   %?????3??????????
				a&b\\  %?????
				c&d\\  %?????
			\end{array}
			\right],
		}
		where
			\eq{\label{lkj;lkji}
		a &= 1-\frac{\mu_m \nu}{1-\beta} + \frac{ 2(1+\beta_1)^2 \gamma^2 v^2 }{(1-\beta)^2} \mu_m^2, \quad b = \frac{\mu_m \beta^{\prime 2} \delta^2}{\nu(1-\beta)} + \frac{ 2(1+\beta_1)^2\gamma^2 v^2}{(1-\beta)^2} \mu_m^2, \nnb
		c &= \frac{2\mu_m^2 \delta^2}{(1-\beta)^3} + \frac{2\mu_m^2 \gamma^2 (1+\beta_1)^2v^2}{(1-\beta)^2}, \quad \ d = \beta +  \frac{2\mu_m^2 \beta^{\prime 2} \delta^2}{(1-\beta)^3} +
		\frac{2 \mu_m^2 \gamma^2 (1+\beta_1)^2v^2}{(1-\beta)^2}.
	}
		When $\mu_m$ is sufficiently small, $a,b,c,d$ are all positive. 
		Since the spectral radius of a matrix is upper bounded by its $1$-norm, we have that
		\eq{
			\label{spectrum radius}
			\rho(\Gamma) \le \max\left\{a+c,\ b+d\right\}.
%			\rho(\Gamma) \le \max\left\{1-\frac{\mu_m \nu}{1-\beta}+O(\mu_m^2), \beta + O(\mu_m)\right\}.
		}
		From \eqref{lkj;lkji}, we further have
		\eq{\label{uy2389}
		a + c & \le 1-\frac{\mu_m \nu}{1-\beta} + \frac{ 2(1+\beta_1)^2 \gamma^2 v^2 }{(1-\beta)^2} \mu_m^2 + \frac{2\mu_m^2 \delta^2}{(1-\beta)^3} + \frac{2\mu_m^2 \gamma^2 (1+\beta_1)^2v^2}{(1-\beta)^2} \nnb
		& = 1-\frac{\mu_m \nu}{1-\beta} + \frac{4(1-\beta)(1+\beta_1)^2\gamma^2\nu^2 + 2\delta^2}{(1-\beta)^3} \mu_m^2 \nnb
		& \le 1-\frac{\mu_m \nu}{1-\beta} + \frac{16\gamma^2\nu^2 + 2\delta^2}{(1-\beta)^3} \mu_m^2,
		}
		where the last inequality holds because $1-\beta < 1$ and $1+\beta_1 < 2$. Similarly, we also have
		\eq{\label{weyu89}
		b + d \le \beta + \frac{\delta^2 \mu_m}{\nu(1-\beta)} + \frac{16\gamma^2\nu^2 + 2\delta^2}{(1-\beta)^3} \mu_m^2.
		}
		Combining \eqref{spectrum radius}--\eqref{weyu89}, we reach
		\eq{\label{kljyweu}
		\rho(\Gamma) \le \max\left\{1-\frac{\mu_m \nu}{1-\beta} + \frac{16\gamma^2\nu^2 + 2\delta^2}{(1-\beta)^3} \mu_m^2, \ \ \beta + \frac{\delta^2 \mu_m}{\nu(1-\beta)} + \frac{16\gamma^2\nu^2 + 2\delta^2}{(1-\beta)^3} \mu_m^2\right\}.
		}
		If the step-size $\mu_m$ is small enough to satisfy the following conditions
		\eq{\label{stability condition}
		\begin{cases}
			\begin{array}{l}
				\frac{\mu_m \nu}{2(1-\beta)}  >  \frac{16\gamma^2\nu^2 + 2\delta^2}{(1-\beta)^3} \mu_m^2, \\
				\frac{\delta^2 \mu_m}{\nu(1-\beta)} > \frac{16\gamma^2\nu^2 + 2\delta^2}{(1-\beta)^3} \mu_m^2, \\
				1 - \beta > \frac{2\delta^2 \mu_m}{\nu(1-\beta)},
			\end{array}
		\end{cases}
		}
		which is also equivalent to 
		\eq{\label{23789adhjj}
		\mu_m < \min\left\{ \frac{(1-\beta)^2 \nu}{32\gamma^2\nu^2 + 4\delta^2},\frac{(1-\beta)^2 \delta^2}{16\gamma^2\nu^3 + 2\delta^2\nu}, \frac{\nu(1-\beta)^2}{2\delta^2} \right\} = \frac{(1-\beta)^2 \nu}{32\gamma^2\nu^2 + 4\delta^2},
		}
		then it holds that 
		\eq{\label{uwehajcm}
		\rho(\Gamma) < \max\left\{1-\frac{\mu_m \nu}{2(1-\beta)}, \ \ \beta + \frac{2\delta^2 \mu_m}{\nu(1-\beta)}\right\} \le 1,
		}
		in which case $\Gamma$ will be a stable matrix. 
		
		{\color{black}
		When $\Gamma$ is stable, it then follows from \eqref{thm: sta-hw-cw-bound} that
			\eq{ \label{2487asdhta}
				%			\label{sta-hb-steady-state}
				\limsup_{i\rightarrow \infty}
				\left[  \hspace{-0.5mm}               %???
				\begin{array}{c}   %?????3??????????
					\bE\|\hw_{i}\|^2\\  %?????
					\bE\|\cw_{i}\|^2\\  %?????
				\end{array}
				\hspace{-0.5mm}\right] &\le (I_{2}-\Gamma)^{-1} \left[  \hspace{-0.5mm}               %???
				\begin{array}{c}   %?????3??????????
					e\\  %?????
					f\\  %?????
				\end{array}
				\hspace{-0.5mm}\right]. 
%				\left[\hspace{-0.5mm}
%				\begin{array}{cc}
%					O(1/\mu_m)&O(1)\\
%					O(\mu_m)&O(1)\\
%				\end{array}
%				\hspace{-0.5mm}\right]
%				\left[\hspace{-0.5mm}
%				\begin{array}{c}
%					O(\mu_m^2)\\
%					O(\mu_m^2)\\
%				\end{array}
%				\hspace{-0.5mm}\right] =\left[\hspace{-0.5mm}
%				\begin{array}{c}
%					O(\mu_m)\\
%					O(\mu_m^2)\\
%				\end{array}
%				\hspace{-0.5mm}\right], 
%				\nonumber
			}	
		Notice that 
		\eq{\label{2348ahsdg}
		(I_2 - \Gamma)^{-1} & = 
		\ba{cc}
		1-a & -b \\
		-c & 1-d
		\ea^{-1}  = 
		\frac{1}{(1-a)(1-d)-bc}
		\ba{cc}
		1 -d & b \\
		c & 1-a
		\ea \nnb
		& \overset{\eqref{lkj;lkji}}{=} \frac{1}{\mu_m \nu + p_1 \mu_m^2 + p_2 \mu_m^3 + p_3 \mu_m^4} 
		\ba{cc}
		1-\beta + p_4 \mu_m^2 & \frac{\mu_m \beta^{\prime2} \delta^2}{\nu(1-\beta)} + p_5 \mu_m^2 \\
		\frac{2\mu_m^2 \delta^2}{(1-\beta)^3} + \frac{2\mu_m^2 \gamma^2 (1+\beta_1)^2v^2}{(1-\beta)^2} & \frac{\mu_m \nu}{1-\beta} + p_6 \mu_m^2
		\ea
%		& \overset{(a)}{=}
%		\ba{cc}
%		\frac{1-\beta}{\mu_m \nu} & \frac{\beta^{\prime 2}\delta^2 }{\nu^2 (1-\beta)} \\
%		\frac{2\mu_m \delta^2}{(1-\beta)^3 \nu} + \frac{2\mu_m \gamma^2 (1+\beta_1)^2v^2}{(1-\beta)^2 \nu} & \frac{1}{1-\beta} 
%		\ea 
		}
		where 
		\eq{\label{pk}
		&p_1 \define -\frac{2(1+\beta_1)^2\gamma^2\nu^2}{1-\beta}<0, \quad p_4\define - \frac{2 \beta^{\prime 2} \delta^2}{(1-\beta)^3} -
		\frac{2 \gamma^2 (1+\beta_1)^2v^2}{(1-\beta)^2} < 0, \nnb
		&p_5 \define \frac{2(1+\beta_1)^2\gamma^2\nu^2}{(1-\beta)^2} > 0, \hspace{0.7cm} p_6 \define - \frac{2(1+\beta_1)^2\gamma^2\nu^2}{(1-\beta)^2}<0.
		}
		For simplicity, we omit the expression of $p_2$ and $p_3$ here. Notice that
		\eq{\label{23678}
		\mu_m \nu + p_1 \mu_m^2 + p_2 \mu_m^3 + p_3 \mu_m^4 = \frac{\mu_m \nu}{2} + \left( \frac{\mu_m \nu}{2} + p_1 \mu_m^2 + p_2 \mu_m^3 + p_3 \mu_m^4 \right).
		}
		Although $p_1<0$, it still holds that $\frac{\mu_m \nu}{2} + p_1 \mu_m^2 + p_2 \mu_m^3 + p_3 \mu_m^4 > 0$  when $\mu_m$ is sufficiently small, which implies that 
		\eq{\label{mc}
		\mu_m \nu + p_1 \mu_m^2 + p_2 \mu_m^3 + p_3 \mu_m^4 > \frac{\mu_m \nu}{2}.
		} 
		Similarly, it holds that
		\eq{\label{b3}
		\frac{\mu_m \beta^\prime \delta^2}{\nu(1-\beta)} + p_5 \mu_m^2 = \frac{2\mu_m \beta^\prime \delta^2}{\nu(1-\beta)} - \left( \frac{\mu_m \beta^\prime \delta^2}{\nu(1-\beta)} - p_5 \mu_m^2 \right) \le \frac{2\mu_m \beta^\prime \delta^2}{\nu(1-\beta)},
		}
		where the last inequality holds because $\frac{2\mu_m \beta^\prime \delta^2}{\nu(1-\beta)} - p_5 \mu_m^2 > 0$ for sufficiently small step-size. Furthermore, since $p_4<0$ and $p_6<0$, we also have
		\eq{\label{3b}
		1-\beta + p_4 \mu_m^2 < 1-\beta, \quad \frac{\mu_m \nu}{1-\beta} + p_6 \mu_m^2 < \frac{\mu_m \nu}{1-\beta}.
		}
		Substitute \eqref{mc}, \eqref{b3} and \eqref{3b} into \eqref{2348ahsdg}, we have
		\eq{\label{yqu}
		(I_2 - \Gamma)^{-1} \le 
		\ba{cc}
				\frac{2(1-\beta)}{\mu_m \nu} & \frac{4\beta^{\prime 2}\delta^2 }{\nu^2 (1-\beta)} \\
				\frac{4\mu_m \delta^2}{(1-\beta)^3 \nu} + \frac{4\mu_m \gamma^2 (1+\beta_1)^2v^2}{(1-\beta)^2 \nu} & \frac{2}{1-\beta} 
				\ea 
		}
		Combining \eqref{2487asdhta} and \eqref{yqu}, we have
		\eq{\label{8o2137ahlc}
			%			\label{sta-hb-steady-state}
			\limsup_{i\rightarrow \infty}
			\left[  \hspace{-0.5mm}               %???
			\begin{array}{c}   %?????3??????????
				\bE\|\hw_{i}\|^2\\  %?????
				\bE\|\cw_{i}\|^2\\  %?????
			\end{array}
			\hspace{-0.5mm}\right] &\le (I_{2}-\Gamma)^{-1} \left[  \hspace{-0.5mm}               %???
			\begin{array}{c}   %?????3??????????
				e\\  %?????
				f\\  %?????
			\end{array}
			\hspace{-0.5mm}\right] \nnb
			&\le \ba{cc}
			\frac{2(1-\beta)}{\mu_m \nu} & \frac{4\beta^{\prime 2}\delta^2 }{\nu^2 (1-\beta)} \\
			\frac{4\mu_m \delta^2}{(1-\beta)^3 \nu} + \frac{4\mu_m \gamma^2 (1+\beta_1)^2v^2}{(1-\beta)^2 \nu} & \frac{2}{1-\beta} 
			\ea 
			\ba{c}
			\frac{\mu_m^2 \sigma_s^2}{(1-\beta)^2} \\
			\frac{\mu_m^2 \sigma_s^2}{(1-\beta)^2}
			\ea \nnb
			&=
			\ba{c}
			\frac{2\mu_m \sigma_s^2}{(1-\beta) \nu} + \frac{4\beta^{\prime 2}\delta^2 \sigma_s^2 \mu_m^2}{(1-\beta)^3\nu^2} \\
			\frac{2\mu_m^2 \sigma_s^2}{(1-\beta)^3} + \frac{4\mu_m^3 \delta^2\sigma_s^2}{(1-\beta)^5 \nu} + \frac{4\mu_m^3 \gamma^2 (1+\beta_1)^2v^2 \sigma_s^2}{(1-\beta)^4 \nu}
			\ea \le
			\ba{c}
			\frac{3\mu_m \sigma_s^2}{(1-\beta)} \\
			\frac{3\mu_m^2 \sigma_s^2}{(1-\beta)^3}
			\ea
		}
		where in the last inequality we choose sufficiently small $\mu_m$ such that
		\eq{
			\frac{4\beta^{\prime 2}\delta^2 \sigma_s^2 \mu_m^2}{(1-\beta)^3\nu^2} < \frac{\mu_m \sigma_s^2}{(1-\beta) \nu}, \quad \frac{4\mu_m^3 \delta^2\sigma_s^2}{(1-\beta)^5 \nu} + \frac{4\mu_m^3 \gamma^2 (1+\beta_1)^2v^2 \sigma_s^2}{(1-\beta)^4 \nu} < \frac{\mu_m^2 \sigma_s^2}{(1-\beta)^3}
		}
		Therefore, we have the following result
		\eq{
			\label{sta-hb-steady-state-2}
			\limsup_{i\rightarrow \infty} \bE\|\hw_i\|^2 = O\left(\frac{\mu_m \sigma_s^2}{(1-\beta) \nu} \right), \quad
			\limsup_{i\rightarrow \infty} \bE\|\cw_i\|^2 = O\left(\frac{\mu_m^2 \sigma_s^2}{(1-\beta)^3} \right).
		}
		and
		\eq{\label{thm2-bound-tw}
			\limsup_{i\rightarrow \infty}
			\bE
			\left\Vert
			\left[
			\begin{array}{c}
				\tw_{i} \\
				\tw_{i-1}
			\end{array}
			\right]
			\right\Vert^2
			&=
			\limsup_{i\rightarrow \infty}
			\bE \left\Vert
			V \left[
			\begin{array}{c}
				\hw_{i} \\
				\cw_{i}
			\end{array}
			\right]
			\right\Vert^2 \nnb
			&\le
			v^2 \left(\hspace{-1mm} \limsup_{i\rightarrow \infty} \bE
			\left\Vert
			\left[
			\begin{array}{c}
				\hw_{i} \\
				\cw_{i}
			\end{array}
			\right]
			\right\Vert^2 \hspace{-0.5mm} \right) \nnb
			&= v^2\left(\hspace{-1mm} \limsup_{i\rightarrow \infty}( \bE \|\hw_i\|^2 \hspace{-1mm} + \hspace{-0.5mm} \bE \|\cw_i\|^2 ) \hspace{-1mm} \right) = O\left(\frac{\mu_m \sigma_s^2}{(1-\beta) \nu} \right),
		}
		from which we conclude that (\ref{thm-conv-tw}) holds.
		
	}
		
%		
%		\eq{
%%			\label{sta-hb-steady-state}
%			\limsup_{i\rightarrow \infty}
%			\left[  \hspace{-0.5mm}               %???
%			\begin{array}{c}   %?????3??????????
%				\bE\|\hw_{i}\|^2\\  %?????
%				\bE\|\cw_{i}\|^2\\  %?????
%			\end{array}
%			\hspace{-0.5mm}\right] &\le (I_{2}-\Gamma)^{-1} \left[  \hspace{-0.5mm}               %???
%			\begin{array}{c}   %?????3??????????
%				e\\  %?????
%				f\\  %?????
%			\end{array}
%			\hspace{-0.5mm}\right] =
%			\left[\hspace{-0.5mm}
%			\begin{array}{cc}
%				O(1/\mu_m)&O(1)\\
%				O(\mu_m)&O(1)\\
%			\end{array}
%			\hspace{-0.5mm}\right]
%			\left[\hspace{-0.5mm}
%			\begin{array}{c}
%				O(\mu_m^2)\\
%				O(\mu_m^2)\\
%			\end{array}
%			\hspace{-0.5mm}\right] =\left[\hspace{-0.5mm}
%			\begin{array}{c}
%				O(\mu_m)\\
%				O(\mu_m^2)\\
%			\end{array}
%			\hspace{-0.5mm}\right], \nonumber
%		}
%		so that

		\section{Proof of Corollary \ref{co:hb-cw-u2}}
		\label{app-cor-2-ub}
		To simplify the notation, we refer to \eqref{thm: sta-hw-cw-bound} and introduce the
		quantities:
		\eq{
			z_i=
			\left[                 %???
			\begin{array}{c}   %?????3??????????
				\bE\|\hw_{i}\|^2\\  %?????
				\bE\|\cw_{i}\|^2\\  %?????
			\end{array}
			\right], \
			\Gamma=
			\left[                 %???
			\begin{array}{cc}   %?????3??????????
				a&b\\  %?????
				c&d\\  %?????
			\end{array}
			\right], \
			r=
			\left[                 %???
			\begin{array}{c}   %?????3??????????
				e\\  %?????
				f\\  %?????
			\end{array}
			\right].
		}
		Then, relation \eqref{thm: sta-hw-cw-bound} can be rewritten as
		\eq{
			z_i\preceq \Gamma z_{i-1}+r.
		}
		It follows that, in terms of the $1-$norm,
		\eq{
			\label{zi inequality}
			\|z_i\|_1 \le \|\Gamma\|_1 \|z_{i-1}\|_1 + \|r\|_1,
		}
		where
		\eq{
			\|\Gamma\|_1=\max\left\{1-\frac{\mu_m \nu}{1-\beta}+B_1 \mu_m^2, \beta + B_2 \mu_m\right\}
		}
		for some constant $B_1$ and $B_2$. Now we can choose $\mu_m$ sufficiently small to satisfy
		\eq{
			 B_1 \mu_m^2 < \frac{ \nu \mu_m}{2(1-\beta)},\quad \quad  \left(B_2+\frac{\nu}{2(1-\beta)}\right)\mu_m < 1-\beta,
		}
		which implies that
		\eq{
			\|\Gamma\|_1 & \le 1-\frac{\mu_m \nu}{1-\beta}+B_1 \mu_m^2  \le 1-\frac{\mu_m \nu}{2(1-\beta)} \define \rho_1 <1 \label{23hasts}
		}
		%%which imply that
		%%\eq{
		%%1-\frac{\mu_m \nu}{1-\beta}+c_1'\mu_m^2 &<& 1-\frac{\mu_m \nu}{2(1-\beta)}\\
		%%\beta+c_2'\mu_m &<&   1-\frac{\mu_m \nu}{2(1-\beta)}
		%%}
		%
		%Recall condition \eqref{ass:beta}, there must exist some $\epsilon > 0$ such that $\beta < 1-\epsilon$. When $\mu_m$ is sufficiently small, 
		%the following inequalities hold:
		%\eq{
		%& c^\prime_1 \mu_m^2 - \frac{\mu_m\nu}{2(1-\beta)} < 0, \nnb
		%&1-\epsilon + c^\prime_2 \mu_m <\  1-\frac{\mu_m\nu}{2(1-\beta)}. 
		%}
		%Therefore, it holds that
		%\eq{
		%\|\Gamma\|_1 \le 1-\frac{\mu_m\nu}{2(1-\beta)} < 1.
		%}
		%Now denote $\rho_1=1-\frac{\mu_m \nu}{2(1-\beta)}$. 
		Then, from \eqref{zi inequality} we have
		\eq{
			\label{zi inequality-2}
			\|z_i\|_1 \le \rho_1 \|z_{i-1}\|_1 + \|r\|_1.
		}
		Iterating \eqref{zi inequality-2} gives
		{\color{black}
			\eq{
				\label{zi inequality-3}
				\|z_i\|_1 \le \rho_1^{i+1} \|z_{-1}\|_1 + \frac{\|r\|_1}{1-\rho_1}.
			}
			Recall the expressions of $e$ and $f$ from \eqref{f u2}, we have $\|r\|_1 \le \frac{B_3 \mu_m^2 \sigma_s^2}{(1-\beta)^2}$ for some constant $B_3$. Since $1-\rho_1= \frac{\mu_m \nu}{2(1-\beta)}$, we get $\|r\|_1/(1-\rho_1) \le \frac{2 B_3 \mu_m \sigma_s^2 }{(1-\beta)\nu}$. From \eqref{zi inequality-3}, we have
			\eq{
				\label{zi inequality-3asdht}
				\|z_i\|_1 \le \rho_1^{i+1} \|z_{-1}\|_1 +  \frac{2 B_3 \mu_m \sigma_s^2 }{(1-\beta)\nu}.
			}
			%		Since $\|z_{-1}\|_1$ is a constant, and $\mu_m$ is sufficiently small, we conclude that $\|z_i\|_1 \le D$ for $i=0,1,2,\cdots$, where $D$ is some constant. 
			Accordingly, using
			\be \|z_i\|_1 =  \bE\|\hw_i\|^2 + \bE\|\cw_i\|^2 \ee
			we also find that
			\eq{\label{2348adsgh}
				\bE\|\hw_i\|^2 \le \rho_1^{i+1} \|z_{-1}\|_1 +  \frac{2 B_3 \mu_m \sigma_s^2 }{(1-\beta)\nu}.
%				\bE\|\cw_i\|^2 \le \rho_1^{i+1} \|z_{-1}\|_1 +  O\left(\frac{2 \mu_m \sigma_s^2 }{(1-\beta)\nu}\right). \label{23adsht7}
			}
			On the other hand,
			we know from the second row of (\ref{thm: sta-hw-cw-bound}) that
			%		, or from inequality \eqref{sta-cwi-final-bound} that
			\eq{
%				\label{rl-cw-bound-0}
%				\bE\|\hw_i\|^2 &\le  \big( 1 - \frac{\mu_m \nu}{1-\beta} +  b_1 \mu_m^2  \big)\bE\|\hw_{i-1}\|^2  +  \left( \frac{\mu_m \delta^2}{\nu(1-\beta)} + b_2 \mu_m^2 \right) \bE\|\cw_{i-1}\|^2 + b_3 \mu_m^2 \\
				\label{rl-cw-bound}
				\bE\|\cw_i\|^2 \le  \big(\beta +  c_1 \mu_m^2  \big)\bE\|\cw_{i-1}\|^2  +  c_2 \mu_m^2 \bE\|\hw_{i-1}\|^2 + c_3 \mu_m^2
			}
		}{\color{black}for constants
%			$$b_1 = \frac{2 (1+\beta_1)^2\gamma^2 v^2}{(1-\beta)^2}, \quad b_2 = \frac{2 (1+\beta_1)^2\gamma^2 v^2}{(1-\beta)^2}, \quad b_3 = \frac{\sigma_s^2}{(1-\beta)^2}, $$
			\eq{\label{238a000}
			c_1 \define \frac{2 \beta^{\prime 2} \delta^2 }{(1-\beta)^3} + \frac{2\gamma^2(1+\beta_1)^2 v^2}{(1-\beta)^2}, \quad c_2 \define \frac{2\delta^2}{(1-\beta)^3} + \frac{2\gamma^2(1+
				\beta)^2 v^2}{(1-\beta)^2}, \quad c_3 \define \frac{\sigma_s^2}{(1-\beta)^2}. 
			}
			To simplify the notation, with the facts that $\beta^{\prime} < 1, \beta<1$ and $\beta_1 < 1$, we have
			\eq{\label{2389adshj}
			c_1 \le \frac{B_4 (\delta^2+\gamma^2)}{(1-\beta)^3}\define c_4, \quad c_2 \le \frac{B_4(\delta^2 + \gamma^2)}{(1-\beta)^3} = c_4
			}
			for some constant $B_4$. Substituting \eqref{2389adshj} into \eqref{rl-cw-bound} we get
			\eq{
				\label{rl-cw-bound-2}
				\bE\|\cw_i\|^2 \le  \big(\beta +  c_4 \mu_m^2  \big)\bE\|\cw_{i-1}\|^2  +  c_4 \mu_m^2 \bE\|\hw_{i-1}\|^2 + c_3 \mu_m^2.
			}
			Now we substitute \eqref{2348adsgh} into \eqref{rl-cw-bound-2}, and reach
			\eq{
				\label{rl-cw-bound-asdluh}
				\bE\|\cw_i\|^2 \le  \big(\beta +  c_4 \mu_m^2  \big)\bE\|\cw_{i-1}\|^2  +  c_4 \rho_1^{i}\|z_{-1}\|_1 \mu_m^2 + \frac{2 B_3 c_4\sigma_s^2}{(1-\beta)\nu}\mu_m^3 + c_3 \mu_m^2.
			}
			When $\mu_m$ is sufficiently small such that
			\eq{\label{238dsahg5} \frac{2 B_3 c_4\sigma_s^2}{(1-\beta)\nu}\mu_m^3 \le   c_3 \mu_m^2,
			}
			\eqref{rl-cw-bound-asdluh} becomes 
			\eq{
				\label{rl-cw-bound-asdluh2}
				\bE\|\cw_i\|^2 \le  \big(\beta +  c_4 \mu_m^2  \big)\bE\|\cw_{i-1}\|^2  +  c_4 \rho_1^{i}\|z_{-1}\|_1 \mu_m^2  + 2 c_3 \mu_m^2.
			}
		}{Notice that 
		\bqq
		\beta + c_4 \mu_m^2 = 1 - (1 - \beta) + c_4 \mu_m^2 = 1-\frac{1-\beta}{2} + \left(c_1 \mu_m^2-\frac{1-\beta}{2}\right). \label{rho epsilon}
		\eqq
		It is clear that we can choose a sufficiently small $\mu_m$ for the last term between brackets to become negative, in which case
		\be
		\beta + c_4 \mu_m^2  \leq 1-\frac{1-\beta}{2} = \frac{1+\beta}{2} \define \alpha < 1
		\ee
		It follows that
		{\color{black}
		\eq{
			\bE\|\cw_i\|^2
			& \le \ \alpha\: \bE\|\cw_{i-1}\|^2  +  \left(c_4 \|z_{-1}\|_1 \rho_1^{i}\right) \mu_m^2 + 2 c_3 \mu_m^2 \nnb
			& \le \ \alpha^{i+1} \bE\|\cw_{-1}\|^2  +  c_4 \|z_{-1}\|_1 \mu_m^2 \rho_1^{i} \sum_{s=0}^{i}\left(\frac{\alpha}{\rho_1}\right)^s + \frac{2c_3 \mu_m^2}{1-\alpha}.  \label{rl-cw-bound-1}
		}
		Recall that $\rho_1 = 1 - \frac{\mu_m \nu}{2(1-\beta)}$ and $\alpha = 1 - (1-\beta)/2$. Therefore, it holds that $\alpha/\rho_1 < 1$ for sufficiently small $\mu_m$. As a result, we have
		\eq{\label{238asdg}
			\sum_{s=0}^{i}\left(\frac{\alpha}{\rho_1}\right)^s \le \frac{1}{1- \frac{\alpha}{\rho_1}} = \frac{\rho_1}{\rho_1 - \alpha} = \frac{2(1-\beta) - \mu_m \nu}{(1-\beta)^2 - \mu_m \nu} \le \frac{B_5}{1-\beta}
		}
		for some constant $B_5$ when $\mu_m$ is sufficiently small.
		Substituting \eqref{238asdg} into \eqref{rl-cw-bound-1}, we get 
		\eq{
			\bE\|\cw_i\|^2
			& \le \ \alpha^{i+1} \bE\|\cw_{-1}\|^2  +  \frac{B_5 c_4 \|z_{-1}\|_1  \rho_1^{i}}{1-\beta} \mu_m^2  + \frac{4 c_3 \mu_m^2}{1-\beta}.  \label{rl-cw-bound-1-asduy}
		}
	}}To assess the term that depends on the initial state, $\bE\|\cw_{-1}\|^2$, let us consider
		the boundary conditions \eqref{sgfat-3}--\eqref{sgfat-4}, 
		%
		%% the initial situation \eqref{sgfat-4}
		%%the initial states are assumed to satisfy 
		%%$\w_{-1}=\w_{-2}-\mu$
		%%in the momentum recursion
		%%\eqref{sgfat-1}--\eqref{sgfat-2}. Then,  from \eqref{transform}
		%%\be \cw_{-1}=\tw_{-1}-\tw_{-2}=\w_{-2}-\w_{-1}=0\ee
		%%Substituting this fact into \eqref{rl-cw-bound-4}, we reach
		%%\eq{
		%%\label{co-rl-cw-bound-5}
		%%\bE\|\cw_i\|^2 \le  \frac{c_4 \mu^2 }{1-\rho}=O(\mu^2), \quad \forall i=0,1,2,\ldots
		%%}
		%%More generally,  the requirement
		%%of $\w_{-1}=\w_{-2}$ is not necessary. In practice, there is another common way to set initial states for the
		%%  momentum method, 
		%  in which $\w_{-2}$ is any arbitrary initial state and
		%\be \w_{-1}=\w_{-2}-\mu_m \grad_w Q(\w_{-2};\btheta_{-1}).\label{laklkdj7183.adlk}\ee
		%where $\grad_w Q(\w_{-2};\btheta_{-1})$ is clearly independent of $\mu_m$. 
		Then, from \eqref{transform} it holds that
		\eq{ \cw_{-1}&=\frac{\tw_{-1}-\tw_{-2}}{1-\beta}=\frac{\w_{-2}-\w_{-1}}{1-\beta} = \frac{\mu_m \grad_w Q(\w_{-2};\btheta_{-1})}{1-\beta}\label{ini-state-142} }
		so that $\bE\|\cw_{-1}\|^2 = c_5 \mu_m^2$, where 
		\eq{c_5\define \bE\|\grad_w Q(\w_{-2};\btheta_{-1})\|^2/(1-\beta)^2.}
		{\color{black}
		Substituting
		this conclusion into \eqref{rl-cw-bound-1-asduy}, and recalling the expression of $c_3$, $c_4$ and $c_5$, we arrive at 
		\eq{
		\bE\|\cw_i\|^2
		& \le \ \frac{B_6 \alpha^{i+1} \mu_m^2}{(1-\beta)^2}  +  \frac{B_7 (\delta^2+\gamma^2)  \rho_1^{i}}{(1-\beta)^4} \mu_m^2  + \frac{B_8\mu_m^2\sigma_s^2}{(1-\beta)^3} \nnb
%		& \le \ \frac{B_6 \alpha^{i+1} \mu_m^2}{(1-\beta)^2}  +  \frac{B_8 (\delta^2+\gamma^2)  \rho_1^{i+1}}{(1-\beta)^4} \mu_m^2  + \frac{B_8\mu_m^2\sigma_s^2}{(1-\beta)^3} \nnb
		& \overset{(a)}{\le} \ \frac{B_6 \rho_1^{i+1} \mu_m^2}{(1-\beta)^2}  +  \frac{B_9 (\delta^2+\gamma^2)  \rho_1^{i+1}}{(1-\beta)^4} \mu_m^2  + \frac{B_8\mu_m^2\sigma_s^2}{(1-\beta)^3} \nnb
		& \overset{(b)}{\le} \ \frac{B_{10}(\delta^2 + \gamma^2)\rho_1^{i+1}}{(1-\beta)^4}\mu_m^2 + \frac{B_8\mu_m^2\sigma_s^2}{(1-\beta)^3},
		  \label{rl-cw-bound-1-asduy-23789}
		}
		where (a) holds because $\alpha \le \rho_1$ when $\mu_m$ is sufficiently small, and there must exist some constant $B_9$ such that $B_7/\rho_1 < B_9$; (b) holds because there must exist some constant $B_{10}$ such that
		\eq{\label{jlui890}
		B_6(1-\beta)^2 + B_9 (\delta^2+\gamma^2) \le B_{10} (\delta^2+\gamma^2).
		} 
	}

		%We remark that while the subsequent results in the article are achieved under the initial condition $\w_{-1}=\w_{-2}$, the conclusions will nevertheless continue to hold under the above more
		%general conclusion for the same reasons explained here.
		
		{\color{black}
		\section{Proof of Theorem \ref{lm:gfat-conv-4-4}}
		\label{app:4-th order moment}
		The argument below is motivated by the derivation of Theorem 9.2 in (Sayed, 2014a). Here, however, we extend the arguments and expand the details in order to clearly identify the constants inside the $O(\mu)$ notation, which was not necessary in (Sayed, 2014a). The derivation becomes more demanding, as the arguments show. 
		
		From the first row of recursion \eqref{sta-hb iteration 2} we have
		\bqq
		\hw_i&=&\left(I_M-\frac{\mu_m \H_{i-1} }{1-\beta}\right)\hw_{i-1}+\frac{\mu_m\beta^\prime \H_{i-1}}{1-\beta}\cw_{i-1} +\frac{\mu_m}{1-\beta}\s_i(\bpsi_{i-1}).
		\label{sta-hw update-2s}
		\eqq
		Now applying the following inequality, for any two vectors $\{a,b\}$:
		\eq{
			\|a+b\|^4\le \|a\|^4+3\|b\|^4+8\|a\|^2\|b\|^2+4\|a\|^2(a\tran b) \label{a+b 4 bound-asdluy}
		}
		we get
		\eq{\label{238asdg23}
		&\ \bE[\|\hw_i\|^4|\filt_{i-1}] \nnb
		&\ = \left\| \left(I_M-\frac{\mu_m \H_{i-1} }{1-\beta}\right)\hw_{i-1}+\frac{\mu_m\beta^\prime \H_{i-1}}{1-\beta}\cw_{i-1} \right\|^4 + \frac{3\mu_m^4}{(1-\beta)^4}\bE[\|\s_i(\bpsi_{i-1})\|^4|\filt_{i-1}] \nnb
		&\ \hspace{5mm} + \frac{8\mu_m^2}{(1-\beta)^2} \left\| \left(I_M-\frac{\mu_m \H_{i-1} }{1-\beta}\right)\hw_{i-1}+\frac{\mu_m\beta^\prime \H_{i-1}}{1-\beta}\cw_{i-1} \right\|^2 \bE[\|\s_i(\bpsi_{i-1})\|^2|\filt_{i-1}] \nnb
		&\ \le \left\| \left(I_M-\frac{\mu_m \H_{i-1} }{1-\beta}\right)\hw_{i-1}+\frac{\mu_m\beta^\prime \H_{i-1}}{1-\beta}\cw_{i-1} \right\|^4 + \frac{3\mu_m^4(\gamma_4^4\|\tpsi_{i-1}\|^4 + \sigma_{s,4}^4)  }{(1-\beta)^4} \nnb
		&\ \hspace{5mm} + \frac{8\mu_m^2}{(1-\beta)^2} \left\| \left(I_M-\frac{\mu_m \H_{i-1} }{1-\beta}\right)\hw_{i-1}+\frac{\mu_m\beta^\prime \H_{i-1}}{1-\beta}\cw_{i-1} \right\|^2 (\gamma^2\|\tpsi_{i-1}\|^2 + \sigma_s^2).
		}
		We next bound each of the terms that appear on the right-hand side.  Using Jensen's inequality, the
		lower and upper bounds on the Hessian matrix from (\ref{cost function assumption}), we have
		\eq{\label{239asdhjn}
		&\ \left\| \left(I_M-\frac{\mu_m \H_{i-1} }{1-\beta}\right)\hw_{i-1}+\frac{\mu_m\beta^\prime \H_{i-1}}{1-\beta}\cw_{i-1} \right\|^4 \nnb
		=&\ \left\| (1-t) \frac{1}{1-t} \left(I_M-\frac{\mu_m \H_{i-1} }{1-\beta}\right)\hw_{i-1}+ t \frac{1}{t}\frac{\mu_m\beta^\prime \H_{i-1}}{1-\beta}\cw_{i-1} \right\|^4 \nnb
		\le&\ \frac{1}{(1-t)^3} \left( 1 - \frac{\mu_m \nu}{1-\beta} \right)^4 \|\hw_{i-1}\|^4 + \frac{1}{t^3} \frac{\mu_m^4 \beta^{\prime 4}\delta^4}{(1-\beta)^4} \|\cw_{i-1}\|^4 \nnb
		\overset{(a)}{=}&\ \left( 1 - \frac{\mu_m \nu}{1-\beta} \right) \|\hw_{i-1}\|^4 + \frac{\mu_m \beta^{\prime 4}\delta^4}{(1-\beta)\nu^3}\|\cw_{i-1}\|^4 \nnb
		=&\ (1-q_1 \mu_m) \|\hw_{i-1}\|^4 + q_2 \mu_m \|\cw_{i-1}\|^4.
		}
		where (a) holds because we set $t = \mu_m \nu/(1-\beta)$, and $q_1$, $q_2$ are defined as
		\eq{\label{238adshjg}
		q_1 \define \frac{\nu}{1-\beta},\quad \quad q_2 \define \frac{\beta^{\prime 4}\delta^4}{(1-\beta)\nu^3}.
		}
		Next we check the terms $\bE[\|\s_i(\bpsi_{i-1})\|^2|\filt_{i-1}]$ and $\bE[\|\s_i(\bpsi_{i-1})\|^4|\filt_{i-1}]$. From \eqref{bound tpsi} we have
		\bqq
		\|\tpsi_{i-1}\|^2&\le B_1 (\|\hw_{i-1}\|^2 + \|\cw_{i-1}\|^2), \label{subbound-3-klklly3}
		\eqq
		where $B_1=2(1+\beta_1)^2 v^2$, which also implies that
		\bqq
		\|\tpsi_{i-1}\|^4 &\le & B_1^2 (\|\hw_{i-1}\|^2 + \|\cw_{i-1}\|^2)^2 \nnb
		&\le & 2 B_1^2 (\|\hw_{i-1}\|^4 + \|\cw_{i-1}\|^4) \ =\  B_2 (\|\hw_{i-1}\|^4 + \|\cw_{i-1}\|^4), \label{subbound-4-lklhyklu2}
		\eqq
		where $B_2=2 B_1^2$. Furthermore, recall in \eqref{sta-hw update-expe-filt} that 
		\eq{\label{238asdh5723}
		&\ \left\| \left(I_M-\frac{\mu_m \H_{i-1} }{1-\beta}\right)\hw_{i-1}+\frac{\mu_m\beta^\prime \H_{i-1}}{1-\beta}\cw_{i-1} \right\|^2 \nnb
		\le&\ \left(1-\frac{\mu_m \nu}{1-\beta}\right) \|\hw_{i-1}\|^2 + \frac{\mu_m \beta^{\prime 2} \delta^2}{\nu(1-\beta)}\|\cw_{i-1}\|^2 \nnb
		=&\ (1-q_1 \mu_m) \|\hw_{i-1}\|^2 + q_3 \mu_m \|\cw_{i-1}\|^2,
		}
		where we define 
		\eq{\label{238da}
		q_3 \define \frac{\beta^{\prime 2}\delta^2}{\nu(1-\beta)}.
		}
		Now substituting \eqref{239asdhjn}, \eqref{subbound-3-klklly3}, \eqref{subbound-4-lklhyklu2} and \eqref{238asdh5723} into \eqref{238asdg23}, we get
		\eq{\label{239asj5}
		&\ \bE[\|\hw_i\|^4|\filt_{i-1}] \nnb
		\le&\ (1-q_1 \mu_m) \|\hw_{i-1}\|^4 + q_2 \mu_m \|\cw_{i-1}\|^4 + \frac{3B_2 \gamma_4^4 \mu_m^4 }{(1-\beta)^4}\left( \|\hw_{i-1}\|^4 + \|\cw_{i-1}\|^4 \right) + \frac{3\sigma_{s,4}^4 \mu_m^4}{(1-\beta)^4}\nnb
		&\quad + \frac{8\mu_m^2}{(1-\beta)^2}\left[ (1-q_1 \mu_m) \|\hw_{i-1}\|^2 + q_3 \mu_m \|\cw_{i-1}\|^2 \right] \left[ \gamma^2 B_1 \left( \|\hw_{i-1}\|^2 + \|\cw_{i-1}\|^2 \right) + \sigma_s^2 \right] \nnb
		=&\ (1-q_1 \mu_m) \|\hw_{i-1}\|^4 + q_2 \mu_m \|\cw_{i-1}\|^4 + q_4 \mu_m^4 \left( \|\hw_{i-1}\|^4 + \|\cw_{i-1}\|^4 \right) + q_5\mu_m^4\nnb
		&\quad + q_6\mu_m^2\left[ (1-q_1 \mu_m) \|\hw_{i-1}\|^2 + q_3 \mu_m \|\cw_{i-1}\|^2 \right] \left[ \gamma^2 B_1 \left( \|\hw_{i-1}\|^2 + \|\cw_{i-1}\|^2 \right) + \sigma_s^2 \right], \nnb
		=&\ (1- q_1 \mu_m + q_4 \mu_m^4) \|\hw_{i-1}\|^4 + (q_2 \mu_m + q_4 \mu_m^4)\|\cw_{i-1}\|^4 + q_5 \mu_m^4  \nnb
		&\quad + q_6 \gamma^2 B_1 (1-q_1\mu_m)\mu_m^2 \|\hw_{i-1}\|^4 + q_6 q_3 \gamma^2 B_1 \mu_m^3 \|\cw_{i-1}\|^4 \nnb
		&\quad + q_6 \gamma^2 B_1 \mu_m^2 (1-q_1\mu_m + q_3 \mu_m) \|\hw_{i-1}\|^2\|\cw_{i-1}\|^2 \nnb
		&\quad + q_6 \sigma_s^2 \mu_m^2 (1-q_1\mu_m)\|\hw_{i-1}\|^2 + q_6 q_3 \sigma_s^2 \mu_m^3 \|\cw_{i-1}\|^2\nnb
		\overset{(a)}{\le} &\ (1- q_1 \mu_m + q_4 \mu_m^4) \|\hw_{i-1}\|^4 + (q_2 \mu_m + q_4 \mu_m^4)\|\cw_{i-1}\|^4 + q_5 \mu_m^4  \nnb
		&\quad + q_6 \gamma^2 B_1 (1-q_1\mu_m)\mu_m^2 \|\hw_{i-1}\|^4 + q_6 q_3 \gamma^2 B_1 \mu_m^3 \|\cw_{i-1}\|^4 \nnb
		&\quad + q_6 \gamma^2 B_1 \mu_m^2 (1-q_1\mu_m + q_3 \mu_m) \left(\|\hw_{i-1}\|^4 + \|\cw_{i-1}\|^4 \right) \nnb
		&\quad + q_6 \sigma_s^2 \mu_m^2 (1-q_1\mu_m)\|\hw_{i-1}\|^2 + q_6 q_3 \sigma_s^2 \mu_m^3 \|\cw_{i-1}\|^2,
		}
		where we define 
		\eq{\label{239asd7h4}
		q_4 \define \frac{3B_2 \gamma_4^4}{(1-\beta)^4}, \quad q_5 \define \frac{3\sigma_{s,4}^4}{(1-\beta)^4}, \quad q_6 \define \frac{8}{(1-\beta)^2},
		}
		and (a) holds because for any two variables $a,b > 0$ we have 
		\eq{\label{284567j}
		ab < 2ab \le a^2 + b^2.
		}
		When $\mu_m$ is chosen sufficiently small, from \eqref{239asj5} we reach
		\eq{\label{lkj;lu28}
		&\ \bE[\|\hw_i\|^4|\filt_{i-1}] \nnb
		\le &\ \left(1- \frac{q_1 \mu_m}{2}\right) \|\hw_{i-1}\|^4 + 2 q_2 \mu_m \|\cw_{i-1}\|^4 + q_6 \sigma_s^2 \mu_m^2\|\hw_{i-1}\|^2 + q_6 q_3 \sigma_s^2 \mu_m^3 \|\cw_{i-1}\|^2+ q_5 \mu_m^4 \nnb
		=&\ \left(1- \frac{q_1 \mu_m}{2}\right) \|\hw_{i-1}\|^4 + 2 q_2 \mu_m \|\cw_{i-1}\|^4 + q_7 \mu_m^2\|\hw_{i-1}\|^2 + q_8 \mu_m^3 \|\cw_{i-1}\|^2+ q_5 \mu_m^4, 
		}
		where we define
		\eq{\label{239adsalj8}
		q_7 \define q_6\sigma_s^2, \quad \quad q_8 \define q_6 q_3 \sigma_s^2.
		}
		
		On the other hand, recall from \eqref{second row} that
		\eq{
			\label{second row-cor-2}
			\cw_i=-\frac{\mu_m}{1-\beta} \H_{i-1} \hw_{i-1} & + \left(\beta I_M + \frac{\mu_m \beta^\prime}{1-\beta} \H_{i-1}\right)\cw_{i-1}  + \frac{\mu_m}{1-\beta} \s_i(\bpsi_{i-1}).
		}
		Now we also apply inequality \eqref{a+b 4 bound-asdluy} to the above equation and get
		\eq{
			&\hspace{-1cm}	\bE[\|\cw_i\|^4|\filt_{i-1}] \nnb
			=	&\left\| -\frac{\mu_m \H_{i-1}}{1-\beta} \hw_{i-1}  + \left(\beta I_M + \frac{\mu_m \beta^\prime}{1-\beta} \H_{i-1}\right)\cw_{i-1} \right\|^4  + \frac{3\mu_m^4 \bE\left[\|\s_i(\bpsi_{i-1})\|^4|\filt_{i-1}\right]}{(1-\beta)^4} \nonumber  \\
			& +\frac{8\mu_m^2}{(1-\beta)^2}\hspace{-1mm} \left\|\frac{-\mu_m \H_{i-1}}{1-\beta} \hw_{i-1} \hspace{-1mm}  + \hspace{-1mm} \left(\hspace{-1mm} \beta I_M \hspace{-0.5mm} + \hspace{-0.5mm} \frac{\mu_m \beta^\prime}{1-\beta} \H_{i-1}\hspace{-1mm} \right)\hspace{-1mm} \cw_{i-1}\right\|^2  \bE\left[\|\s_i(\bpsi_{i-1})\|^2|\filt_{i-1}\right]  \nnb
			\le & \left\| -\frac{\mu_m \H_{i-1}}{1-\beta} \hw_{i-1}  + \left(\beta I_M + \frac{\mu_m \beta^\prime}{1-\beta} \H_{i-1}\right)\cw_{i-1} \right\|^4 + \frac{3\mu_m^4 (\gamma_4^4 \|\tpsi_{i-1}\|^4+\sigma_{s,4}^4)}{(1-\beta)^4} \nnb
			& + \frac{8\mu_m^2}{(1-\beta)^2}\hspace{-1mm} \left\|\frac{-\mu_m \H_{i-1}}{1-\beta} \hw_{i-1} \hspace{-1mm}  + \hspace{-1mm} \left(\hspace{-1mm} \beta I_M \hspace{-0.5mm} + \hspace{-0.5mm} \frac{\mu_m \beta^\prime}{1-\beta} \H_{i-1}\hspace{-1mm} \right)\hspace{-1mm} \cw_{i-1}\right\|^2  (\gamma^2 \|\tpsi_{i-1}\|^2+\sigma_s^2) . \label{cw-4-bound}
		}
		We next bound each of the terms that appear on the right-hand side.  Using Jensen's inequality, the
		lower and upper bounds on the Hessian matrix from (\ref{cost function assumption}), and
		the inequality $\|a+b\|^4\leq 8\|a\|^4+8\|b\|^4$, we have
		\eq{
			&\hspace{-2cm} \left\| -\frac{\mu_m \H_{i-1}}{1-\beta} \hw_{i-1}  + \left(\beta I_M + \frac{\mu_m \beta^\prime}{1-\beta} \H_{i-1}\right)\cw_{i-1} \right\|^4 \nnb
			=&\ \left\|\beta\cw_{i-1}+(1-\beta)\left(\frac{\mu_m\beta'}{(1-\beta)^2}
			\H_{i-1}\cw_{i-1}-\frac{\mu_m}{(1-\beta)^2}\H_{i-1}\hw_{i-1}\right)\right\|^4\nnb
			\le&\ \beta\|\cw_{i-1}\|^4 + (1-\beta)\left\|\frac{\mu_m \beta'}{(1-\beta)^2}
			\H_{i-1}\cw_{i-1}-\frac{\mu_m}{(1-\beta)^2}\H_{i-1}\hw_{i-1}\right\|^4\nnb
			\le&\ \beta\|\cw_{i-1}\|^4  + \frac{8\mu_m^4 \beta^{\prime 4}\delta^4}{(1-\beta)^7}\|\cw_{i-1}\|^4 +\frac{8\mu_m^4\delta^4}{(1-\beta)^7}\|\hw_{i-1}\|^4\nnb
			=&\ (\beta+p_1 \mu_m^4) \|\cw_{i-1}\|^4 + p_2 \mu_m^4 \|\hw_{i-1}\|^4, \label{subbound-1}
		}
		where we define
		\eq{\label{238asdj;lo8}
		p_1 \define \frac{8 \beta^{\prime 4}\delta^4}{(1-\beta)^7},\quad \quad p_2 \define \frac{8\delta^4}{(1-\beta)^7}.
	}
		Moreover, recall in \eqref{sta-cwi-bound-1} that 
		\eq{\label{238asdn;}
		&\hspace{-3cm} \left\|\beta \cw_{i-1} + \left(\frac{\mu_m \beta^\prime \H_{i-1}}{1-\beta}\cw_{i-1} -\frac{\mu_m \H_{i-1}}{1-\beta} \hw_{i-1}\right) \right\|^2  \nnb
		&\le\  \Big(\beta \hspace{-0.5mm}+\hspace{-0.5mm}  \frac{2\mu_m^2 \beta^{\prime 2} \delta^2}{(1\hspace{-0.5mm}-\hspace{-0.5mm}\beta)^3}  \Big) \|\cw_{i\hspace{-0.5mm}-\hspace{-0.5mm}1}\|^2 \hspace{-0.5mm}+\hspace{-0.5mm} \frac{2\mu_m^2 \delta^2}{(1\hspace{-0.5mm}-\hspace{-0.5mm}\beta)^3} \|\hw_{i\hspace{-0.5mm}-\hspace{-0.5mm}1}\|^2 \nnb
		&=\ (\beta + p_3 \mu_m^2) \|\cw_{i-1}\|^2 + p_4 \mu_m^2 \|\hw_{i-1}\|^2,
	}
	where we define
	\eq{\label{28asdhn}
	p_3 \define \frac{2\beta^{\prime 2}\delta^2}{(1-\beta)^3}, \quad p_4 \define \frac{2\delta^2}{(1-\beta)^3}.
	}
	Now substituting \eqref{subbound-1}, \eqref{238asdn;}, \eqref{subbound-3-klklly3} and \eqref{subbound-4-lklhyklu2} into \eqref{cw-4-bound}, we have
	\eq{
		&\	\bE[\|\cw_i\|^4|\filt_{i-1}] \nnb
		\le &\ (\beta+p_1 \mu_m^4) \|\cw_{i-1}\|^4 + p_2 \mu_m^4 \|\hw_{i-1}\|^4 + \frac{3B_2 \gamma_4^4 \mu_m^4 }{(1-\beta)^4}\left( \|\hw_{i-1}\|^4 + \|\cw_{i-1}\|^4 \right) + \frac{3\sigma_{s,4}^4 \mu_m^4}{(1-\beta)^4}\nnb
		& \quad + \frac{8\mu_m^2}{(1-\beta)^2}\left[ ( \beta + p_3 \mu_m^2 ) \|\cw_{i-1}\|^2 + p_4 \mu_m^2 \|\hw_{i-1}\|^2 \right] \left[ \gamma^2 B_1 \left( \|\hw_{i-1}\|^2 + \|\cw_{i-1}\|^2 \right) + \sigma_s^2 \right] \nnb
		= &\ (\beta+p_1 \mu_m^4) \|\cw_{i-1}\|^4 + p_2 \mu_m^4 \|\hw_{i-1}\|^4 + p_5 \mu_m^4\left( \|\hw_{i-1}\|^4 + \|\cw_{i-1}\|^4 \right) + p_6 \mu_m^4\nnb
		& \quad + p_7\mu_m^2 \left[ ( \beta + p_3 \mu_m^2 ) \|\cw_{i-1}\|^2 + p_4 \mu_m^2 \|\hw_{i-1}\|^2 \right] \left[ \gamma^2 B_1 \left( \|\hw_{i-1}\|^2 + \|\cw_{i-1}\|^2 \right) + \sigma_s^2 \right] \nnb
		=&\ [\beta + (p_1 + p_5) \mu_m^4]\|\cw_{i-1}\|^4 + (p_2 + p_5)\mu_m^4 \|\hw_{i-1}\|^4 + p_6 \mu_m^4\nnb
		&\quad + p_7\gamma^2 B_1\mu_m^2(\beta + p_3 \mu_m^2)\|\cw_{i-1}\|^4 + p_4 p_7 \gamma^2 B_1 \mu_m^4 \|\hw_{i-1}\|^4 \nnb
		&\quad + p_7 \gamma^2 B_1 \mu_m^2 (\beta + p_3 \mu_m^2 + p_4 \mu_m^2)\|\hw_{i-1}\|^2\|\cw_{i-1}\|^2\nnb
		&\quad + p_7 \sigma_s^2 \mu_m^2 (\beta + p_3 \mu_m^2)\|\cw_{i-1}\|^2 + p_4 p_7 \sigma_s^2 \mu_m^4 \|\hw_{i-1}\|^2 \nnb
		\le&\ [\beta + (p_1 + p_5) \mu_m^4]\|\cw_{i-1}\|^4 + (p_2 + p_5)\mu_m^4 \|\hw_{i-1}\|^4 + p_6 \mu_m^4\nnb
		&\quad + p_7\gamma^2 B_1\mu_m^2(\beta + p_3 \mu_m^2)\|\cw_{i-1}\|^4 + p_4 p_7 \gamma^2 B_1 \mu_m^4 \|\hw_{i-1}\|^4 \nnb
		&\quad + p_7 \gamma^2 B_1 \mu_m^2 (\beta + p_3 \mu_m^2 + p_4 \mu_m^2)\left(\|\cw_{i-1}\|^4 + \|\hw_{i-1}\|^4 \right)\nnb
		&\quad + p_7 \sigma_s^2 \mu_m^2 (\beta + p_3 \mu_m^2)\|\cw_{i-1}\|^2 + p_4 p_7 \sigma_s^2 \mu_m^4 \|\hw_{i-1}\|^2,
		 \label{cw-4-bound-2hasu}
	}
	where we define 
	\eq{\label{2397an7}
	p_5\define \frac{3B_2\gamma_4^4}{(1-\beta)^4},\quad p_6\define \frac{3\sigma_{s,4}^4}{(1-\beta)^4},\quad p_7 \define \frac{8}{(1-\beta)^2}.
	}
	When $\mu_m$ is sufficiently small, we obtain from \eqref{cw-4-bound-2hasu}:
	\eq{\label{lkj;lu28239ah}
		&\ \bE[\|\cw_i\|^4|\filt_{i-1}] \nnb
		\le &\ (\beta + 2p_7\gamma^2 B_1 \mu_m^2) \|\cw_{i-1}\|^4 + 2p_7\gamma^2 B_1 \mu_m^2 \|\hw_{i-1}\|^4 + p_4 p_7 \sigma_s^2 \mu_m^4 \|\hw_{i-1}\|^2 \nnb
		&\quad + 2 p_7 \beta \sigma_s^2 \mu_m^2 \|\cw_{i-1}\|^2 + p_6 \mu_m^4 \nnb
		= &\ (\beta + p_8 \mu_m^2) \|\cw_{i-1}\|^4 + p_8 \mu_m^2 \|\hw_{i-1}\|^4 +  p_9 \mu_m^4 \|\hw_{i-1}\|^2  + p_{10} \mu_m^2 \|\cw_{i-1}\|^2 + p_6 \mu_m^4
	}
	where we define
	\eq{\label{239adsalj8-27ah}
		p_8 \define 2p_7\gamma^2 B_1, \quad p_9 \define p_4 p_7 \sigma_s^2, \quad p_{10} \define 2 p_7 \beta \sigma_s^2.
	}
	Combining \eqref{lkj;lu28} and \eqref{lkj;lu28239ah}, we have
	\eq{\label{238asd7}
	\ba{c}
	\bE[\|\hw_i\|^4|\filt_{i-1}]\\
	\bE[\|\cw_i\|^4|\filt_{i-1}]
	\ea
	\le
	\ba{cc}
	a & b\\
	c & d
	\ea
	\ba{c}
	\|\hw_{i-1}\|^4\\
	\|\cw_{i-1}\|^4
	\ea +
	\ba{cc}
	a' & b'\\
	c' & d'
	\ea
	\ba{c}
	\|\hw_{i-1}\|^2\\
	\|\cw_{i-1}\|^2
	\ea
	+
	\ba{c}
	e\\
	f
	\ea,
	}
	where the constants are
	\eq{\label{238adh5ck;}
	& a \define 1 - \frac{q_1}{2}\mu_m, \quad b \define 2q_2 \mu_m, \quad \hspace{5mm}a' \define q_6 \sigma_s^2 \mu_m^2, \quad b' \define q_3 q_6 \sigma_s^2 \mu_m^3, \nnb
	& c \define p_8 \mu_m^2,\quad \hspace{7.5mm}d \define \beta + p_8 \mu_m^2,\quad c'\define p_9\mu_m^4,\quad \hspace{3mm}d'\define p_{10}\mu_m^2, \nnb
	& e \define q_5\mu_m^4,\quad \hspace{7.5mm}f\define p_6\mu_m^4.
	}
	Taking expectations again over $\filt_{i-1}$ for both sides of the inequality \eqref{238asd7}, we have
	\eq{\label{238asd7-237anml}
		\ba{c}
		\bE\|\hw_i\|^4\\
		\bE\|\cw_i\|^4
		\ea
		\le
		\underbrace{
			\ba{cc}
			a & b\\
			c & d
			\ea}_{\Gamma}
		\ba{c}
		\bE\|\hw_{i-1}\|^4\\
		\bE\|\cw_{i-1}\|^4
		\ea +
		\ba{cc}
		a' & b'\\
		c' & d'
		\ea
		\ba{c}
		\bE\|\hw_{i-1}\|^2\\
		\bE\|\cw_{i-1}\|^2
		\ea
		+
		\ba{c}
		e\\
		f
		\ea,
	}
	Recall from Theorem \ref{thm-conv-cw} that
	\eq{\label{238adsb}
	\limsup_{i\to \infty}\bE\|\hw_{i-1}\|^2 = O(\mu_m), \quad \limsup_{i\to \infty}\bE\|\cw_{i-1}\|^2 = O(\mu_m^2),
	}
	then it holds that
	\eq{\label{nadhk237}
	\limsup_{i\to \infty} 
	\ba{cc}
	a' & b' \\
	c' & d'
	\ea
	\ba{c}
	\bE\|\hw_{i-1}\|^2 \\
	\bE\|\cw_{i-1}\|^2
	\ea
	+ 
	\ba{c}
	e\\
	f
	\ea
	=
	\ba{c}
	O(\mu_m^3) \\
	O(\mu_m^4)
	\ea.
	}
	When $\mu_m$ is sufficiently small, it can be verified that $\Gamma$ is stable. Therefore, it holds that
	\eq{\label{2389adsbn}
	\ba{c}
	\limsup_{i\to \infty}\bE\|\hw_i\|^4\\
	\limsup_{i\to \infty}\bE\|\cw_i\|^4
	\ea
	&= (I- \Gamma)^{-1} \left( \limsup_{i\to \infty} 
	\ba{cc}
	a' & b' \\
	c' & d'
	\ea
	\ba{c}
	\bE\|\hw_{i-1}\|^2 \\
	\bE\|\cw_{i-1}\|^2
	\ea
	+ 
	\ba{c}
	e\\
	f
	\ea \right) \nnb
	&=
	\ba{c}
	O(\mu_m^2) \\
	O(\mu_m^4)
	\ea.
	}
		Furthermore, 
		\eq{\label{238adh}
			\limsup_{i\to \infty} \bE\|\tw_i\|^4 \le 2v^4 \limsup_{i\to \infty}(\bE\|\hw_i\|^4 + \bE\|\cw_i\|^4 ) = O \left(\mu_m^2\right).
		}

}
		
		\section{Proof of Corollary \ref{co:hb-cw-u2-4}}\label{app:corr-2aa}
		{\color{black}
		Recall from \eqref{238asd7-237anml} that
		\eq{\label{238asd7-237anml-alkj2367}
			\underbrace{
			\ba{c}
			\bE\|\hw_i\|^4\\
			\bE\|\cw_i\|^4
			\ea}_{y_i}
			\le
			\underbrace{
				\ba{cc}
				a & b\\
				c & d
				\ea}_{\Gamma_1}
			\underbrace{
			\ba{c}
			\bE\|\hw_{i-1}\|^4\\
			\bE\|\cw_{i-1}\|^4
			\ea}_{y_{i-1}} +
			\underbrace{
			\ba{cc}
			a' & b'\\
			c' & d'
			\ea}_{\Gamma_2}
			\underbrace{
			\ba{c}
			\bE\|\hw_{i-1}\|^2\\
			\bE\|\cw_{i-1}\|^2
			\ea}_{z_{i-1}}
			+
			\underbrace{
			\ba{c}
			e\\
			f
			\ea}_{r}.
		}
		We then have
		\eq{\label{,asdhk8}
		\|y_i\|_1 \le \|\Gamma_1\|_1 \|y_{i-1}\|_1 + \|\Gamma_2\|_1\|z_{i-1}\|_1 + \|r\|_1.
		}
		Notice that 
		\eq{\label{n8236ahdl}
		\|\Gamma_1\|_1 = \max\left\{1 - \frac{q_1\mu_m}{2} + p_8 \mu_m^2, \beta + 2q_2 \mu_m + p_8 \mu_m^2 \right\},
		}
		we can always choose $\mu_m$ small enough such that 
		\eq{\label{l237adsbco98}
		\|\Gamma_1\|_1 \le 1 - \frac{q_1 \mu_m}{4} = 1 - \frac{\mu_m \nu}{4(1-\beta)}\define \rho_2.
		}
		Similarly, we can also choose $\mu_m$ small enough such that 
		\eq{\label{m279h}
		\|\Gamma_2\|_1 = \max\left\{q_6 \sigma_s^2 \mu_m^2 + p_9 \mu_m^4, q_3q_6\sigma_s^2 \mu_m^3 + p_{10}\mu_m^2 \right\} \le (q_6\sigma_s^2 + p_{10})\mu_m^2.
		}
		Also recall \eqref{zi inequality-3asdht} that 
		\eq{
			\label{zi inequality-3asdht-236a}
			\|z_i\|_1 \le \rho_1^{i+1} \|z_{-1}\|_1 +  s_1 \mu_m \overset{(a)}{\le} \rho_2^{i+1} \|z_{-1}\|_1 + s_1 \mu_m
		}
		where we define 
		\eq{\label{238adn}
		s_1 \define \frac{2E_1\sigma_s^2}{(1-\beta)\nu},
		}
		for some constant $E_1$, and (a) holds because $\rho_1 = 1 - \frac{\mu_m \nu}{2(1-\beta)}\le \rho_2$. Substituting \eqref{l237adsbco98}, \eqref{m279h} and \eqref{zi inequality-3asdht-236a} into \eqref{,asdhk8}, we have
		\eq{\label{b7628ayh}
		\|y_i\|_1 \le&\ \rho_2 \|y_{i-1}\|_1 + (q_6 \sigma_s^2 + p_{10})(\rho_2^i\|z_{-1}\|_1 + s_1\mu_m)\mu_m^2 + 2p_6 \mu_m^4 \nnb
		=&\ \rho_2 \|y_{i-1}\|_1 + s_2 \rho_2^i\mu_m^2 + s_3 \mu_m^3 + 2p_6 \mu_m^4 \nnb
		\overset{(a)}{\le}&\ \rho_2 \|y_{i-1}\|_1 + s_2 \rho_2^i\mu_m^2 + 2 s_3 \mu_m^3,
		}
		where (a) holds when $\mu_m$ is sufficiently small, and the constants $s_2$ and $s_3$ are defined as
		\eq{\label{2890jl}
		s_2 \define (q_6 \sigma_s^2 + p_{10})\|z_{-1}\|_1,\quad s_3 \define (q_6 \sigma_s^2 + p_{10})s_1.
		}
		Iterating \eqref{b7628ayh}, we reach
		\eq{\label{8a76cna0}
		\|y_i\|_1 \le \rho_2^{i+1} \|y_{-1}\|_1 + s_2 (i+1)\rho_2^i \mu_m^2 + \frac{2 s_3 \mu_m^3}{1-\rho_2}.
		}
		Since $\|y_i\|_1 = \bE\|\hw_i\|^4 + \bE\|\cw_i\|^4$, we have
		\eq{\label{8a76cna0-236an,l}
			\bE\|\hw_i\|^4 & \le \rho_2^{i+1} \|y_{-1}\|_1 + s_2 (i+1)\rho_2^i \mu_m^2 + \frac{2 s_3 \mu_m^3}{1-\rho_2} \nnb
			& = \rho_2^{i+1} \|y_{-1}\|_1 + s_2 (i+1)\rho_2^i \mu_m^2 + \frac{8(1-\beta)s_3\mu_m^2}{\nu} \nnb
			& = \rho_2^{i+1} \|y_{-1}\|_1 + s_2 (i+1)\rho_2^i \mu_m^2 + s_4 \mu_m^2,
		}
		where $s_4$ is defined as
		\eq{\label{b98236}
		s_4 \define \frac{8(1-\beta)s_3}{\nu}
		}
		Now we substitute \eqref{8a76cna0-236an,l} and \eqref{2348adsgh} into the second row of \eqref{238asd7-237anml} and reach
		\eq{\label{2380adan}
		\bE\|\cw_i\|^4 \le&\ (\beta + p_8\mu_m^2)\bE\|\cw_{i-1}\|^4 + p_8\mu_m^2 \left[ \rho_2^{i} \|y_{-1}\|_1 + s_2 i \rho_2^{i-1} \mu_m^2 + s_4 \mu_m^2 \right]\nnb
		&\quad + p_{10}\mu_m^2 \|\cw_{i-1}\|^2 + p_9 \mu_m^4 \left[ \rho_2^i \|z_{-1}\|_1 + s_1 \mu_m \right] + p_6\mu_m^4. 
		}
		Using the bounds for $\bE\|\cw_i\|^2$ from Corollary \ref{co:hb-cw-u2}, the above inequality becomes
		\eq{\label{2380adan-adshlakj7}
			\bE\|\cw_i\|^4 \le&\ (\beta + p_8\mu_m^2)\bE\|\cw_{i-1}\|^4 + p_8\mu_m^2 \left[ \rho_2^{i} \|y_{-1}\|_1 + s_2 i \rho_2^{i-1} \mu_m^2 + s_4 \mu_m^2 \right]\nnb
			&\quad + p_{10}E_2\mu_m^2 \left(\frac{(\delta^2 + \gamma^2)\rho_1^{i}\mu_m^2}{(1-\beta)^4} + \frac{\sigma_s^2 \mu_m^2}{(1-\beta)^3}  \right) \nnb
			&\quad + p_9 \mu_m^4 \left[ \rho_2^i \|z_{-1}\|_1 + s_1 \mu_m \right] + p_6\mu_m^4 \nnb
			\overset{(a)}{\le}&\ (1-\frac{1-\beta}{2}) \bE\|\cw_{i-1}\|^4 + s_5 \rho_2^{i}\mu_m^2 + s_6 i \rho_2^{i-1}\mu_m^4 + s_7 \mu_m^4  \nnb
			&\quad + s_8 \rho_1^{i}\mu_m^4 + s_9\mu_m^4 + s_{10} \rho_2^i \mu_m^4 + p_6\mu_m^4 + s_1 p_9\mu_m^5 \nnb
			\overset{(b)}{\le}&\ \alpha \bE\|\cw_{i-1}\|^4 + s_5 \rho_2^i \mu_m^2 + s_6i\rho_2^{i-1}\mu_m^4 + s_8 \rho_1^{i}\mu_m^4 + s_{10} \rho_2^i \mu_m^4 \nnb
			&\quad + (s_7 + s_9 + 2p_6) \mu_m^4 \nnb
			\overset{(c)}{\le}&\ \alpha \bE\|\cw_{i-1}\|^4 + s_5 \rho_2^i \mu_m^2 + s_6i\rho_2^{i-1}\mu_m^4 + s_8 \rho_2^{i}\mu_m^4 + s_{10} \rho_2^i \mu_m^4 \nnb
			&\quad + (s_7 + s_9 + 2p_6) \mu_m^4 
%			\overset{(b)}{\le}&\ \alpha_1 \bE\|\cw_{i-1}\|^4 + s_5 \alpha_1^i \mu_m^2 + s_6i\alpha_1^{i-1}\mu_m^4 + s_8 \alpha_1^i \mu_m^4 + s_8 \alpha_1^{i-1}\mu_m^4 + s_{10} \alpha_1^i \mu_m^4 \nnb
%			&\quad + (s_7 + s_9 + 2p_6) \mu_m^4
		}
		where $E_2$ is some constant. The inequality (a) holds because step-size $\mu_m$ is chosen small enough such that ${(1-\beta)}/{2}>p_8\mu_m^2$, (b) holds because $\alpha = 1 - (1-\beta)/2$ and $\mu_m$ is chosen such that $p_6\mu_m^4 > s_1p_9 \mu_m^5$, and (c) holds because $\rho_1 < \rho_2$.
		 Moreover, the other constants are defined as
		\eq{\label{76239an}
		& s_5\define p_8\|y_{-1}\|_1, \quad \hspace{1.05cm}s_6\define p_8s_2, \quad \hspace{9mm}s_7 \define p_8s_4 \nnb
		& s_8\define \frac{p_{10}E_2(\delta^2+\gamma^2)}{(1-\beta)^4}, \quad s_9\define \frac{p_{10}E_2\sigma_s^2}{(1-\beta)^3},\quad s_{10}\define p_9\|z_{-1}\|_1.
		}
		Now we continue iterating \eqref{2380adan-adshlakj7} and reach
		\eq{\label{2380adan-adshlakj7-236a8m}
			\bE\|\cw_i\|^4 \le&\ \alpha^{i+1}\bE\|\cw_{-1}\|^4 + s_5\mu_m^2 \rho_2^i\sum_{k=0}^{i}\left(\frac{\alpha}{\rho_2}\right)^k + s_6\mu_m^4\rho_2^{i-1} \sum_{k=0}^{i}(i-k)\left(\frac{\alpha}{\rho_2}\right)^k \nnb
			&\quad + s_8\mu_m^4\rho_2^{i}\sum_{k=0}^{i}\left(\frac{\alpha}{\rho_2}\right)^k +s_{10}\mu_m^4\rho_2^i \sum_{k=0}^{i}\left(\frac{\alpha}{\rho_2}\right)^k + \frac{2(s_7 + s_9 + 2p_6) \mu_m^4 }{1-\beta}.
		}
%		{\color{red}[Stop!]}
		Recall $\rho_2 = 1 - \frac{\mu_m \nu}{4(1-\beta)}$, and we can choose $\mu_m$ small enough such that
		$\rho_2  > \alpha = 1 - \frac{1-\beta}{2}$. In this situation, we have 
		\eq{\label{238abnc}
			\frac{\alpha}{\rho_2} < 1,\quad \text{and} \quad \sum_{k=0}^{i}\left(\frac{\alpha}{\rho_2}\right)^k < \sum_{k=0}^{\infty}\left(\frac{\alpha}{\rho_2}\right)^k = \frac{\rho_2}{\rho_2 - \alpha} \le  \frac{E_3}{1-\beta}, 
		}
		where $E_3$ is some constant. Meanwhile, we also have
		\eq{\label{238adbcmn}
		\sum_{k=0}^{i}(i-k)\left(\frac{\alpha}{\rho_2}\right)^k \le i \sum_{k=0}^{i}\left(\frac{\alpha}{\rho_2}\right)^k \le i \sum_{k=0}^{\infty}\left(\frac{\alpha}{\rho_2}\right)^k \le \frac{ i E_3}{1-\beta}.
		}
		We substitute \eqref{238abnc} and \eqref{238adbcmn} into \eqref{2380adan-adshlakj7-236a8m}, and reach
		\eq{\label{2380adan-adshlakj7-236a8m-2367ads78}
			\bE\|\cw_i\|^4 \le&\ \rho_2^{i+1}\bE\|\cw_{-1}\|^4 + \frac{E_3 s_5\mu_m^2 \rho_2^i}{1-\beta} + \frac{iE_3 s_6\mu_m^4\rho_2^{i-1}}{1-\beta}  \nnb
			&\quad  + \frac{E_3 s_8\mu_m^4\rho_2^{i}}{1-\beta} +\frac{E_3 s_{10}\mu_m^4\rho_2^i}{1-\beta}   + \frac{2(s_7 + s_9 + 2p_6) \mu_m^4 }{1-\beta}.
		}
		Recall from \eqref{ini-state-142}
		that 
		$ \bE\|\cw_{-1}\|^4=E_4 \mu_m^4,$
		where $E_4=\bE\|\grad_w Q(\w_{-2};\btheta_{-1})\|^4$. Substituting into \eqref{2380adan-adshlakj7-236a8m-2367ads78} we reach that
		\eq{\label{m9236adgjkl-237adg}
			\bE\|\cw_i\|^4 \le&\ E_4 \rho_2^{i+1}\mu_m^4 + \frac{E_3 s_5\mu_m^2 \rho_2^i}{1-\beta} + \frac{iE_3 s_6\mu_m^4\rho_2^{i-1}}{1-\beta} \nnb
			&\quad  + \frac{E_3 s_8\mu_m^4\rho_2^{i}}{1-\beta} +\frac{E_3 s_{10}\mu_m^4\rho_2^i}{1-\beta}   + \frac{2(s_7 + s_9 + 2p_6) \mu_m^4 }{1-\beta}.
		}
		Substituting \eqref{76239an} into \eqref{m9236adgjkl-237adg} and recall $\alpha < \rho_2$, we finally reach
		\eq{\label{2397cnam}
			\bE\|\cw_i\|^4 =&\ O\left( \frac{\gamma^2 \rho_2^i}{(1-\beta)^3}\mu_m^2 + \rho_2^{i+1}\mu_m^4 + \frac{\gamma^2\sigma_s^2 i \rho_2^{i-1}}{(1-\beta)^5}\mu_m^4 + \frac{\sigma_s^2(\delta^2 + \gamma^2)\rho_2^{i}}{(1-\beta)^7}\mu_m^4 + \frac{\delta^2 \sigma_s^2 \rho_2^i}{(1-\beta)^6}\mu_m^4 \right.\nnb
			& \quad \quad \quad \left. + \frac{\gamma^2 \sigma_s^4}{(1-\beta)^5\nu^2}\mu_m^4 + \frac{\sigma_s^4}{(1-\beta)^6}\mu_m^4 + \frac{\sigma_{s,4}^4}{(1-\beta)^5}\mu_m^4  \right). \nnb
%			=&\ O \left( \frac{\gamma^2\rho_2^{i+1}}{(1-\beta)^3}\mu_m^2 + \frac{\sigma_s^2(\delta^2 + \gamma^2)(i+1)\rho_2^{i+1}}{(1-\beta)^7}\mu_m^4 + \frac{\gamma^2\sigma_s^4 + \sigma_{s,4}^4}{(1-\beta)^5\nu^2}\mu_m^4\right).
		}
		Since there must exist some constants $E_5$, $E_6$ and $E_7$ such that
		\eq{\label{ah2378}
		\frac{\gamma^2 \rho_2^i}{(1-\beta)^3}\mu_m^2 &\le \frac{E_5 \gamma^2 \rho_2^{i+1}}{(1-\beta)^3}\mu_m^2,\nnb
		\rho_2^{i+1}\mu_m^4 + \frac{\gamma^2\sigma_s^2 i \rho_2^{i-1}}{(1-\beta)^5}\mu_m^4 + \frac{\sigma_s^2(\delta^2 + \gamma^2)\rho_2^{i}}{(1-\beta)^7}\mu_m^4 + \frac{\delta^2 \sigma_s^2 \rho_2^i}{(1-\beta)^6}\mu_m^4  & \le \frac{E_6\sigma_s^2(\delta^2 + \gamma^2)(i+1)\rho_2^{i+1}}{(1-\beta)^7}\mu_m^4, \nnb
		\frac{\gamma^2 \sigma_s^4}{(1-\beta)^5\nu^2}\mu_m^4 + \frac{\sigma_s^4}{(1-\beta)^6}\mu_m^4 + \frac{\sigma_{s,4}^4}{(1-\beta)^5}\mu_m^4 & \le \frac{E_7 [(\gamma^2 + \nu^2) \sigma_s^4 + \sigma_{s,4}^4\nu^2]}{(1-\beta)^6\nu^2}\mu_m^4,
		}
		we finally reach the conclusion in \eqref{238asdbmnc}.
		}

		\section{Proof of Theorem \ref{thm:gfat-stochastic gradient method-dist-hw}}
		\label{proof of thm 3}
		Subtracting (\ref{stochastic gradient method tx update}) and (\ref{rl tw-2}) we get
		\beqn
		\label{lm1-hw-tx}
		\hw_i-\tx_i=(I_M-\mu R_u)(\hw_{i-1}-\tx_{i-1})  + \mu (\s_i(\bpsi_{i-1}) - \s_i(\x_{i-1}))  + {\mu \beta^\prime} R_u\cw_{i-1},
		\eeqn
		where
		\eq{
			\s_i(\bpsi_{i-1})=(R_u-\u_i \u_i\tran) \tpsi_{i-1} - \u_i \v(i),\quad \s_i(\x_{i-1})=(R_u-\u_i \u_i\tran) \tx_{i-1} - \u_i \v(i). \label{si_x}
		}
		Substituting into \eqref{lm1-hw-tx} gives
		\eq{
			\label{hw-tx}
			\hw_i-\tx_i=&(I_M-\mu R_u)(\hw_{i-1}-\tx_{i-1}) + \mu (R_u-\u_i \u_i\tran)(\tpsi_{i-1} - \tx_{i-1})  + {\mu \beta^\prime}R_u\cw_{i-1}.
		}
		Now note that in the quadratic case, the Hessian matrix of $J(w)$ is equal to $R_u$. It follows that condition
		(\ref{cost function assumption}) is satisfied with the identifications
		$
		\nu=\lambda_{\min}(R_u),
		\delta=\lambda_{\max}(R_u).
		$
		Let $t\in(0,1)$. By squaring (\ref{hw-tx}) and taking expectations, and applying Jensen's inequality, we obtain
		\bqq
		& &\hspace{-1cm}  \bE \|\hw_i-\tx_i\|^2 \nnb
		&\le & \bE\left\|(I_M-\mu R_u)(\hw_{i-1}-\tx_{i-1}) + {\mu \beta^\prime}R_u\cw_{i-1}\right\|^2 + \mu^2 \bE\|R_u-\u_i \u_i\tran\|^2 \bE\|\tpsi_{i-1} - \tx_{i-1}\|^2 \nnb
		&\overset{\text{(a)}}{\le}&\  \frac{1}{1-t} (1-\mu \nu)^2 \bE\|\hw_{i-1}-\tx_{i-1} \|^2 + \frac{1}{t}{\mu^2\beta^{\prime 2}\delta^2}\bE \|\cw_{i-1}\|^2+ B_1 \mu^2  \bE\|\tpsi_{i-1} - \tx_{i-1}\|^2 \nnb
		&\overset{\text{(b)}}{\le}&\  (1-\mu \nu) \bE\|\hw_{i-1}-\tx_{i-1} \|^2+\frac{\mu \beta^{\prime 2}\delta^2}{\nu}\bE \|\cw_{i-1}\|^2 + B_1 \mu^2  \bE\|\tpsi_{i-1} - \tx_{i-1}\|^2,
		\label{hw-tx-expectation}
		\eqq
		where (a) holds because of Jensen's inequality and we let $B_1=\bE\|R_u-\u_i \u_i\tran\|^2$, and (b) holds by choosing $t=\mu\nu$. To bound the last term in the above relation, we use \eqref{sgfat-1-tilde} to note that
		\begin{align}
		\label{gfat:relation btw hw and tx}
		\tpsi_i-\tx_i&=\tw_{i}+\beta_1 (\tw_i-\tw_{i-1})-\tx_i =(\tw_{i}-\tx_i)-\beta_1(\tw_{i-1}-\tw_{i})
		\end{align}
		On the other hand, from (\ref{transform}) we have
		\begin{align}
		\label{relation btw hw and tx}
		\hw_i-\tx_i&=\frac{1}{1-\beta}(\tw_i-\tx_i)-\frac{\beta}{1-\beta}(\tw_{i-1}-\tx_i) = (\tw_i-\tx_i) - \frac{\beta}{1-\beta}(\tw_{i-1}-\tw_i).
		\end{align}
		so that
		\eq{
		\label{nag:relation btw hw and tx}
		\tpsi_i\hspace{-0.5mm}-\hspace{-0.5mm}\tx_i=\hw_i\hspace{-0.5mm}-\hspace{-0.5mm}\tx_i \hspace{-0.5mm}+\hspace{-0.5mm} \frac{\beta \hspace{-0.5mm}-\hspace{-0.5mm} \beta_1 \hspace{-0.5mm}+\hspace{-0.5mm} \beta \beta_1}{1\hspace{-0.5mm}-\hspace{-0.5mm}\beta}(\tw_{i\hspace{-0.5mm}-\hspace{-0.5mm}1}\hspace{-0.5mm}-\hspace{-0.5mm}\tw_i) = \hw_i\hspace{-0.5mm}-\hspace{-0.5mm}\tx_i \hspace{-0.5mm}+\hspace{-0.5mm} \frac{\beta^\prime}{1\hspace{-0.5mm}-\hspace{-0.5mm}\beta}(\tw_{i\hspace{-0.5mm}-\hspace{-0.5mm}1}\hspace{-0.5mm}-\hspace{-0.5mm}\tw_i) =\hw_i\hspace{-0.5mm}-\hspace{-0.5mm}\tx_i \hspace{-0.5mm}-\hspace{-0.5mm} \beta^\prime \cw_i.
	}
		where we used the definition for $\beta^\prime$ from \eqref{beta prime} and the definition for $\cw_i$ from \eqref{transform}.
		Therefore, from Jensen's inequality again, we get
		\eq{
			\label{2nd bound-2}
			\bE\|\tpsi_i-\tx_i\|^2 \le 2 \bE\| \hw_i-\tx_i \|^2 + {2\beta^{\prime\hspace{0.02cm}2}} \bE\|\cw_{i}\|^2.
		}
		Substituting into \eqref{hw-tx-expectation} gives
		\eq{
		\bE \|\hw_i-\tx_i\|^2 &\le  (1-\mu \nu + 2B_1 \mu^2 ) \bE\|\hw_{i-1}-\tx_{i-1} \|^2 +\left(\frac{\mu \beta^{\prime 2}\delta^2}{\nu} + {2B_1\beta^{\prime 2} \mu^2} \right)\bE \|\cw_{i-1}\|^2\nnb
		&\overset{(a)}{\le} \left(1-\frac{\mu \nu}{2} \right) \bE\|\hw_{i-1}-\tx_{i-1} \|^2 +  \frac{2\mu \beta^{\prime 2}\delta^2}{\nu} \bE\|\cw_{i-1}\|^2 \nnb
		&\le \left(1-\frac{\mu \nu}{2} \right) \bE\|\hw_{i-1}-\tx_{i-1} \|^2 +  \frac{B_2 \mu \delta^2}{\nu} \bE\|\cw_{i-1}\|^2,
		\label{hw-tx-final bound}
		}
		where $B_2 = 2 \beta^{\prime 2}$ and the inequality (a) holds when $\mu$ is chosen small enough such that
		\eq{
		\frac{\mu\nu}{2} > 2 B_1 \mu^2 \quad \text{and} \quad \frac{\mu \beta^{\prime 2} \delta^2}{\nu} > 2 B_1 \beta^{\prime 2} \mu^2.
		}
		{\color{black}
		Recall from Corollary \ref{co:hb-cw-u2} 
		%and inequality \eqref{co-rl-cw-bound-5}
		that 
		\eq{\label{ash762d}
		\bE\|\cw_i\|^2 \le  C_1 \left( \frac{(\delta^2 + \gamma^2)\rho_1^{i+1}\mu_m^2}{(1-\beta)^4} + \frac{\sigma_s^2 \mu_m^2}{(1-\beta)^3} \right)
		}
		for each iteration $i=0,1,2,3,\ldots$, where $C_1$ is some constant. Recall $\mu=\mu_m/(1-\beta)$, we then have
		 \eq{\label{ash762dsdga2}
		 	\bE\|\cw_i\|^2 \le C_1 \left( \frac{(\delta^2 + \gamma^2) \rho_1^{i+1} \mu^2}{(1-\beta)^2} + \frac{\sigma_s^2 \mu^2}{1-\beta} \right)
		 }
		This fact,
		together with inequality \eqref{hw-tx-final bound}, leads to
		\eq{
		\bE \|\hw_i-\tx_i\|^2 & \leq
		\left(1-\frac{\mu \nu}{2} \right) \bE\|\hw_{i-1}-\tx_{i-1} \|^2 + B_2 C_1 \left( \frac{\delta^2 (\delta^2 + \gamma^2) \rho_1^{i} \mu^3}{\nu(1-\beta)^2} + \frac{\delta^2 \sigma_s^2 \mu^3}{\nu(1-\beta)}\right).
		\label{hw-tx-final bound-2}
		}
		Recall from Corollary \ref{co:hb-cw-u2} that $\rho_1 = 1 - \frac{\mu_m \nu}{2(1-\beta)}=1 - \frac{\mu \nu}{2}$, then \eqref{hw-tx-final bound-2} becomes
		\eq{
			\bE \|\hw_i-\tx_i\|^2  \leq
			\rho_1 \bE\|\hw_{i-1}-\tx_{i-1} \|^2 + B_2 C_1 \left(  \frac{\delta^2(\delta^2 + \gamma^2) \rho_1^{i}  \mu^3}{\nu(1-\beta)^2} + \frac{\delta^2 \sigma_s^2 \mu^3}{\nu(1-\beta)}\right).
			\label{hw-tx-final bound-2-q2lkjsad67}
		}
		For brevity, we denote 
		\eq{\label{238dha}
		e_1 \define \frac{B_2 C_1(\delta^2 + \gamma^2) \delta^2}{\nu(1-\beta)^2}, \quad e_2 \define \frac{B_2 C_1 \delta^2 \sigma_s^2}{\nu(1-\beta)}. 
		}
		Inequality \eqref{hw-tx-final bound-2-q2lkjsad67} will become
		\eq{
			\bE \|\hw_i-\tx_i\|^2  &\leq
			\rho_1 \bE\|\hw_{i-1}-\tx_{i-1} \|^2 + e_1 \rho_1^{i} \mu^3 + e_2 \mu^3 \nnb
			&\le \rho_1^{i+1} \bE\|\hw_{-1}-\tx_{-1} \|^2  + e_1(i+1)\rho_1^{i}\mu^3 + \frac{e_2 \mu^3}{1-\rho_1}.
%			&\overset{\eqref{238dha}}{\le} \rho_1^{i+1} \bE\|\hw_{-1}-\tx_{-1} \|^2 + \frac{B_2 C_1(\delta^2 + \gamma^2) \delta^2(i+1)\rho_1^{i}}{\nu(1-\beta)^2}\mu^3 + \frac{2B_2C_1\delta^2 \sigma_s^2}{\nu^2(1-\beta)}\mu^2 
%			\nnb
%			&\le \rho_1^{i+1} \bE\|\hw_{-1}-\tx_{-1} \|^2 + \frac{4B_2C_1\delta^2 \sigma_s^2}{\nu^2(1-\beta)}\mu^2
			\label{salkhreu}
		}
%			\label{hw-tx-final bound-2-q2lkjsad67-asdy}
%		}
%		Recall that $\alpha = 1 - \epsilon/2$ in Corollary \ref{co:hb-cw-u2}, we can adjust $\mu_m$ small enough such that $\alpha < \rho_1$. Under this situation, \eqref{hw-tx-final bound-2-q2lkjsad67-asdy} becomes 
%		\eq{
%			\bE \|\hw_i-\tx_i\|^2  &\leq
%			\rho_1 \bE\|\hw_{i-1}-\tx_{i-1} \|^2 + e_1 \alpha^{i} \mu^3 + e_2 \rho_1^{i-1} \mu^3 + e_3 \mu^3 \nnb
%			&\le \rho_1^{i+1} \bE\|\hw_{-1}-\tx_{-1} \|^2 + e_1 (i+1) \rho_1^i \mu^3 + e_2(i+1)\rho_1^{i}\mu^3 + \frac{e_3 \mu^3}{1-\rho_1} \nnb
%			&\le \rho_1^{i+1} \bE\|\hw_{-1}-\tx_{-1} \|^2 + (e_1 + e_2) (i+1) \rho_1^{i} \mu^3 + \frac{e_3 \mu^3}{1-\rho_1} \nnb
%			&\overset{\eqref{238dha}}{\le} \rho_1^{i+1} \bE\|\hw_{-1}-\tx_{-1} \|^2 + (e_1 + e_2) B_3 \mu^3 + \frac{2B_2C_1\delta^2 \sigma_s^2}{\nu^2(1-\beta)}\mu^2 \nnb
%			&\le \rho_1^{i+1} \bE\|\hw_{-1}-\tx_{-1} \|^2 + \frac{4B_2C_1\delta^2 \sigma_s^2}{\nu^2(1-\beta)}\mu^2
%			\label{salkhreu},
%		}
%		where $B_3$ is a constant that upper bounds the vanishing series $(i+1)\rho_1^{i}$, and the last inequality holds because we adjust $\mu$ small enough such that
%		$$\frac{2B_2C_1\delta^2 \sigma_s^2}{\nu^2(1-\beta)}\mu^2 \ge e_2 B_3 \mu^3. $$
	}
%	
%	
%	 	Note that
%		\bqq
%		\alpha \define 1-\mu \nu + 2b_1 \mu^2 = 1-\frac{1}{2}\mu \nu + \left(2b_1 \mu^2-\frac{1}{2}\mu\nu\right). \label{189-1}
%		\eqq
%		It is clear that we can choose a sufficiently small $\mu$ for the last term between brackets to become negative, in which case
%		\be
%		\alpha \leq 1-\frac{1}{2}\mu \nu<1.
%		\ee
%		Let $\rho=1-\mu\nu/2$. It then holds that
%		{\color{black}
%		\be\bE \|\hw_i-\tx_i\|^2\leq
%		\rho \bE\|\hw_{i-1}-\tx_{i-1} \|^2 + b_2 \left( \frac{\delta^2 (1+\gamma^2 + \sigma_s^2) \mu^3}{\nu} + \frac{\delta^4 \mu^3}{\nu(1-\beta)} \right).
%		\label{hw-tx-final bound-3}\ee
%	}
		Recall from the first equation in \eqref{relation btw hw and tx} that for $i=-1$:
		\eq{
			\hw_{-1}-\tx_{-1}  =\frac{1}{1-\beta}(\tw_{-1}-\tx_{-1})-\frac{\beta}{1-\beta}(\tw_{-2}-\tx_{-1}).
		}
		Now, using the assumption that the momentum and standard recursions started from the same initial states,
		$\w_{-2}=\x_{-1}$ and $\w_{-1}=\w_{-2}-\mu_m\grad_w Q(\w_{-2};\d(-1),\u_{-1})$, and recall $\mu=\mu_m/(1-\beta)$, then we have
		\eq{
			\label{hw-tx=0}
			{\hw}_{-1}-\widetilde{\x}_{-1}&= \frac{1}{1-\beta}(\tw_{-1}-\tw_{-2}) =\mu \grad_w Q(\w_{-2};\d(-1),\u_{-1}).
		}
		Therefore, it holds that 
		\eq{
			\bE\|\widehat{\w}_{-1}-\widetilde{\x}_{-1}\|^2=B_4 \mu^2, \label{sdau23}
		}
		where $B_4=\bE\|\grad_w Q(\w_{-2};\d(-1),\u_{-1})\|^2$. {\color{black} Substituting \eqref{sdau23} and \eqref{238dha} into \eqref{salkhreu}, we reach
		\eq{\label{238asdh}
		\bE \|\hw_i-\tx_i\|^2 \le B_4 \rho_1^{i+1} \mu^2 + \frac{B_2 C_1(\delta^2 + \gamma^2) \delta^2(i+1)\rho_1^{i}}{\nu(1-\beta)^2}\mu^3 + \frac{4B_2C_1\delta^2 \sigma_s^2}{\nu^2(1-\beta)}\mu^2.
		}	
%		
%		
%		
%		On the other hand, it follows by iterating (\ref{hw-tx-final bound-3}) that
%		{\color{black}
%		\eq{
%		\bE \|\hw_i-\tx_i\|^2&\leq \ \rho^{i+1}\bE\|\hat{\w}_{-1}-\widetilde{\x}_{-1}\|^2+ \sum_{j=0}^i \rho^j b_2 \left( \frac{\delta^2 (1+\gamma^2 + \sigma_s^2) \mu^3}{\nu} + \frac{\delta^4 \mu^3}{\nu(1-\beta)} \right) \nnb
%		&\le\ c_2 \rho^{i+1}\mu^2 + \frac{b_2}{1-\rho}\left( \frac{\delta^2 (1+\gamma^2 + \sigma_s^2) \mu^3}{\nu} + \frac{\delta^4 \mu^3}{\nu(1-\beta)} \right)\nnb
%		&\le \ c_2 \mu^2 + 2b_2 \left( \frac{\delta^2 (1+\gamma^2 + \sigma_s^2) \mu^2}{\nu^2} + \frac{\delta^4 \mu^2}{\nu^2(1-\beta)} \right)
%%		&\leq \  c_2 \rho^{i+1}\mu^2 +\frac{2b_2\mu^3}{\nu\mu}=c\rho^{i+1}\mu^2+\frac{2b_2\mu^2}{\nu}\le b_3\mu^2. 
%\label{dist-upper-bound-ht-tx}
%		}
%	}
%		for some constant $b_3$. 
Furthermore, using (\ref{relation btw hw and tx}) and $\cw_i = \frac{\tw_i - \tw_{i-1}}{1-\beta}  $ from \eqref{transform} we have
\eq{\label{2389adsh}
\hw_i - \tx_i = (\tw_i - \tx_i ) + \beta \cw_i,
}
which implies that
		\eq{
			\label{2nd bound-20}
			\bE\|\tw_i-\tx_i\|^2 \le 2 \bE\| \hw_i-\tx_i \|^2 + 2\beta^2 \bE\|\cw_{i}\|^2 \le 2 \bE\| \hw_i-\tx_i \|^2 + 2 \bE\|\cw_{i}\|^.
		}
		Substituting \eqref{ash762dsdga2} and \eqref{238asdh} into \eqref{2nd bound-20}, we have
		\eq{\label{238asdhj4}
		\bE\|\tw_i-\tx_i\|^2 \le&\ B_4 \rho_1^{i+1} \mu^2 + \frac{B_2 C_1(\delta^2 + \gamma^2) \delta^2(i+1)\rho_1^{i}}{\nu(1-\beta)^2}\mu^3 + \frac{4B_2C_1\delta^2 \sigma_s^2}{\nu^2(1-\beta)}\mu^2 \nnb
		&\quad\quad + 2 C_1 \left( \frac{(\delta^2 + \gamma^2) \rho_1^{i+1} \mu^2}{(1-\beta)^2} + \frac{\sigma_s^2 \mu^2}{1-\beta} \right)\nnb
%		&\ O\left( \rho_1^{i+1} \mu^2 + \frac{\delta^2 \sigma_s^2 \mu^2}{\nu^2(1-\beta)} + \alpha^{i+1}\mu^2 +  \frac{(\delta^2 + \gamma^2) \rho_1^i \mu^2}{(1-\beta)^2} + \frac{\sigma_s^2 \mu^2}{1-\beta} \right) \nnb
%		=&\ O\left( \alpha^{i+1}\mu^2 + \left[ 1 + \frac{\delta^2 + \gamma^2}{(1-\beta)^2} \right] \rho_1^{i+1} \mu^2 + \frac{\delta^2 \sigma_s^2 \mu^2 }{\nu^2(1-\beta)} \right) \nnb
%		=&\ O\left( \alpha^{i+1}\mu^2 +  \frac{\delta^2 + \gamma^2}{(1-\beta)^2} \rho_1^{i+1} \mu^2 + \frac{\delta^2 \sigma_s^2 \mu^2 }{\nu^2(1-\beta)} \right) \nnb
		=&\ O\left( \frac{\delta^2 + \gamma^2}{(1-\beta)^2} \rho_1^{i+1} \mu^2 +  \frac{\delta^2(\delta^2 + \gamma^2) (i+1)\rho_1^{i+1}}{\nu(1-\beta)^2}\mu^3 + \frac{\delta^2 \sigma_s^2 \mu^2 }{\nu^2(1-\beta)} \right).
		}
%		where the last equality holds because $\alpha < \rho_1$ and there exists some constant $B_5$ such that
%		\eq{\label{79qahl}
%		\alpha^{i+1}\mu^2 < \frac{B_5(\delta^2 + \gamma^2)}{(1-\beta)^2} \rho_1^{i+1} \mu^2.
%		}
	}
%		
%		
%		Since we have shown that $\bE\| \hw_i-\tx_i \|^2 \le b_3 \mu^2$ in \eqref{dist-upper-bound-ht-tx}
%		and $\bE\|\cw_{i}\|^2 \le c\mu^2$ in \eqref{hb-tw-u2-4} for each iteration $i=0,1,2,3\ldots$, and it is obvious that $\bE\|\tw_i-\tx_i\|^2=\bE\|\w_i-\x_i\|^2$, we then arrive at the desired conclusion (\ref{hb-stochastic gradient method-dist-tw-2}).

		\section{Verifying Assumptions \ref{ass:noise-lipschitz} and  \ref{ass: hessian lipschitz} }
		\label{app-feasibility}
		\noindent{\textbf{Least-mean-squares problem.}} Consider first the mean-squares cost \eqref{lms model}. Since in this case $\H_{i-1}=\R_{i-1}=R_u$, we find that Assumption
		\ref{ass: hessian lipschitz} holds automatically. With regards to Assumption \ref{ass:noise-lipschitz}, at any iteration $i$, we have
		\eq{
			\s_i(\w_{i-1})\hspace{-0.8mm}-\hspace{-0.8mm}\s_i(\x_{i-1})=(R_u-\u_i \u_i\tran)(\tw_{i-1} - \tx_{i-1}).
		}
		so that, under the assumption of independent and stationary regression vectors,
		\be
		\bE[\|\s_i(\w_{i-1})-\s_i(\x_{i-1})\|^2|\filt_{i-1}]\le
		\ \xi_1\, \|\tw_{i-1} - \tx_{i-1}\|^2,
		\ee
		where $\xi_1=\bE\|R_u-\u_i \u_i\tran\|^2.$ Similarly,
		\be
		\bE[\|\s_i(\w_{i-1})-\s_i(\x_{i-1})\|^4|\filt_{i-1}]
		\le\ \xi_2\, \|\tw_{i-1} - \tx_{i-1}\|^4,
		\ee
		where $\xi_2=\bE\|R_u-\u_i \u_i\tran\|^4$. Therefore, Assumption \ref{ass:noise-lipschitz} holds.
		
		\vspace{2mm}
		\noindent{\textbf{Regularized logistic regression.}} Consider next the regularized logistic regression risk
		\eq{\label{app-prob-log}
			\ J(w)\define \frac{\rho}{2} \|w\|^2 + \bE \Big\{ \ln \big[ 1+ \exp(-\bgm(i) \h_i \tran w) \big] \Big\},
		}
		where $\h_i\in\mathbb{R}^M$ is a streaming sequence of independent feature vectors with $R_h=\bE \h_i \h_i \tran >0$,  and $\bgm(i)\in \{-1,+1\}$ is a streaming sequence of class labels. We assume the random processes $\{\bgm(i), \h_i\}$ are wide-sense stationary. Moreover,
		$\rho>0$ is a regularization parameter. We first verify the feasibility of Assumption \ref{ass:noise-lipschitz}.
		Note that the approximate gradient vector is given by:
		\eq{
			\widehat{\grad_w  J}(\w) & = \rho \w-\frac{\exp(-\bgm(i) \h_i \tran \w)}{1+ \exp(-\bgm(i) \h_i \tran \w) } \bgm(i) \h_i
		}
		and, hence,
		\eq{
		&\hspace{-5mm} \widehat{\grad_w  J}(\bpsi_{i-1}) - \widehat{\grad_w J}(\x_{i-1}) \nnb
		& \leq \rho \|\bpsi_{i-1} - \x_{i-1}\|+
		 \|\h_i\| \left \Vert \frac{\exp(-\bgm_i \h_i \tran \bpsi_{i-1})}{1+ \exp(-\bgm_i \h_i \tran \bpsi_{i-1}) } - \frac{\exp(-\bgm(i) \h_i \tran \x_{i-1})}{1+ \exp(-\bgm(i) \h_i \tran \x_{i-1}) } \right \Vert \label{app-a6-bound1}
		}
		Note that
		\eq{
			&\hspace{-2cm}\left \Vert \frac{\exp(-\bgm(i) \h_i \tran \bpsi_{i-1})}{1+ \exp(-\bgm(i) \h_i \tran \bpsi_{i-1}) } - \frac{\exp(-\bgm(i) \h_i \tran \x_{i-1})}{1+ \exp(-\bgm(i) \h_i \tran \x_{i-1}) } \right \Vert \nnb
			=&\ \left\Vert \frac{\exp(-\bgm(i) \h_i \tran \bpsi_{i-1}) - \exp(-\bgm(i) \h_i \tran \x_{i-1})}{[1+ \exp(-\bgm(i) \h_i \tran \bpsi_{i-1})] [1+ \exp(-\bgm(i) \h_i \tran \x_{i-1})]} \right\Vert \nnb \le&\ \left\Vert \frac{\exp(-\bgm(i) \h_i \tran \bpsi_{i-1}) - \exp(-\bgm(i) \h_i \tran \x_{i-1})}{\exp(-\bgm(i) \h_i \tran \bpsi_{i-1}) + \exp(-\bgm(i) \h_i \tran \x_{i-1})} \right\Vert \nnb
			= &\  \left\Vert \frac{\exp( \bgm(i) \h_i \tran \frac{\x_{i-1}-\bpsi_{i-1}}{2}) - \exp(-\bgm(i) \h_i \tran\frac{\x_{i-1}-\bpsi_{i-1}}{2})}{\exp(\bgm(i) \h_i \tran \frac{\x_{i-1}-\bpsi_{i-1}}{2}) + \exp(-\bgm(i) \h_i \tran \frac{\x_{i-1}-\bpsi_{i-1}}{2})} \right\Vert \nnb =&\ \left|\tanh\left( \bgm(i) \h_i \tran {(\x_{i-1}-\bpsi_{i-1})}/{2} \right) \right| \le \frac{1}{2}\|\h_i\| \|\bpsi_{i-1}-\x_{i-1}\|. \label{app-a6-bound2}
		}
		where in the last inequality we used the property
		$
			|\tanh(y)| \le |y|,\ \forall y\in\real.
		$
		Substituting \eqref{app-a6-bound2} into \eqref{app-a6-bound1}, we get
		\eq{
			\| \widehat{\grad_w  J}(\bpsi_{i-1}) - \widehat{\grad_w  J}(\x_{i-1})  \| \le \boldsymbol \eta_{1,i}  \|\bpsi_{i-1}-\x_{i-1}\|,
		}
		where $\boldsymbol \eta_{1,i}=\rho+\|\h_i\|^2/2$ is a random variable.
		
		On the other hand, it is shown in Eq. (2.20) of \citep{sayed2014adaptation} that the Hessian matrix $\grad_w^2 J(w)$ is upper bounded by $\delta I_M$, where  
		$\delta=\big(\rho + \lambda_\mathrm{max}(R_h) \big)$. We conclude from Lemma E.3 in the same reference that $\grad_w J(w)$ is Lipschitz continuous with modulus $\delta$, i.e.,
		\eq{
			\| {\grad_w  J}(\bpsi_{i-1}) - {\grad_w  J}(\x_{i-1})  \| \le \delta  \|\bpsi_{i-1}-\x_{i-1}\|.
		}
		Combining these results we get
		\bqq
		\|\s_i(\bpsi_{i\hspace{-0.5mm}-\hspace{-0.5mm}1})\hspace{-0.5mm}-\hspace{-0.5mm}\s_i(\x_{i\hspace{-0.5mm}-\hspace{-0.5mm}1})\| &=& \left\| [\widehat{\grad  J}(\bpsi_{i\hspace{-0.5mm}-\hspace{-0.5mm}1}) \hspace{-0.5mm}-\hspace{-0.5mm} \widehat{\grad  J}(\x_{i\hspace{-0.5mm}-\hspace{-0.5mm}1})]\hspace{-0.5mm}-\hspace{-0.5mm}[{\grad  J}(\bpsi_{i\hspace{-0.5mm}-\hspace{-0.5mm}1}) \hspace{-0.5mm}-\hspace{-0.5mm} {\grad  J}(\x_{i-1}) ]\right\| \nnb
		&\le& \boldsymbol \eta_i \|\bpsi_{i-1}-\x_{i-1}\|.
		\eqq
		where $\boldsymbol \eta_i = \boldsymbol \eta_{1,i} + \delta $ is a random variable. Since the
		$\{\h_i\}$ are independent feature vectors and $\boldsymbol \eta_i$ is only related to $\h_i$,  it follows that 
		\eq{
			\bE [\|\s_i(\bpsi_{i\hspace{-0.4mm}-\hspace{-0.4mm}1})\hspace{-0.8mm}-\hspace{-0.8mm}\s_i(\x_{i\hspace{-0.4mm}-\hspace{-0.4mm}1})\|^2|\filt_{i\hspace{-0.4mm}-\hspace{-0.4mm}1}] & \le \xi_1  \|\bpsi_{i\hspace{-0.4mm}-\hspace{-0.4mm}1}\hspace{-0.8mm}-\hspace{-0.8mm}\x_{i\hspace{-0.4mm}-\hspace{-0.4mm}1}\|^2, \\
			\bE [\|\s_i(\bpsi_{i\hspace{-0.4mm}-\hspace{-0.4mm}1})\hspace{-0.8mm}-\hspace{-0.8mm}\s_i(\x_{i\hspace{-0.4mm}-\hspace{-0.4mm}1})\|^4|\filt_{i\hspace{-0.4mm}-\hspace{-0.4mm}1}] & \le \xi_2  \|\bpsi_{i\hspace{-0.4mm}-\hspace{-0.4mm}1}\hspace{-0.8mm}-\hspace{-0.8mm}\x_{i\hspace{-0.4mm}-\hspace{-0.4mm}1}\|^4,
		}
		where $\xi_1 =\bE {\boldsymbol \eta_i}^2$ and $\xi_2 =\bE {\boldsymbol \eta_i}^4$.
		
		Next we check the feasibility of Assumption \ref{ass: hessian lipschitz}. For simplicity, we write $\bm{\gamma}$ instead of $\bm{\gamma}(i)$.
		It can be verified that for the cost function $J(w)$ in \eqref{app-prob-log}:
		\eq{
			\grad_w^2 J(w)& = \rho I_M + \bE \Big\{ \boldsymbol{ h}_i \h_i\tran \Big( \frac{\exp(-\bgm \h_i\tran w)}{[1+\exp(-\bgm \h_i\tran w)]^2}\Big) \Big\}.
		}
		Now, for any two variables $w_1$ and $w_2$ we have
		\eq{
			\label{app-fea-bound1}
			 &\hspace{-1cm} \|\grad_w^2 J(w_1) - \grad_w^2 J(w_2)\| \nnb
			 & =\  \left\|\bE \Big\{ \hspace{-0.8mm}  \h_i \h_i\tran \hspace{-0.8mm}  \Big( \hspace{-0.5mm} \frac{\exp(-\bgm \h_i\tran w_1)}{[1+\exp(-\bgm \h_i\tran w_1)]^2} \hspace{-0.3mm}  - \hspace{-0.3mm}  \frac{\exp(-\bgm \h_i\tran  w_2)}{[1+\exp(-\bgm \h_i\tran w_2)]^2}\hspace{-0.8mm}   \Big)  \hspace{-0.8mm}   \Big\}
			\hspace{-0.8mm} \right\| \nnb
			& \le\ \bE \left\Vert \h_i \h_i\tran \hspace{-0.8mm} \Big( \hspace{-0.5mm}\frac{\exp(-\bgm \h_i\tran  w_1)}{[1+\exp(-\bgm \h_i\tran  w_1)]^2}  - \frac{\exp(-\bgm \h_i\tran w_2)}{[1+\exp(-\bgm \h_i\tran  w_2)]^2}  \hspace{-0.8mm} \Big)   \hspace{-0.8mm}  \right\Vert \nnb
			& \le \bE \left\{ \hspace{-0.8mm} \left\Vert \h_i \h_i\tran \hspace{-0.8mm}  \right\Vert \hspace{-0.8mm} \left\Vert \hspace{-0.5mm} \frac{\exp(-\bgm \h_i\tran  w_1)}{[1+\exp(-\bgm \h_i\tran  w_1)]^2} \hspace{-0.3mm} - \hspace{-0.3mm}\frac{\exp(-\bgm \h_i\tran  w_2)}{[1+\exp(-\bgm \h_i\tran  w_2)]^2} \hspace{-0.8mm}\right\Vert \hspace{-0.8mm} \right\}
		}
		Let $\x_1=-\bgm \h_i\tran  w_1$ and $\x_2=-\bgm \h_i\tran  w_2$. Then,
		\eq{
			& \left\Vert \frac{\exp(-\bgm \h_i\tran  w_1)}{[1+\exp(-\bgm \h_i\tran  w_1)]^2}  - \frac{\exp(-\bgm \h_i\tran  w_2)}{[1+\exp(-\bgm \h_i\tran  w_2)]^2} \right\Vert \nnb
			=& \left\| \frac{\exp(\x_1)}{[1+\exp(\x_1)]^2}  - \frac{\exp(\x_2)}{[1+\exp(\x_2)]^2} \right\| \nnb
			=& \left\| \frac{\exp(\x_1)[1+\exp(\x_2)]^2-\exp(\x_2)[1+\exp(\x_1)]^2}{[1 + \exp(\x_1)]^2 [1+ \exp(\x_2)]^2}\right\| \nnb
			\overset{\mbox{(a)}}{\le}& \left\| \frac{\exp(-\x_2)-\exp(-\x_1)+\exp(\x_2)-\exp(\x_1)}{2\big(\exp(-\x_2)+\exp(-\x_1)+\exp(\x_2)+\exp(\x_1)\big)}  \right\| \nnb
			\le& \left\|\hspace{-0.7mm} \frac{\exp(-\x_2)\hspace{-1mm}-\hspace{-1mm}\exp(-\x_1)}{2\big(\exp(\hspace{-0.5mm}-\x_2\hspace{-0.5mm})\hspace{-1mm}+\hspace{-1mm}\exp(\hspace{-0.5mm}-\x_1\hspace{-0.5mm})\hspace{-1mm}+\hspace{-1mm}\exp(\hspace{-0.5mm}\x_2\hspace{-0.5mm})\hspace{-1mm}+\hspace{-1mm}\exp(\hspace{-0.5mm}\x_1\hspace{-0.5mm})\big)}  \hspace{-0.7mm}\right\|  \hspace{-1.5mm}+\hspace{-1.5mm} \left\|\hspace{-0.7mm} \frac{\exp(\x_2)\hspace{-0.5mm}-\hspace{-0.5mm}\exp(\x_1)}{2\big(\exp(\hspace{-0.5mm}-\x_2\hspace{-0.5mm})\hspace{-1mm}+\hspace{-1mm}\exp(\hspace{-0.5mm}-\x_1\hspace{-0.5mm})\hspace{-1mm}+\hspace{-1mm}\exp(\hspace{-0.5mm}\x_2\hspace{-0.5mm})\hspace{-1mm}+\hspace{-1mm}\exp(\hspace{-0.5mm}\x_1\hspace{-0.5mm})\big)}  \hspace{-0.7mm}\right\| \nnb
			\le& \left\| \frac{\exp(-\x_2)-\exp(-\x_1)}{2\big(\exp(-\x_2)+\exp(-\x_1)\big)}  \right\| + \left\| \frac{\exp(\x_2)-\exp(\x_1)}{2\big(\exp(\x_2)+\exp(\x_1)\big)}  \right\| \nnb
			\overset{\mbox{(b)}}{=}& \frac{1}{2} \left\| \frac{\exp(-\frac{\x_2-\x_1}{2})-\exp(\frac{\x_2-\x_1}{2})}{\exp(-\frac{\x_2-\x_1}{2})+\exp(\frac{\x_2-\x_1}{2})}  \right\|  + \frac{1}{2}\left\| \frac{\exp(\frac{\x_2-\x_1}{2})-\exp(-\frac{\x_2-\x_1}{2})}{\exp(\frac{\x_2-\x_1}{2})+\exp(-\frac{\x_2-\x_1}{2})}  \right\| \nnb
			=& | \tanh[(\x_2-\x_1)/2] |,
		}
		where (a) holds because of the following two facts:
		\eq{
			&\hspace{-2cm} \exp(\x_1)[1+\exp(\x_2)]^2-\exp(\x_2)[1+\exp(\x_1)]^2 \nnb
			=&\ \exp(\x_1) + \exp(\x_1+2\x_2) - \exp(\x_2) - \exp(\x_2+2\x_1) \nnb
			=&\ \exp(\x_1\hspace{-0.8mm}+\hspace{-0.8mm}\x_2) [\exp(-\x_2)\hspace{-0.8mm}+\hspace{-0.8mm} \exp(\x_2) \hspace{-0.8mm}-\hspace{-0.8mm} \exp(-\x_1) \hspace{-0.8mm}-\hspace{-0.8mm}\exp(\x_1)],
		}
		and
		\eq{
			&\hspace{-2cm} [1 + \exp(\x_1)]^2 [1+ \exp(\x_2)]^2 \nnb =&\ \big(1+2\exp(\x_1) + \exp(2\x_1)\big)\big(1+2\exp(\x_2) + \exp(2\x_2)\big) \nnb
			\ge&\  2\exp(\x_1)\hspace{-0.8mm}+\hspace{-0.8mm}2\exp(\x_2)\hspace{-0.8mm}+\hspace{-0.8mm}2\exp(\x_1+2\x_2) \hspace{-0.8mm}+\hspace{-0.8mm}2\exp(\x_2+2\x_1) \nnb
			=&\ 2 \exp(\x_1\hspace{-0.8mm}+\hspace{-0.8mm}\x_2) [\exp(-\x_2) \hspace{-0.8mm}+\hspace{-0.8mm} \exp(\x_2) \hspace{-0.8mm}+\hspace{-0.8mm} \exp(-\x_1) \hspace{-0.8mm}+\hspace{-0.8mm} \exp(\x_1)].
		}
		In addition, (b) holds if we extract $\exp(-\frac{\x_1+\x_2}{2})$ and $\exp(\frac{\x_1+\x_2}{2})$ from both the denominator and numerator of the first and second terms respectively.
		
		Using the definitions for $\x_1$ and $\x_2$, this last expression gives
		\bqq
		|\tanh[(\x_2-\x_1)/2]| = \left| \tanh\left(\frac{1}{2} \bgm \h_i\tran  (w_2 - w_1) \right) \right| \le \frac{1}{2}\|\h_i\| \|w_2-w_1\|. \label{app-fea-bound2}
		\eqq
		Substituting \eqref{app-fea-bound2} into \eqref{app-fea-bound1}, we obtain
		$
			\|\grad_w ^2 J(w_1) - \grad_w^2 J(w_2)\| \le \kappa \|w_1 - w_2\|,
		$
		where $\kappa=\bE\|\h_i\h_i\tran \|\|\h_i\|/2.$ Therefore, Assumption \ref{ass: hessian lipschitz} holds.

		\section{Proof of Lemma \ref{lm: hw-tx-4-bound}}
		\label{app-lm-hw-tx-4-bound}
		Referring to relation \eqref{ge-dist-hw-tx} and apply the inequality \eqref{a+b 4 bound-asdluy}, we reach
		%along with
		% Jensen's inequality, to obtain in a derivation similar to the one that led to
		% (\ref{cw-4-bound})--(\ref{subbound-1}):
		\eq{
			\label{188}
			& \bE[\|\hw_i-\tx_i\|^4|\filt_{i-1}] \nnb
			=& \|(\I_M-\mu\H_{i-1})(\hw_{i-1}-\tx_{i-1})+\mu(\R_{i-1}-\H_{i-1})\tx_{i-1}  +\mu\beta^\prime \H_{i\hspace{-0.3mm}-\hspace{-0.3mm}1}\cw_{i-1} \|^4 \nnb & \ + 3\mu^4 \bE[\|\s_i(\hspace{-0.4mm}\bpsi_{i\hspace{-0.3mm}-\hspace{-0.3mm}1}\hspace{-0.4mm})\hspace{-0.5mm}-\hspace{-0.5mm}\s_i(\hspace{-0.4mm}\x_{i\hspace{-0.3mm}-\hspace{-0.3mm}1}\hspace{-0.4mm})\|^4|\filt_{i\hspace{-0.3mm}-\hspace{-0.3mm}1}]  + 8\mu^2 \|(\hspace{-0.4mm}I_M\hspace{-0.5mm}-\hspace{-0.5mm}\mu\H_{i\hspace{-0.3mm}-\hspace{-0.3mm}1}\hspace{-0.4mm})(\hspace{-0.4mm}\hw_{i\hspace{-0.3mm}-\hspace{-0.3mm}1}\hspace{-0.5mm}-\hspace{-0.5mm}\tx_{i\hspace{-0.3mm}-\hspace{-0.3mm}1}\hspace{-0.4mm})\hspace{-0.5mm} \nnb
			&\ +\mu(\hspace{-0.4mm}\R_{i\hspace{-0.3mm}-\hspace{-0.3mm}1}\hspace{-0.7mm}-\hspace{-0.7mm}\H_{i\hspace{-0.3mm}-\hspace{-0.3mm}1}\hspace{-0.4mm})\tx_{i\hspace{-0.3mm}-\hspace{-0.3mm}1} + \mu\beta^\prime \H_{i-1}\cw_{i-1}\|^2 \bE[\|\s_i(\hspace{-0.4mm}\bpsi_{i\hspace{-0.3mm}-\hspace{-0.3mm}1}\hspace{-0.4mm})\hspace{-0.5mm}-\hspace{-0.5mm}\s_i(\hspace{-0.4mm}\x_{i\hspace{-0.3mm}-\hspace{-0.3mm}1}\hspace{-0.4mm})\|^2|\filt_{i\hspace{-0.3mm}-\hspace{-0.3mm}1}] \nnb
			\overset{\mbox{(a)}}{\le}& \frac{1}{(1\hspace{-0.8mm}-\hspace{-0.8mm}t)^3}\|(I_M\hspace{-0.8mm}-\hspace{-0.8mm}\mu\H_{i\hspace{-0.3mm}-\hspace{-0.3mm}1})(\hw_{i\hspace{-0.3mm}-\hspace{-0.3mm}1}\hspace{-0.5mm}-\hspace{-0.5mm}\tx_{i\hspace{-0.3mm}-\hspace{-0.3mm}1})\|^4  \hspace{-0.8mm}+\hspace{-0.8mm}\frac{8\mu^4}{t^3}\|(\R_{i\hspace{-0.3mm}-\hspace{-0.3mm}1}\hspace{-0.5mm}-\hspace{-0.5mm}\H_{i\hspace{-0.3mm}-\hspace{-0.3mm}1})\tx_{i\hspace{-0.3mm}-\hspace{-0.3mm}1}\|^4 \hspace{-0.8mm}+\hspace{-0.8mm} \frac{8\mu^4\beta^{\prime4}}{t^3}\|\H_{i\hspace{-0.3mm}-\hspace{-0.3mm}1}\cw_{i\hspace{-0.3mm}-\hspace{-0.3mm}1}\|^4 \nnb
			&\ + 3\mu^4  \bE[\|\s_i(\bpsi_{i\hspace{-0.3mm}-\hspace{-0.3mm}1})\hspace{-0.5mm}-\hspace{-0.5mm}\s_i(\x_{i\hspace{-0.3mm}-\hspace{-0.3mm}1})\|^4|\filt_{i\hspace{-0.3mm}-\hspace{-0.3mm}1}]  + 8\mu^2 \left(\frac{1}{1-t} \|(I_M\hspace{-0.5mm}-\hspace{-0.5mm}\mu\H_{i\hspace{-0.3mm}-\hspace{-0.3mm}1})(\hw_{i\hspace{-0.3mm}-\hspace{-0.3mm}1}\hspace{-0.5mm}-\hspace{-0.5mm}\tx_{i\hspace{-0.3mm}-\hspace{-0.3mm}1})\|^2  \right. \nnb 
			&\ \left. +\frac{2\mu^2}{t}\|(\R_{i\hspace{-0.3mm}-\hspace{-0.3mm}1}\hspace{-0.7mm}-\hspace{-0.7mm}\H_{i\hspace{-0.3mm}-\hspace{-0.3mm}1})\tx_{i\hspace{-0.3mm}-\hspace{-0.3mm}1}\|^2 + \frac{2\mu^2\beta^{\prime 2}}{t}\|\H_{i-1}\cw_{i-1}\|^2 \right)  \bE[\|\s_i(\bpsi_{i-1})-\s_i(\x_{i-1})\|^2|\filt_{i-1}] \nnb
			\overset{\mbox{(b)}}{\le} & (1-\mu \nu) \|\hw_{i-1}-\tx_{i-1}\|^4 + \frac{8\mu}{\nu^3} \|(\R_{i-1}-\H_{i-1})\tx_{i-1}\|^4 +\hspace{-0.5mm} \frac{8\mu\beta^{\prime 4}\delta^4}{ \nu^3} \|\cw_{i\hspace{-0.3mm}-\hspace{-0.3mm}1}\|^4 \nnb
			&\ + 3\mu^4  \bE[\|\s_i(\bpsi_{i\hspace{-0.3mm}-\hspace{-0.3mm}1})\hspace{-0.5mm}-\hspace{-0.5mm}\s_i(\x_{i\hspace{-0.3mm}-\hspace{-0.3mm}1})\|^4|\filt_{i\hspace{-0.3mm}-\hspace{-0.3mm}1}] + 8\mu^2\hspace{-0.5mm} \Big(\hspace{-1mm}(1\hspace{-0.5mm}-\hspace{-0.5mm}\mu\nu) \|\hw_{i\hspace{-0.3mm}-\hspace{-0.3mm}1}\hspace{-0.5mm}-\hspace{-0.5mm}\tx_{i\hspace{-0.3mm}-\hspace{-0.3mm}1}\|^2\hspace{-0.8mm} \nnb 
			&\  +\frac{2\mu}{\nu}\|(\R_{i\hspace{-0.3mm}-\hspace{-0.3mm}1}\hspace{-0.7mm}-\hspace{-0.7mm}\H_{i\hspace{-0.3mm}-\hspace{-0.3mm}1})\tx_{i\hspace{-0.3mm}-\hspace{-0.3mm}1}\|^2    + \frac{2\mu \beta^{\prime 2} \delta^2}{\nu}\|\cw_{i-1}\|^2 \Big) \bE[\|\s_i(\bpsi_{i-1})\hspace{-0.5mm}-\hspace{-0.5mm}\s_i(\x_{i-1})\|^2|\filt_{i-1}],
			%& + 8\mu^2 \Big[(1-\mu \nu)\|\hw_{i-1}-\tx_{i-1}\|^2  + \frac{2\mu \beta^{\prime 2}\delta^2}{\nu} \|\cw_{i-1}\|^2 \nnb
			%& + \frac{2\mu}{\nu}\|(\R_{i-1}\hspace{-0.8mm}-\hspace{-0.8mm}\H_{i-1})\tx_{i-1}\|^2\hspace{-0.3mm}\Big]\hspace{-1mm} \nnb
			%& \quad\quad  \boldmath \cdot \hspace{-1mm} \Big[ \bE\|\s_i(\bpsi_{i-1})\hspace{-0.8mm}-\hspace{-0.8mm}\s_i(\x_{i-1})\|^2|\filt_{i-1}\hspace{-0.3mm} \Big]
		}
		where (a) holds because the facts that for any $a,b,c\in \real^m$,
		\eq{\|a+b+c\|^4&=\|(1-t)\frac{1}{1-t}a + t\frac{1}{t}(b+c)\|^4 \nnb
			&\le (1-t)\|\frac{1}{1-t}a\|^4 + t\|\frac{1}{t}(b+c)\|^4 = \frac{1}{(1-t)^3}\|a\|^4 + \frac{1}{t^3}\|b+c\|^4 \nnb
			&\le \frac{1}{(1-t)^3}\|a\|^4 + \frac{8}{t^3}\|b\|^4 + \frac{8}{t^3}\|c\|^4,
		}
		and
		\eq{\|a+b+c\|^2&=\|(1-t)\frac{1}{1-t}a + t\frac{1}{t}(b+c)\|^2 \nnb
			&\le (1-t)\|\frac{1}{1-t}a\|^2 + t\|\frac{1}{t}(b+c)\|^2 = \frac{1}{1-t}\|a\|^2 + \frac{1}{t}\|b+c\|^2 \nnb
			&\le \frac{1}{1-t}\|a\|^2 + \frac{2}{t}\|b\|^2 + \frac{2}{t}\|c\|^2,
		}
		In addition, (b) holds by choosing $t=\mu\nu$.
		
		To further simplify inequality \eqref{188}, we first note that
		\eq{
			\|\R_{i-1}-\H_{i-1}\|^2 & \hspace{-1mm} \le 2(\|\R_{i-1}\|^2 + \|\H_{i-1}\|^2) \le 4\delta^2, \label{208-1}\\
			\|\R_{i-1}-\H_{i-1}\|^4 & \hspace{-1mm}\le 8(\|\R_{i-1}\|^4 + \|\H_{i-1}\|^4) \le 16\delta^4.
		}
		As a result, we have
		\eq{
			\frac{8\mu}{\nu^3} \|(\R_{i-1}-\H_{i-1})\tx_{i-1}\|^4 \le a_1 \mu \|\tx_{i-1}\|^4,\quad \frac{2\mu}{\nu} \|(\R_{i-1}-\H_{i-1})\tx_{i-1}\|^2 \le a_2 \mu \|\tx_{i-1}\|^2,
		}
		where we define
		\eq{\label{2348a7c}
		a_1\define 128\delta^4/\nu^3, \quad a_2 \define 8\delta^2/\nu.
		}
		On the other hand, from conditions (\ref{cond-1a})--(\ref{cond-2a}), we have
		\eq{
			\hspace{-10mm}3\mu^4 \bE[\|\s_i(\bpsi_{i-1})-\s_i(\x_{i-1})\|^4|\filt_{i-1}] & \le 3\xi_2 \mu^4 \|\tpsi_{i-1} - \tx_{i-1}\|^4, \label{212-1} \\
			\hspace{-10mm}8\mu^2 \bE[\|\s_i(\bpsi_{i-1})-\s_i(\x_{i-1})\|^2|\filt_{i-1}]
			& \le 8 \xi_1 \mu^2 \|\tpsi_{i-1} - \tx_{i-1}\|^2, \label{app-ass-7}
		}
		In addition, from \eqref{nag:relation btw hw and tx} we get
		\eq{
			\|\tpsi_i-\tx_i\|^2 \le 2 \| \hw_i-\tx_i \|^2 + 2\beta^{\prime2}\|\cw_{i}\|^2, \quad
			\|\tpsi_i-\tx_i\|^4 \le 8 \| \hw_i-\tx_i \|^4 +  8\beta^{\prime4} \|\cw_{i}\|^4. \label{215-1}
		}
		Combining \eqref{212-1}--\eqref{215-1}, we have
		\bqq
		\hspace{-1.5cm} 3\mu^4 \bE[\| \s_i(\bpsi_{i-1})-\s_i(\x_{i-1})\|^4|\filt_{i-1}] 
		&\le& 24 \xi_2\mu^4 \| \hw_{i-1}-\tx_{i-1} \|^4+ 24 \xi_2\mu^4 \|\cw_{i-1}\|^4, \\
		\hspace{-1.5cm} 8\mu^2 \bE[\| \s_i(\bpsi_{i-1})-\s_i(\x_{i-1})\|^2|\filt_{i-1}] 
		&\le& 16\xi_1 \mu^2 \| \hw_{i-1}-\tx_{i-1} \|^2+ 16 \xi_1 \mu^2 \|\cw_{i-1}\|^2, \label{217-1}
		\eqq
		{\color{black}
		In this way, relation \eqref{188} becomes
		\eq{
			\label{195}
			& \bE[\|\hw_i-\tx_i\|^4|\filt_{i-1}] \nnb
			\le &\; (1-\mu \nu) \|\hw_{i-1}-\tx_{i-1}\|^4 + a_1 \mu \|\tx_{i-1}\|^4 + a_3 \mu \|\cw_{i-1}\|^4 +  \big[(1-\mu \nu) \|\hw_{i-1} -\tx_{i-1}\|^2 \nnb
			& + a_4 \mu \|\cw_{i-1}\|^2+ a_2 \mu \|\tx_{i-1}\|^2 \big]  \big[ a_5\mu^2 \| \hw_{i-1}-\tx_{i-1} \|^2+ a_5\mu^2 \|\cw_{i-1}\|^2 \big] \nnb
			&+   a_6\mu^4 \| \hw_{i-1}-\tx_{i-1} \|^4+ a_6\mu^4 \|\cw_{i-1}\|^4 \nnb
			\le &\; (1-\mu \nu) \|\hw_{i-1}-\tx_{i-1}\|^4 + a_1 \mu \|\tx_{i-1}\|^4 + a_3 \mu \|\cw_{i-1}\|^4 + a_5(1-\mu\nu)\mu^2 \|\hw_{i-1} - \tx_{i-1}\|^4\nnb
			&+ a_5\mu^2(1-\mu \nu + a_4 \mu)\|\hw_{i-1} - \tx_{i-1}\|^2\|\cw_{i-1}\|^2 + a_4a_5\mu^3\|\cw_{i-1}\|^4 \nnb
			&+ a_2a_5\mu^3 \|\tx_{i-1}\|^2\|\hw_{i-1}-\tx_{i-1}\|^2+ a_2a_5\mu^3 \|\tx_{i-1}\|^2\|\cw_{i-1}\|^2 \nnb
			& + a_6\mu^4 \| \hw_{i-1}-\tx_{i-1} \|^4+ a_6\mu^4 \|\cw_{i-1}\|^4 \nnb
			\overset{(a)}{\le} &\;(1-\mu \nu) \|\hw_{i-1}-\tx_{i-1}\|^4 + a_1 \mu \|\tx_{i-1}\|^4 + a_3 \mu \|\cw_{i-1}\|^4 + a_5(1-\mu\nu)\mu^2 \|\hw_{i-1} - \tx_{i-1}\|^4\nnb
			&+ a_5\mu^2(1-\mu \nu + a_4 \mu)\|\hw_{i-1} - \tx_{i-1}\|^4 + a_5\mu^2(1-\mu \nu + a_4 \mu) \|\cw_{i-1}\|^4 + a_4a_5\mu^3\|\cw_{i-1}\|^4 \nnb
			&+ a_2a_5\mu^3 \|\tx_{i-1}\|^4+a_2a_5\mu^3\|\hw_{i-1}-\tx_{i-1}\|^4+ a_2a_5\mu^3 \|\tx_{i-1}\|^4+a_2a_5\mu^3\|\cw_{i-1}\|^4 \nnb
			& + a_6\mu^4 \| \hw_{i-1}-\tx_{i-1} \|^4+ a_6\mu^4 \|\cw_{i-1}\|^4 \nnb
			\le&\; (1-\frac{\mu\nu}{2})\|\hw_{i-1}-\tx_{i-1}\|^4 + 2a_1\mu\|\tx_{i-1}\|^4 + 2a_3\mu \|\cw_{i-1}\|^4.
		}
		where we define
		\eq{\label{sdhy723}
		a_3\define \frac{8\beta^{\prime4}\delta^4}{\nu^4}, \quad a_4 \define \frac{2\beta^{\prime 2}\delta^2}{\nu}, \quad a_5 \define 16\xi_1,\quad a_6 \define 24\xi_2.
		}
		Taking expectations over $\filt_{i-1}$ for both sides of \eqref{195}, we have 
		\eq{\label{2389asdn}
		\bE\|\hw_i-\tx_i\|^4 &\le (1-\frac{\mu\nu}{2})\bE\|\hw_{i-1}-\tx_{i-1}\|^4 + 2a_1\mu\bE\|\tx_{i-1}\|^4 + 2a_3\mu\bE\|\cw_{i-1}\|^4 \nnb
		&= \rho_1 \bE\|\hw_{i-1}-\tx_{i-1}\|^4 + 2a_1\mu\bE\|\tx_{i-1}\|^4 + 2a_3\mu\bE\|\cw_{i-1}\|^4.
		}
		Now recall from (\ref{lm: x4-evolve-3}) that
		\eq{\label{234790jn}
			\bE\|\tx_i\|^4 \le \rho^{i+1}\bE\|\tx_{-1}\|^4  + A_1 \sigma_s^2 (i+1)\rho^{i+1}\mu^2 + \frac{A_2 \sigma_s^4\mu^2}{\nu^2},
		}
		where $\rho = 1 - \mu\nu$ and $A_1$ and $A_2$ are some constants. On the other hand, recall from \eqref{238asdbmnc} that 
		\eq{\label{238asdbmnc-237809m}
			\bE\|\cw_i\|^4 \le  \frac{B_1 \gamma^2\rho_2^{i+1}}{1-\beta}\mu^2 + \frac{B_1 \sigma_s^2(\delta^2 + \gamma^2)(i+1)\rho_2^{i+1}}{(1-\beta)^3}\mu^4 + \frac{B_1 [(\gamma^2 + \nu^2)\sigma_s^4 + \sigma_{s,4}^4 \nu^2]}{(1-\beta)^2\nu^2}\mu^4,
		}
		where $\rho_2 = 1 - \mu\nu/4$.
		Besides, we denote $\rho_1 = 1 - \mu\nu/2$ and clearly $\rho < \rho_1 < \rho_2$. Substituting \eqref{234790jn} and \eqref{238asdbmnc-237809m} into \eqref{2389asdn} we reach
	\eq{\label{27abdm9}
	\bE\|\hw_i - \tx_i\|^4 \ \le&\ \rho_1 \bE\|\hw_{i-1} - \tx_{i-1}\|^4 + a_7 \rho^{i} \mu + a_8 i \rho^{i} \mu^3 + a_9 \mu^3  + a_{10} \rho_2^{i} \mu^3 + a_{11} i \rho_2^{i} \mu^5 + a_{12}\mu^5\nnb
	\overset{(a)}{\le}&\ \rho_1 \bE\|\hw_{i-1} - \tx_{i-1}\|^4 + a_7 \rho_1^{i} \mu + a_8 i \rho_2^{i} \mu^3   + a_{10} \rho_2^{i} \mu^3 + a_{11} i \rho_2^{i} \mu^5 + 2 a_9 \mu^3,
	}
	where the constants are defined as
	\eq{\label{023hadsbn}
	& a_7 \define 2a_1 \bE\|\tx_{-1}\|^4, \quad a_8 \define 2 A_1 a_1 \sigma_s^2,\quad a_9 \define \frac{2A_2a_1\sigma_s^4}{\nu^2},\quad a_{10} \define \frac{2B_1a_3\gamma^2}{1-\beta} \nnb
	& a_{11} \define \frac{2 B_1 a_3 \sigma_s^2(\delta^2 + \gamma^2)}{(1-\beta)^3}, \quad a_{12} \define \frac{2B_1 a_3[(\gamma^2 + \nu^2) \sigma_s^4 + \sigma_{s,4}^4 \nu^2]}{(1-\beta)^2\nu^2}.
	}
	The inequality (a) holds because $\mu$ is chosen small enough such that $a_9 \mu^3 > a_{12} \mu^5$. Now we continue iterating \eqref{27abdm9} and get
	\eq{\label{27abdm9-237anl}
		&\hspace{-5mm} \bE\|\hw_i - \tx_i\|^4  \nnb
		\le&\ \rho_1^{i+1} \bE\|\hw_{-1} - \tx_{-1}\|^4 + a_7(i+1)\rho_1^i\mu + a_8 \rho_2^i \mu^3 \sum_{k=0}^{i}(i-k)\left(\frac{\rho_1}{\rho_2}\right)^k \nnb
		&\ + a_{10} \rho_2^{i}\mu^3 \sum_{k=0}^{i} \left(\frac{\rho_1}{\rho_2}\right)^k + a_{11}\rho_2^{i}\mu^5 \sum_{k=0}^{i}(i-k)\left(\frac{\rho_1}{\rho_2}\right)^k + \frac{2a_9\mu^3}{1-\rho_2}.
	}
	Note that $\rho_1 < \rho_2$, we then have
	\eq{\label{0872ha}
	\sum_{k=0}^i \left(\frac{\rho_1}{\rho_2}\right)^k \le \sum_{k=0}^\infty \left(\frac{\rho_1}{\rho_2}\right)^k \le \frac{\rho_2}{\rho_2 - \rho_1} = \frac{4 - \mu\nu}{\mu\nu} \le \frac{B_2}{\mu\nu},
	}
	where $B_2$ is some constant. Meanwhile, it also holds that
	\eq{\label{23790m}
	\sum_{k=0}^{i}(i-k)\left(\frac{\rho_1}{\rho_2}\right)^k \le i \sum_{k=0}^{i}\left(\frac{\rho_1}{\rho_2}\right)^k \le \frac{i B_2}{\mu \nu}.
	}
	Substituting \eqref{0872ha} and \eqref{23790m} into \eqref{27abdm9-237anl}, we get
	\eq{\label{27abdm9-237anl-0asdhjk}
		\bE\|\hw_i - \tx_i\|^4  \le&\ \rho_1^{i+1} \bE\|\hw_{-1} - \tx_{-1}\|^4 + a_7(i+1)\rho_1^i\mu + \frac{a_8 B_2 i \rho_2^i \mu^2}{\nu}  \nnb
		&\quad + \frac{a_{10}B_2 \rho_2^{i}\mu^2}{\nu} + \frac{a_{11} i \rho_2^{i}\mu^4}{\nu} + \frac{4a_9\mu^2}{\nu}.
	}
	Recall from \eqref{hw-tx=0} that $\hw_{-1}-\tx_{-1}=\mu \grad_wQ(\w_{-2};\btheta_{-1})$. Then it holds that
	\eq{
		\bE\|\hw_{-1}-\tx_{-1}\|^4=B_3 \mu^4,
	}
	where $B_3=\bE\|\grad_wQ(\w_{-2};\btheta_{-1})\|^4$. With this fact, expressions \eqref{27abdm9-237anl-0asdhjk} becomes 
	\eq{\label{27abdm9-237anl-0asdhjk-2367890mab}
		&\ \bE\|\hw_i - \tx_i\|^4  \nnb
		\le&\ B_3 \rho_1^{i+1} \mu^4  + a_7(i+1)\rho_1^i\mu + \frac{a_8 B_2 i \rho_2^i \mu^2}{\nu}  + \frac{a_{10}B_2 \rho_2^{i}\mu^2}{\nu} + \frac{a_{11} i \rho_2^{i}\mu^4}{\nu} + \frac{4a_9\mu^2}{\nu} \nnb
		\le&\ B_3 \rho_2^{i+1} \mu^4  + a_7(i+1)\rho_2^i\mu + \frac{a_8 B_2 i \rho_2^i \mu^2}{\nu}   + \frac{a_{10}B_2 \rho_2^{i}\mu^2}{\nu} + \frac{a_{11} i \rho_2^{i}\mu^4}{\nu} + \frac{4a_9\mu^2}{\nu}.
	}
	Furthermore, recall from \eqref{215-1} that 
	\eq{\label{nclai823}
	\bE\|\tpsi_i-\tx_i\|^4 \le 8 \bE\| \hw_i-\tx_i \|^4 +  8\bE \|\cw_{i}\|^4.
	}
	and recall the upper bound of $\bE\|\cw_i\|^4$ in \eqref{238asdbmnc-237809m}. With the definition of all constants we finally reach
	\eq{\label{2379anmco}
	\bE\|\tpsi_i-\tx_i\|^4 = O\left( \frac{\delta^4 i \rho_2^i\mu}{\nu^3} + \frac{\delta^4(\sigma_s^2+\gamma^2) i \rho_2^i\mu^2}{(1-\beta)\nu^5} + \frac{\delta^4(\delta^2 + \gamma^2)\sigma_s^2 i  \rho_2^i\mu^4}{(1-\beta)^3\nu} + \frac{\delta^4\sigma_s^4}{\nu^6}\mu^2\right).
	}
	To further simplify the notation, we notice that when $\mu$ is sufficiently small it holds that
	\eq{\label{28adnrj}
	&\ \frac{\delta^4 i \rho_2^i\mu}{\nu^3} + \frac{\delta^4(\sigma_s^2+\gamma^2) i \rho_2^i\mu^2}{(1-\beta)\nu^5} + \frac{\delta^4(\delta^2 + \gamma^2)\sigma_s^2 i  \rho_2^i\mu^4}{(1-\beta)^3\nu} \nnb
	=&\ i\rho_2^i\mu \left(\frac{\delta^4 }{\nu^3} + \frac{\delta^4(\sigma_s^2+\gamma^2)\mu}{(1-\beta)\nu^5} + \frac{\delta^4(\delta^2 + \gamma^2)\sigma_s^2 \mu^3}{(1-\beta)^3\nu} \right) \nnb
	\le&\ \frac{2 \delta^4i\rho_2^i\mu}{\nu^3} \le \frac{B_4 \delta^4(i+1)\rho_2^{i+1}\mu}{\nu^3}
	}
	for some constant $B_4$. As a result, we obtain
	\eq{\label{2379anmco-237898}
		\bE\|\tpsi_i-\tx_i\|^4 = O\left( \frac{\delta^4 (i+1) \rho_2^{i+1}\mu}{\nu^3} + \frac{\delta^4\sigma_s^4}{\nu^6}\mu^2\right).
	}
}

		\section{Proof of Lemma~\ref{lm: hw-tx-4-bound-hessian}}
		\label{Hessian diff-app}
		Under Assumption \ref{ass: hessian lipschitz}, we have
		\eq{
			&\hspace{-1.2cm} \|\H_{i-1} - \R_{i-1}\| \nnb
			=&  \left\| \int_0^1 \grad^2_w J(w^o - r \tpsi_{i-1})dr - \int_0^1 \grad^2_w J(w^o - r \tx_{i-1})dr  \right\|  \nnb
			\le& \int_0^1\| \grad^2_w J(w^o - r \tpsi_{i-1}) -  \grad^2_w J(w^o - r \tx_{i-1}) \| dr\nnb
			\le & \int_0^1 \kappa r \| \tpsi_{i-1} -   \tx_{i-1} \| dr = \frac{\kappa}{2}  \| \tpsi_{i-1} -   \tx_{i-1} \|.
		}
		As a result, it holds that
		\eq{
			\bE \| \R_{i-1}-\H_{i-1} \|^4 \le \bar{\kappa}\; \bE \|\tpsi_{i-1} - \tx_{i-1} \|^4,
		}
		where $\bar{\kappa}=\kappa^{4}/16$. {\color{black}Using \eqref{23789asdnm}, we reach the desired bounds shown in \eqref{23789asdnm-23bgasd}.
		}
%		\eq{\label{bound-e-h}
%			\bE \| \R_{i-1}-\H_{i-1} \|^4 \le C\mu,\ \ \forall i=0,1,2,\ldots.
%		}
%		and
%		\eq{
%			\label{limsup e-h bound}
%			\limsup_{i\rightarrow \infty} \| \R_{i-1}-\H_{i-1} \|^4  = O(\mu^2).
%		}

		\section{Proof of Theorem \ref{thm:ge-dist}}
		\label{thm-ge-equi-app}
		For \eqref{cond-1a} in Assumption \ref{ass:noise-lipschitz}, if we take expectation over $\filt_{i-1}$ of both sides, it holds that
		\eq{
			\bE \|\s_i(\bpsi_{i-1})\hspace{-0.8mm}-\hspace{-0.8mm}\s_i(\x_{i-1})\|^2 & \le \xi_1 \bE\|\bpsi_{i-1}\hspace{-0.8mm}-\hspace{-0.8mm}\x_{i-1}\|^2. \label{cond-1a-320} 
		}
		Combining the above fact and inequalities \eqref{ge-dist-hw-tx-e-2}--\eqref{bound-72}, we get
		\eq{
			&\hspace{-8mm} \bE\|\hw_i-\tx_i\|^2 \nnb
			\le&\ (1-\mu \nu)\bE \| \hw_{i-1}-\tx_{i-1}\|^2 + r_1 \mu \bE\|\cw_{i-1} \|^2 \nnb
			&\ + r_2 \mu \sqrt{\bE \| \R_{i-1}-\H_{i-1} \|^4 \bE\|\tx_{i-1}\|^4}  + \xi_1\mu^2 \bE\| \tpsi_{i-1}-\tx_{i-1}\|^2,
			\label{ge-dist-hw-tx-e-2-89}
		}
		{\color{black}where the constants are defined as  
		\eq{\label{238adnk}
		r_1 \define \frac{2\beta^{\prime 2}\delta^2}{\nu}, \quad r_2 \define \frac{2}{\nu}.
		}}Likewise, from \eqref{2nd bound-2} we have
		\eq{
			\label{2nd bound-2-90}
			\bE\|\tpsi_i-\tx_i\|^2 \le 2 \bE\| \hw_i-\tx_i \|^2 + 2 \bE\|\cw_{i}\|^2.
		}
		Substituting the above inequality along with
		\eqref{23789asdnm-23bgasd} into \eqref{ge-dist-hw-tx-e-2-89} gives
		{\color{black}
		\eq{
			\label{ge-dist-hw-tx-e-2-102}
			&\ \bE\|\hw_i-\tx_i\|^2 \nnb
			\le&\ (1\hspace{-0.8mm}-\hspace{-0.8mm}\mu \nu \hspace{-0.8mm}+\hspace{-0.8mm} 2\xi_1 \mu^2 )\bE \| \hw_{i-1}\hspace{-0.8mm}-\hspace{-0.8mm}\tx_{i-1}\|^2 \hspace{-0.8mm}+\hspace{-0.8mm} (r_1 \mu \hspace{-0.8mm}+\hspace{-0.8mm} 2 \xi_1 \mu^2) \bE\|\cw_{i-1} \|^2 + \left[\sqrt{i} r_3 \rho_2^{i/2} \mu^{3/2} + r_4 \mu^2 \right] \sqrt{\bE\|\tx_{i-1}\|^4} \nnb
			\le&\ \left(1\hspace{-0.8mm}-\hspace{-0.8mm}\frac{\mu \nu}{2} \right)\bE \| \hw_{i-1}\hspace{-0.8mm}-\hspace{-0.8mm}\tx_{i-1}\|^2 \hspace{-0.8mm}+\hspace{-0.8mm} 2 r_1 \mu \bE\|\cw_{i-1} \|^2 + \left[r_3 \sqrt{i} \rho_2^{i/2} \mu^{3/2} + r_4 \mu^2 \right] \sqrt{\bE\|\tx_{i-1}\|^4} \nnb
			=&\ \rho_1 \bE \| \hw_{i-1}\hspace{-0.8mm}-\hspace{-0.8mm}\tx_{i-1}\|^2 \hspace{-0.8mm}+\hspace{-0.8mm} 2 r_1 \mu \bE\|\cw_{i-1} \|^2 + \left[r_3 \sqrt{i} \rho_2^{i/2} \mu^{3/2} + r_4 \mu^2 \right] \sqrt{\bE\|\tx_{i-1}\|^4}
%			 r_3 \mu^{3/2} \sqrt{\bE\|\tx_{i-1}\|^4},
		}
		where the constants are defined as
		\eq{\label{0abhnk237ads}
		& r_3\define \frac{r_2 \delta^2}{\nu^{3/2}}, \quad r_4\define \frac{r_2 \delta^2\sigma_s^2}{\nu^3}.
		}
		Recall the upper bound of $\bE\|\cw_i\|^2$ in \eqref{hb-tw-u2} that
		\eq{\label{237nam}
		\bE\|\cw_i\|^2 \le&\ \frac{C_2 (\delta^2 + \gamma^2)\rho_1^i \mu^2}{(1-\beta)^2} + \frac{C_1 \sigma_s^2 \mu^2}{1-\beta}
%		\nnb
%		\le&\  \frac{C_1 (\delta^2 + \gamma^2)\rho_1^i \mu^2}{(1-\beta)^2} + \frac{C_1 \sigma_s^2 \mu^2}{1-\beta}\nnb
%		\le&\ \frac{C_2(\delta^2 + \gamma^2)\rho_1^i\mu^2}{(1-\beta)^2} + \frac{C_1 \sigma_s^2 \mu^2}{1-\beta}
		}
		where $\alpha = 1 - \epsilon/2 < \rho_1$, and $C_1$ and $C_2$ are some constants. Substituting \eqref{237nam} into \eqref{ge-dist-hw-tx-e-2-102}, we have
		\eq{\label{lkn6789}
		&\hspace{-5mm} \bE\|\hw_i - \tx_i\|^2 \nnb
		\le&\ \rho_1 \bE \| \hw_{i-1} - \tx_{i-1}\|^2 + (r_5 \rho_1^{i} \mu^3 + r_6 \mu^3)+ \left[2r_3 \sqrt{i} \rho_2^{i/2} \mu^{3/2} + r_6 \mu^2 \right] \sqrt{\bE\|\tx_{i-1}\|^4},
		}
		where the constants are defined as
		\eq{\label{nas;lic8}
		r_5 \define \frac{2 C_2 r_1 (\delta^2 + \gamma^2)}{(1-\beta)^2},\quad r_6 \define \frac{2C_1 r_1 \sigma_s^2}{1-\beta}.
		}
		Next, using (\ref{lm: x4-evolve-3}) we have
		\eq{
			\label{lm: x4-evolve-3-92}
			\sqrt{\bE\|\tx_i\|^4} \le&\  \sqrt{\rho^{i+1}\bE\|\tx_{-1}\|^4  + A_3 \sigma_s^2 (i+1)\rho^{i+1}\mu^2 + \frac{A_2 \sigma_s^4\mu^2}{\nu^2}} \nnb
			\le&\  C_3 \rho^{(i+1)/2}  + C_4\sigma_s\sqrt{i+1}\rho^{(i+1)/2}\mu+C_5\frac{\sigma_s^2\mu}{\nu} \nnb
			\le&\ C_3 \rho_2^{(i+1)/2}  + C_4\sigma_s\sqrt{i+1}\rho_2^{(i+1)/2}\mu+C_5\frac{\sigma_s^2\mu}{\nu}.
		}
		Substituting \eqref{lm: x4-evolve-3-92} into \eqref{lkn6789}, we reach
		\eq{\label{asln876}
		&\hspace{-5mm} \bE\|\hw_i - \tx_i\|^2 \nnb
		\le&\ \rho_1 \bE \| \hw_{i-1} - \tx_{i-1}\|^2 + (r_5 \rho_1^{i} \mu^3 + r_6 \mu^3) + r_7 \sqrt{i}\rho_2^i\mu^{3/2} + r_8 i\rho_2^i \mu^{5/2} +  r_9 \sqrt{i}\rho_2^{i/2}\mu^{5/2} \nnb
		&\ + r_{10}\mu^2\rho_2^{i/2} + r_{11}\sqrt{i}\rho_2^{i/2}\mu^3 + r_{12} \mu^3,
		}
		where the constants are defined as
		\eq{\label{97623hla}
		& r_7\define 2C_3 r_3,\quad r_8 \define 2C_4r_3\sigma_s, \quad r_9 \define \frac{2C_5r_3\sigma_s^2}{\nu}\nnb
		& r_{10}\define C_3r_6,\quad r_{11}\define C_4r_6\sigma_s,\quad r_{12}\define \frac{C_5r_6\sigma_s^2}{\nu}.
		}
		Now we denote 
		\eq{\label{29asdn}
		\tau_1 \define \rho_1^{1/2},\quad \tau_2 \define \rho_2^{1/2}. 
		}
		Clearly, we have
		\eq{\label{cxhn23678}
		\rho_1 < \tau_1,\quad \rho_2 < \tau_2, \quad \tau_1 < \tau_2.
		}
		With the above relation, expressions \eqref{asln876} becomes 
				\eq{\label{0062378}
					&\hspace{-5mm} \bE\|\hw_i - \tx_i\|^2 \nnb
					\le&\ \tau_2 \bE \| \hw_{i-1} - \tx_{i-1}\|^2 + (r_5 \tau_2^{i} \mu^3 + r_6 \mu^3) + r_7 \sqrt{i} \tau_2^i\mu^{3/2} + r_8 i\tau_2^i \mu^{5/2} +  r_9 \sqrt{i} \tau_2^{i}\mu^{5/2} \nnb
					&\ + r_{10}\mu^2\tau_2^{i} + r_{11} \sqrt{i} \tau_2^{i}\mu^3 + r_{12} \mu^3 \nnb
					\le&\ \tau_2 \bE \| \hw_{i-1} - \tx_{i-1}\|^2 + 2r_{10}\tau_2^i\mu^2 + 2r_7\sqrt{i}\tau_2^i\mu^{3/2} + r_8 i \tau_2^i\mu^{5/2} + (r_6 + r_{12})\mu^3 \nnb
					\le&\ \tau_2^{i+1} \bE \| \hw_{-1} - \tx_{-1}\|^2 + 2r_{10}(i+1)\tau_2^i \mu^2 + 2r_7\tau_2^i\mu^{3/2}\left(\sum_{k=0}^{i}\sqrt{i-k}\right) \nnb
					&\ +r_8\tau_2^i\mu^{5/2}i(i+1)  + \frac{(r_6+r_{12})\mu^3}{1-\tau_2}.
					%		\le&\  \rho_1^{i+1}  \bE\|\hw_{-1} - \tx_{-1}\|^2 + r_5(i+1)\rho_1^i\mu^3 + r_7\rho_2^i\mu^{3/2}\sum_{k=0}^{i}\sqrt{i-k}\left(\frac{\rho_1}{\rho_2}\right)^k \nnb
					%		&\ + r_8\mu^{5/2}\rho_2^i\sum_{k=0}^{i}(i-k)\left(\frac{\rho_1}{\rho_2}\right)^k + r_9\tau_2^i\mu^{5/2}\sum_{k=0}^{i}\sqrt{i-k}\left(\frac{\rho_1}{\tau_2}\right)^k \nnb
					%		&\ + r_{10}\mu^2 \sum_{k=0}^{i}\left(\frac{\rho_1}{\tau_2}\right)^k + r_{11}\tau_2^i\mu^3\sum_{k=0}^{i}\sqrt{i-k}\left(\frac{\rho_1}{\tau_2}\right) + \frac{(r_6 + r_{12})\mu^3}{1-\rho_1}
				}
		Note that $\tau_2 = \sqrt{\rho_2} = \sqrt{1-\mu\nu/4}$. When $\mu$ is sufficiently small, we have $\tau_2 = 1 - \mu\nu/8$ and hence $1-\tau_2 = \mu\nu/2$. With this fact and recall that $\bE \| \hw_{-1}-\tx_{-1}\|^2=C_6 \mu^2$, finally we can show that
		\eq{
			\label{110-1}
			\bE\|\hw_i - \tx_i\|^2 \le&\ C_6 \tau_2^{i+1} \mu^2 + 2r_{10}(i+1)\tau_2^i \mu^2 + 2r_7\tau_2^i\mu^{3/2}\left(\sum_{k=0}^{i}\sqrt{i-k}\right) \nnb
			&\ +r_8\tau_2^i\mu^{5/2}i(i+1)  + \frac{8(r_6+r_{12})\mu^2}{\nu}.
		}
		Substituting the definitions of all constants, we get 
		\eq{
			\label{110-1-23hasd}
			&\hspace{-5mm} \bE\|\hw_i - \tx_i\|^2 \nnb
			\le&\ C_7 \left( \frac{\delta^2 s_1(i) \tau_2^i \mu^{3/2}}{\nu^{5/2}} +\tau_2^{i+1}\mu^2+ \frac{\sigma_s^2\delta^2(i+1)\tau_2^i\mu^2}{(1-\beta)\nu} + \frac{\sigma_s\delta^2 s_2(i) \tau_2^i\mu^{5/2}}{\nu^{5/2}} + \frac{\delta^2 \sigma_s^4 \mu^2}{(1-\beta)\nu^2}\right) \nnb
			\le&\ C_8 \left( \frac{\delta^2\sigma_s^2 s_2(i) \tau_2^{i+1} \mu^{3/2} }{(1-\beta)\nu^{5/2}} +
			\frac{\delta^2 \sigma_s^4 \mu^2}{(1-\beta)\nu^2} \right). 
		}
		where 
		\eq{\label{sum}
		s_1(i) \define \sum_{k=0}^{i}\sqrt{i-k}, \quad s_2(i)\define i(i+1)
		}
		Furthermore, it holds that
		\eq{
			\label{2nd bound-20-2l;an}
			\bE\|\tw_i-\tx_i\|^2 \le 2 \bE\| \hw_i-\tx_i \|^2 + 2\beta^2 \bE\|\cw_{i}\|^2.
		}
		Using the upper bound for $\bE\|\cw_i\|^2$ in \eqref{237nam}, we then have
		\eq{
			\label{110-1-23hasd-asdnkj}
			\bE\|\hw_i - \tx_i\|^2 = O \left( \frac{\delta^2\sigma_s^2 s_2(i) \tau_2^{i+1} \mu^{3/2} }{(1-\beta)\nu^{5/2}} + \frac{(\delta^2 + \gamma^2)\rho_1^{i+1} \mu^2}{(1-\beta)^2} +
			\frac{\delta^2 \sigma_s^4 \mu^2}{(1-\beta)\nu^2} \right).
		}

	\bibliography{icassp_ref,sto_pgextra_jmlr}

\end{document}